\documentclass[11pt]{article}

\usepackage{mathrsfs}
\usepackage{amsfonts}
\usepackage{latexsym,bm,amsfonts,amssymb,pifont}
\usepackage{amsmath}
\usepackage{color}

\usepackage{enumerate}

\newcommand{\ngs}{ G^N(\mathbb{R}^m)   }
\newcommand{\cm}{C^{1\text{-var}} ([0, 1];\, \mathbb{R}^m)}

\newtheorem{theorem}{Theorem}[section]
\newtheorem{Def}[theorem]{Definition}
\newtheorem{thm}[theorem]{Theorem}
\newtheorem{prop}[theorem]{Proposition}
\newtheorem{cor}[theorem]{Corollary}

\newtheorem{lemma}[theorem]{Lemma}
\newtheorem{remark}[theorem]{Remark}
\newtheorem{example}[theorem]{Example}
\numberwithin{equation}{section}
\newenvironment{proof}[1][Proof]{\textbf{#1.} }{ \hfill $\Box$}

\def\tvr{\text{-var}}

\def\al{\alpha}

 \def\mE{\mathbb{E}}

\def\wt{\widetilde}

\def \eref#1{\hbox{(\ref{#1})}}

\def\RR{\mathbb{R}}

\def\NN{\mathbb{N}}
\def\EE{\mathbb{E}}

\def\wt{\widetilde}

\def\bfx{{\bf x}}

\def\bfz{{\bf z}}

\def\be{{\beta}}

\def\al{{\alpha}}

\def\be{{\beta}}
\def\Ga{{\Gamma}}

\def \eref#1{\hbox{(\ref{#1})}}

\def\th{{\theta}}

\begin{document}

 \title{Taylor schemes  for rough differential equations and   fractional diffusions}
 \date{}\;
  \author{  {\sc Yaozhong Hu}\thanks{Y.  Hu is
partially supported by a grant from the Simons Foundation
\#209206.},   \;\;  {\sc Yanghui Liu}, \;\;  {\sc David Nualart}\thanks{D. Nualart is supported by the NSF grant  DMS 1512891 and the ARO grant FED0070445. \newline
 }  \\ Department of Mathematics \\
	The University of Kansas \\
Lawrence, Kansas, 66045 USA}

 \maketitle
 \begin{abstract}
In this paper, we   study two   variations  of the time discrete Taylor   schemes   for        rough differential equations and for   stochastic differential equations driven by fractional Brownian motions. One is the incomplete Taylor scheme  which
  excludes some   terms  of an Taylor scheme in its recursive computation  so as to reduce the computation time.
The other one
is to add some deterministic terms to an incomplete Taylor scheme to improve the mean rate of convergence.
Almost sure rate of   convergence and   $L_p$-rate of convergence are obtained for the incomplete
Taylor schemes.  Almost sure rate is expressed in terms of the H\"older exponents of the driving signals
and the $L_p$-rate is expressed by the Hurst parameters.  Both rates involves with
 the incomplete Taylor scheme in a very explicit way and then  provide us with the best incomplete schemes,
 depending on that one needs  the almost sure convergence or one needs  $L_p$-convergence.
As in the smooth case,  general
   Taylor schemes are always complicated to deal with. The incomplete Taylor scheme
   is even more sophisticated to analyze. A new feature of our
   approach is  the  explicit expression of  the error functions
   which will be easier to study.
 Estimates for  multiple integrals and  formulas for  the iterated vector fields are
   obtained to analyze the error functions and  then to obtain the rates  of convergence.
     \end{abstract}

\section{Introduction}\label{section1}
Consider  the $d$-dimensional   differential equation
 \begin{equation}\label{e.1.1}
dy_{t}  =  V (y_{t}) dx_{t},   \quad t \in [0,T], \quad y_{0} \in \RR,
\end{equation}
where $  x=(x^{1},\dots, x^{m})  $ is H\"older continuous of order $\beta>1/2$   and
$V = (V^{i}_{j})_{1\leq i\leq d, 1\leq j \leq m}$ is a continuous mapping from $\RR^{d}$ to $\RR^{d   \times  m}$.  It is  well-known
(see  \cite{Ly} and  \cite{NuRa}) that if   $ V $ is continuously differentiable
and its  partial derivatives   are  bounded and locally H\"older continuous of
order  $\delta> \frac{1}{\beta} -1$,   then   equation \eref{e.1.1} has
 a unique solution  that  is H\"older continuous of order $\beta$.

Our goal  in this paper is to study     numerical approximations for  the solution of   equation \eref{e.1.1}.   We    briefly recall  the way to obtain
some general numerical approximation schemes for   equation \eref{e.1.1}.

Assume that $V$ has sufficient regularity.  A simple Taylor expansion
(iterated application of chain rule) leads,   when $t$ is sufficiently
close to $s$,  to the following  approximation
 	\begin{eqnarray} \label{e.taylor-s-t}
y_t &\approx & y_s + \mathcal{E}_{s,t}^{(N)} (y_{s})\,,
\end{eqnarray}
where
\begin{eqnarray*} 
\mathcal{E}_{s,t}^{(N)} (y) &:=&  \sum_{r=1}^N
\sum_{\substack{  ( j_{1},\dots, j_{r})  \in \Gamma_r   } } {\mathcal{V}}_{j_1} \cdots {\mathcal{V}}_{j_{r}} I (y ) \int_s^t \int_s^{u_r} \cdots \int_s^{u_2} dx^{j_1}_{u_1} \cdots dx^{j_r}_{u_r}\,,\quad y\in \RR^d\,.
\end{eqnarray*}
 In this expression,  $\Gamma_r$ is the collection of multi-indices
 of length $r$ with elements in $\{1,2,\dots, m\}$, $I  $ is the identity function
 ($I(y)=y$) from $\RR^d$ to $\RR^d$ ,
 and ${\mathcal{V}}_j$ is the vector field
\begin{eqnarray}\label{eq 1.3i}
{\mathcal{V}}_j f = \sum_{i=1}^dV_j^i \partial_i f, ~~j=1,\dots, m,
\end{eqnarray}
where   $\partial_{i}$ denotes the differential operator $\frac{\partial}{\partial y^{i}} $, $i=1,\dots, d$
(we refer the reader to   \cite{BZ, FZ, hustochastics} for more details;
\cite{hustochastics} gives a Taylor    expansion with explicit form of the residual).

Let $0=t_0<t_1<\cdots<t_n=T$ be  any partition of the interval $[0, T]$.
On the interval $  [t_k, t_{k+1}]$, we  may  use
$y^n_{t_k}+ \mathcal{E}_{ t_k,   t  }^{(N)}  ( y^n_{t_k   }   )$ to approximate $y_{t }$.
We iterate this on each subinterval  of the partition to obtain the following recursive scheme  for \eref{e.1.1},
 	 \begin{eqnarray} \label{eqn scheme1-interval}
	 y^{n}_{t }  &=&
y^n_{t_k}+ \mathcal{E}_{ t_k,   t  }^{(N)}  ( y^n_{t_k   }   ), \quad t \in [t_{k},t_{k+1}],   	 \end{eqnarray}
for   $  k=0,1,\dots, n-1$, with $y^{n}_{0} = y_{0}$,
  In this paper,
 we shall take $t_{k} = \frac{T}{n}k$, $k=0,1\dots, n$.  The recursive  scheme \eref{eqn scheme1-interval}
 can also be written as
	 \begin{eqnarray}\label{eqn scheme1}
	 y^{n}_{t }  &=&
y_0+	\sum_{k=0}^{\lfloor  \frac{nt}{T}   \rfloor } \mathcal{E}_{ t_k , t_{k+1} \wedge t   }^{(N)}  ( y^n_{t_k   }  )\,,
	 \end{eqnarray}
	 where $a \wedge b$ is the smaller of the numbers $a$ and $b$ and $  \lfloor  a   \rfloor $ is the integer part of $a$.
The recursive scheme   \eref{eqn scheme1-interval}  is usually called
 the \emph{time discrete Taylor     scheme} or simply \emph{Taylor scheme}, of order $N$.
Note that the  interpolation on each interval $[t_k, t_{k+1}]$ used  in   \eref{eqn scheme1-interval}
 and \eref{eqn scheme1}  guarantees  that   the numerical scheme  has the same convergence rate  at non-discretization points $t\in [0,T] \setminus D$ as
 at the discretization points $t\in D = \{  \frac kn T, \, k=0,1,\dots, n \}$.
	 	
Taylor scheme \eref{eqn scheme1} has been    considered in \cite{FV}  when
 $x$ is a \emph{weak geometric $p$-rough path} (see Section \ref{section7}),
$p\geq 1$.
It  is proved    that under some additional
regularity assumptions on $V$,  for $N\geq \lfloor p \rfloor$,
the rate of convergence of $y_t^n$ to $y_t$ is
$n^{1-(N+1)/p }$.
Clearly, the larger the   $N$ in \eref{eqn scheme1} is
the higher will be  the convergence rate.
If  $N=1$,  then   \eref{eqn scheme1} is reduced to the  classical Euler scheme
\begin{eqnarray} \label{eq 1.6i}
y^n_{t } = y^{n}_{t_{k}} +   V(y^n_{t_k}) (x_{t } - x_{t_{k}}), \quad  t \in [t_{k},t_{k+1}],  \end{eqnarray}
 for  $ k=0,1,\dots, n-1$, with  $y^n_0=y_0$.
This classical Euler scheme    has  been studied, for instance, in \cite{HLN,Mishura,Neuenkirch}.

\begin{remark} With an abuse of notation we shall use the same notation $y_t^n$ to denote
the approximation obtained by different schemes  when there is no confusion.
\end{remark}

When one of the  driving signals  $x$ is the   time, say $x^1(t)=t$,  and when the others are independent standard
Brownian motions, an important scheme is the so-called Milstein scheme, which  has the following
form
 \begin{eqnarray}\label{e.milstein}
	 y^{n}_{t }  &=&
y^n_{t_k} +	\bar {\mathcal{E}}_{ t_k, t }   ( y^n_{t_k  }  )\,, \quad t \in [t_{k},t_{k+1}], 	 \end{eqnarray}
	  for  $ k=0,1,\dots, n-1$, with  $y^n_0 =y_0$, 
where
\begin{eqnarray*} 
\bar {\mathcal{E}}_{ s,  t } (y) &:=&
\sum_{\substack{  j  \in    \Gamma_1   } } {\mathcal{V}}_{j }   I (y ) \int_s^t   dx^{j}_{u }  +
\sum_{\substack{  ( j ,  j')  \in\bar   \Gamma_2   } } {\mathcal{V}}_{j }   {\mathcal{V}}_{j'} I (y ) \int_s^t     \int_s^{u'} dx^{j }_{u }   dx^{j'}_{u'}
 \end{eqnarray*}
and   $\bar{\Gamma}_2 = \{(j ,  j')\in \Gamma_2: j ,  j'=2,\dots, m\}$. This scheme does not include the   terms    $\int_s^t \int_s^{u }   dx^{j }_{v }   dx^{j'}_{u }$,  where
one or both  of the $j $ and  $j'$ are   $1$.  In the Brownian motion case, it is well-known that the Milstein scheme
has the same rate of convergence as the order $2$ Taylor scheme while it requires fewer computations.

\medskip
This motivates us to ask the following question.
\smallskip

\noindent{\bf Question 1.} 
How to   eliminate as many terms  as possible    in $ \mathcal{E}_{s,t}^{(N)} \left( y \right)$ while keeping the same rate of
convergence?  \
More precisely,   we want to  find
subsets $\wt{\Gamma}_r\subseteq \Ga_r$ so that $\wt{\Gamma}_r$ contains as few  elements
as possible and  when we  replace
$ \mathcal{E}_{t_k  ,t_{k+1} \wedge t }^{(N)} \left( y_{t_k}^n \right)$ in the Taylor scheme \eref{eqn scheme1} by
$\wt{\mathcal{E}}_{t_k ,t_{k+1}  \wedge t  }^{(N)}\left( y_{t_k}^n \right)$, we have  the same rate of convergence
as that of the  original one.  Here
\begin{eqnarray}\label{e1.8i}
\wt{\mathcal{E}}_{s,t}^{(N)} (y) &:=& \sum_{r=1}^N \sum_{ (j_{1},\dots,j_{r}) \in \wt{\Gamma}_r } {\mathcal{V}}_{j_1} \cdots {\mathcal{V}}_{j_{r}} I (y ) \int_s^t \int_s^{u_r} \cdots \int_s^{u_2} dx^{j_1}_{u_1} \cdots dx^{j_r}_{u_r}\,.
\end{eqnarray}
We shall call such new Taylor scheme an {\it incomplete  Taylor scheme}, which has the following form:
\begin{eqnarray}\label{eq1.7}
y^{n}_{t }  &=&
y^n_{t_k}+ \wt{\mathcal{E}}^{(N)}_{ t_k,     t } ( y^n_{t_k   }   ),   \quad t \in [t_{k},t_{k+1}],  
\end{eqnarray}
for  $k=0,1,\dots, n-1$, with $y^n_0=y_0$.

We shall study the rate of convergence of   $y^{n}_{t}$ to $y_t$ for any choice of
$ \wt{\Gamma}_r$ in \eref{e1.8i}. Two types   of convergence will be studied in detail: almost sure
convergence (when the $x^j$  are H\"older continuous   with exponents
$\be_j$)  and the $L_p$-convergence (when the  $x^j$ are fractional Brownian
motions of Hurst parameters $H_j$). The rates  will be different for these two types  of
convergence.
Fix a set 
\begin{equation*} 
 \wt{\Gamma}=\cup_{r=1}^N \wt{\Gamma}_r, \quad 
N=\max\{|\al|\, : \al\in \wt{\Ga}\}\,,
\end{equation*}
where throughout the paper $|\al|$ denotes the length of the multi-index $\al$.
   The almost sure rate $\theta_{\wt{\Ga}} $ can
be expressed in terms of $\be_j$ (see \eref{e.def-theta} below) and the $L_p$-rate
$\rho_{\wt{\Ga}} $  can be expressed in term of
Hurst parameters $H_j$ (see \eref{e.def-theta-lp} below).  These two expressions lead to  the best choices of
$ \wt{\Gamma}_r$ in \eref{e1.8i}, depending on that one needs  the almost sure
convergence ($ \wt{\Gamma}_r$ is given by \eref{e.gamma-theta}
in Section 4) or one needs  $L_p$-convergence ($ \wt{\Gamma}_r$ is given by \eref{e.gamma-rho} in Section 6).

\medskip

To motivate our second problem, let us  recall that when the driving signals are fractional Brownian motions of Hurst
  parameter $H>1/2$,  the classical Euler scheme \eref{eq 1.6i}   has the exact convergence rate $n^{1-2H}$
  (see \cite{HLN,Neuenkirch}).  When we formally equal $H$   to $1/2$ (the standard Brownian motion case),
  we obtain no convergence!
  This demonstrates on one hand, that in dealing with the incomplete  Taylor schemes we may not be able to
  use the same ideas  from the Brownian motion case
  (\cite{Hu,KP}).  This is largely due to the lack of the martingale property of the driving signals.
  We will pay special attention to this fact.  On the other hand,  to improve the Euler scheme  for the fractional Brownian motion case, a modified  Euler scheme
is proposed and investigated in   \cite{HLN}:
 \begin{eqnarray}\label{eq1.8}
y^n_{t } = y^{n}_{t_{k}} +   V(y^n_{t_k}) (x_{t } - x_{t_{k}})  + \frac12 \sum_{j=1}^{m} ( \partial V_{j} V_{j}) (y^{n}_{t_{k} }) (  t-t_{k})^{2H}\,, \quad t \in [t_{k},t_{k+1}],
\end{eqnarray}
 for $k=0, 1, \dots, n-1$, with $y^n_0=y_0$.
 Here we denote $V= (V_{1},\dots, V_{m})$. It has been shown that this modified Euler scheme has a higher rate of convergence
 than the classical Euler scheme. In particular, it is proved in \cite{HLN} that 
 for any $t\in [0,T]$,
\begin{eqnarray}\label{e1.10i}
\mE(|y_{t} - y^{n}_{t}|^{p} )^{1/p} &\leq& \begin{cases}
      K  n^{  1/2-2H     }
  &\qquad \hbox{if}\quad    { \frac12< H < \frac34 }\, ,\\
    K  n^{  -1     } \sqrt{\log n}
  &\qquad \hbox{if}\quad    {  H = \frac34 }\, ,\\
 K  n^{         -1        }
   &\qquad \hbox{if}\quad      { \frac34  < H   < 1   }\, ,\\
\end{cases}
\end{eqnarray}
under   proper regularity assumptions on $V$,
where $K$ is a constant independent of $n$.

The scheme  \eref{eq1.8}  is obtained by adding to the classical Euler scheme \eref{eq 1.6i}
a  deterministic term (note that for simplicity, we   assume here that $x$ is a standard $m$-dimensional fractional Brownian motion).
The inclusion in \eref{eq1.8} of the deterministic terms $\frac12 \sum_{j=1}^{m} ( \partial V_{j} V_{j}) (y^{n}_{t_{k} }) (t-t_{k})^{2H}$, as opposed to double integral terms as in \eref{e.milstein}, helps
to save  computation time due to the  evaluation of  double stochastic integrals.
 It is then natural to ask the following question:
 \smallskip

\noindent{\bf Question 2.} 
Can we add   some deterministic terms to  the incomplete  Taylor scheme \eref{eq1.7}  so as to increase
the rate of convergence?

    We shall answer this  question in the case when the  $x^j$'s are fractional Brownian motions
    or $x^j_t=t$  by  introducing the following  {\it modified   Taylor  scheme}:
    \begin{equation}\label{eq1.7-modify}
y^{n}_{t }  =
y^n_{t_k}+ \wt{\mathcal{E}}_{t_k, t }^{(N)} (y^n_{t_k} )+   
\sum_{\substack{  (j_{1},\dots, j_{r})  \in \Gamma ' } }
{\mathcal{V}}_{j_1} \cdots {\mathcal{V}}_{j_{r}} I (y^n_{t_k}  )D_{j_{1},\dots, j_{r}}(t-t_{k})  ,
 \end{equation}
for $t\in [t_{k},t_{k+1}]$,   $   k=0,1,\dots, n-1 
$, $y^n_0=y_0$.
The above set $ \Gamma ' $ is a finite subset of $   \Gamma \setminus   \wt{\Gamma} $, where $\Gamma=\bigcup_{r=1} ^\infty \Gamma_r$ and $\wt\Gamma=\bigcup_{r=1}^N \wt\Gamma_r$,  that  will be given
explicitly in  Section \ref{section 7}. The explicit form of  $D_{j_{1},\dots, j_{r}}(t)$,
  $t\in [0,T]$,  
is given in   Remark \ref{remark5.15}.

The main tasks of this paper are  to  establish  the almost sure and the $L_p$-rate  of convergence results  for  the incomplete  Taylor  scheme \eref{eq1.7}
and the modified Taylor scheme \eref{eq1.7-modify}. It is worthy  to emphasize that
the modified Taylor scheme \eref{eq1.7-modify}  has a higher $L_p$-rate of convergence than    the incomplete  Taylor scheme  \eref{eq1.7} (compare Theorem \ref{thm7.2} with Theorem \ref{thm6.3}).
 We also point out  that our result extends that of \cite{HLN}: in the simplest case $N=1$,
our result  recovers    the upper bound estimate \eref{e1.10i} (see Example \ref{ex6.10}).

\medskip

The remainder of the Taylor expansion \eref{e.taylor-s-t} has an involved expression  (see \cite{BZ}).
If we throw some terms away, then
the remainder is even  more complicated. In the study of the convergence rate for the schemes \eref{eq1.7}
and  \eref{eq1.7-modify}
it is necessary to
investigate   this type of remainders.  We shall express the error 
in the following form:   %
%
  \begin{eqnarray}\label{eq 1.8}
y_t-y^{n}_{t} = \Phi_t^{-1} \int_0^t \Phi_s d  R_s  ,
  \end{eqnarray}
where $\Phi  \in \RR^{d\times d}$ is the   solution of a linear   differential equation, $\Phi^{-1}$ is its inverse, that is, $\Phi \Phi^{-1}\equiv I$, and $R_t$ is the remainder term, whose upper bound usually provides the desired convergence rate.
  The study  of \eref{eq 1.8} is based on the algebraic properties of   equation \eref{e.1.1}, which are interesting in its own right (see Section \ref{section shuffle}-\ref{subsection 4.3}).   It is well know that for $i_{1}, \dots, i_{r}$, $j_{1}, \dots, j_{r'} =1,\dots , m$, the product   $ \int_s^t   \cdots \int_s^{u_2} dx^{i_{1}}_{u_1}   \cdots   dx^{i_{r}}_{u_r}     $ and  $ \int_s^t   \cdots \int_s^{u_2} dx^{j_{1}}_{u_1}   \cdots   dx^{j_{r'}}_{u_{r'}}    $ is equal to the summation of integrals of the form   $ \int_s^t   \cdots \int_s^{u_2} dx^{l_{1}}_{u_1}   \cdots   dx^{l_{r+r'}}_{u_{r+r'}}    $, where the summation  runs over the multi-indices
  $ (l_{1},\dots, l_{r+r'} ) $  obtained by shuffling the two multi-indices   $(i_{1}, \dots, i_{r})$ and $(j_{1}, \dots, j_{r'} ) $. The study of the error function $R_t$ needs an expansion of the multiple integral
  $\int^{s }_{ \tau }   \int^{s_p }_{\tau}  \cdots   \int^{s_{2 } }_{\tau}d g^{\gamma^1}_{\tau,  s_1}  \cdots dg^{\gamma^{p-1}}_{\tau,  s_{p-1}}  dg^{\gamma^p }_{\tau,  s_p}$, where each $g^{\gamma }_{\tau,  s }$ is itself
  a multiple integral. This expansion of   multiple integral of the multiple integrals can also be done by the
  shuffle-type permutations.
   A key ingredient  in our proof is to establish a relation between
    these shuffle-type permutations with  the permutations
  when we expand the iterated vector fields ${\mathcal{V}}_{j_1} \cdots {\mathcal{V}}_{j_{r}} I (y )$  through  a generalized Leibniz rule (see Propositions \ref{lem 2.1}, \ref{prop 4} and \ref{lem 2.5}).

To obtain the rate of convergence for the modified Taylor scheme \eref{eq1.7-modify}   we need some
 subtle $L_{2}$-estimates of a multiple Riemann-Stieltjes integral
 \begin{eqnarray*}
J_{r}(\mathcal{A}) &:= & \int_{0}^{T} \cdots \int_{0}^{T}  \mathbf{1}_{\mathcal{A}}  d B^{1}_{s_{1}} \cdots dB^{r}_{s_{r}},
\end{eqnarray*}
 and its centralization
 \begin{eqnarray*}
\wt{J}_{r} (\mathcal{A})&:=& J_{r} (\mathcal{A}) -  \mE [ J_{r} (\mathcal{A})  ],
\end{eqnarray*}
for $\mathcal{A} = \bigcup_{k=0}^{n-1} \{(s_{1},\dots, s_{r}): t_{k} \leq s_{1}<\cdots<s_{r}\leq t_{k+1} \}$,
where $B^{j}$, $j=1,\dots, r$ is either a fractional Brownian motion with Hurst parameter larger than $1/2$ or the identity function. Note that  $J_{r} (\mathcal{A})$ is well defined  as an integrated Riemann-Stieltjes integral. 
  The $L_{2}$-estimates are made possible by a monotonicity property of the multiple  integral obtained in Section \ref{section5.3}; that is, for $\mathcal{A}' = \bigcup_{k=0}^{n-1} [t_{k},t_{k+1}]^{r}$, the $L_{2}$-norms of $J_{r}(\mathcal{A})$ and $\wt{J}_{r}(\mathcal{A})$ are less than those of $J_{r}(\mathcal{A}')$ and $\wt{J}_{r}(\mathcal{A}')$, respectively (see Sections \ref{section5.3} and \ref{section5.4}).

The paper is organized as follows.
In  Section \ref{section bijection},  we introduce some shuffle-type permutations  and then apply them  to   expand
    the multiple integral  of the multiple integrals and we also  derive a generalized   Leibniz rule
    for iterated vector fields.  
 With these preparations we  derive, in Section \ref{section 5},   an explicit expression for the error function for  the scheme     \eref{eq1.7}.  In Section \ref{section 6}, we  obtain
   the almost sure convergence rate for  the     scheme \eref{eq1.7}.  In Section \ref{section5}, we prove some $L_{p}$-estimate results.
   These   estimates are  applied to obtain
   the $L_p$-convergence rate for  the   incomplete   scheme \eref{eq1.7}  in subsection \ref{section6}
   and the $L_p$-convergence rate for  the  modified   Taylor scheme \eref{eq1.7-modify} in subsection \ref{section 7}.  In Section \ref{section7}, we generalize the results in Section \ref{section 5} to the rough paths case.
  In the appendix, we provide   some necessary estimates of some multiple integrals and the solution of some differential equations.
  
  Along the paper we denote by $C$ a generic constant, that may be different form line to line, and which might  depend on $T$ and the vector fields $V_j^i$.  

  \section{Multiple  integral of multiple integrals  and   generalized  Leibniz rule  }\label{section bijection}
The primary aim of this section is to prove an identity on multiple integral of multiple integrals   (see Proposition \ref{prop 4}) and a generalized Leibniz rule (see Proposition \ref{lem 2.5}). To do so, we   need to introduce some shuffle-type permutations and their inverses (see Section \ref{section shuffle}).

\subsection{Shuffle-type  permutations and their inverses}\label{section shuffle}

  Let   $\alpha = (\al_{1}, \dots, \al_{r}) \in \Gamma_r$, where $\Gamma_r$ is the collection of    multi-indices of length $r$  with elements in $ \{1,\dots, m\} $. Take   $ \vec{l}= (l_1,\dots, l_p) $ such that $ 1\leq l_1<\cdots<l_p = r$,   $ p\leq r$.
 Assume that $  f^i_j  \in C^{r-p} (\RR)$, $i=1,\dots, p$, $j=1,\dots, m$.
 As a motivation,  we first   consider the following expression\,:
  \begin{eqnarray}\label{eq 4.1}
   {\mathcal{V}}_{\alpha_{0, l_1  } }
       \left( f^{1 }_{\alpha_{l_1}  } \cdots
             {\mathcal{V}}_{\alpha_{ l_{p-2}, l_{p-1}} } \left(  f^{ p-1 }_{\alpha_{ l_{p-1} } }   {\mathcal{V}}_{\alpha_{ l_{p-1}, l_p } }  f^{ p}_{\alpha_{l_p }  }  \right)
  \right).
  \end{eqnarray}
  Here we denote $\al_{i,j}:=(\al_{i+1},\dots, \al_{j-1} )$,      ${\mathcal{V}}_{j_1,\dots, j_k} := {\mathcal{V}}_{j_1}\cdots {\mathcal{V}}_{j_k}$, and
   recall that  ${\mathcal{V}}_j$ is the differential operator defined in \eref{eq 1.3i}. Note that the subindex of $\al$ in each element     in \eref{eq 4.1}, either an operator  or a function, identifies the location of this element in \eref{eq 4.1}. For example, $\mathcal{V}_{\al_{j}}$ is the $j$th element  and $f^{i}_{\al_{l_{i}}}$ is the $l_{i}$th element.

    It is easy to verify that $\mathcal{V}_{j}$ satisfies the product rule, that is, $\mathcal{V}_{j} (fg) = g \mathcal{V}_{j}f+ f \mathcal{V}_{j}g$ for $f,g \in C^{1}$.
By applying the product rule to     \eref{eq 4.1},  the operators ${\mathcal{V}}_{\al_j }$, $j \in \{1,\dots, r\} \setminus \{l_1,\dots, l_p\}$ act on   functions $f^i_{\al_{ l_i} }$, $i=1,\dots, p$, in such a way   that:

\noindent(i)
for  $j$ such that $l_{i-1}<j<l_{i}$, the operator ${\mathcal{V}}_{\al_j }$ act on one of the functions $f^i_{\al_{i}}$,  $\dots$,  $f^p_{\al_{p}}$\,;

\noindent  (ii) if two operators ${\mathcal{V}}_i$ and ${\mathcal{V}}_j$ act on the same function $f^k_{\al_{l_k} }$ then their order in \eref{eq 4.1} is kept.

\smallskip

\noindent Note that we take   $l_0 =0$ in (i). The quantity \eref{eq 4.1}
  is then expanded into the summation of quantities of the following form,
 \begin{eqnarray}\label{eq 4.2i}
\left( {\mathcal{V}}_{\al'_{ 0,\tau_1 } } f^1_{\al'_{ \tau_1} } \right)
\cdots
\left( {\mathcal{V}}_{\al'_{\tau_{p-1},\tau_p } } f^p_{\al'_{ \tau_p } } \right),
\end{eqnarray}
 where $\al'$ is some permutation of $\al$ such that $\al_{\tau_{i}}' = \al_{l_{i}}$, $i=1,\dots, p$, and $  (\tau_1,\dots, \tau_p)   \, $ are constants such that $1\leq \tau_1 < \cdots < \tau_p     = r $.
Denote by ${\mu}(i)$ the new location of the $i$th element of \eref{eq 4.1} in \eref{eq 4.2i}, then $\mu$
is
 the permutation of $\{1,2,\dots, r\}$ such that $ \al = \al' \circ {\mu}  $. In particular, we have $\mu(l_{i}) = \tau_{i}$, $i=1,\dots, p$.
Each quantity of the form \eref{eq 4.2i} obtained from applying the product rule to \eref{eq 4.1} is then identified with a permutation $\mu $ on $\{1,\dots, r\}$.  It is easy to see that these permutations satisfy:

\medskip
\noindent
{\bf Rule 1.}  $ {\mu}(l_i) < {\mu}(l_{i+1}) $, $i=1,\dots, p$\,;

\smallskip
\noindent
{\bf Rule 2.} 
   ${\mu}( y )  <  {\mu}(y' ) $
   \emph{if  $y<y'$ and $  {\mu}(y), {\mu} (y') \in I_i $ for some $i$, where $(I_i,~ i=1,\dots, p)$ is the partition of $\{1,\dots, r\}$ defined as follows}
   \begin{eqnarray}\label{e.3.3}
   I_1= \{1, \dots, {\mu}(l_1) \}; ~~I_i = \{{\mu}(l_{i-1}) +1 , \dots, {\mu}(l_i) \}, ~~i=2,\dots, p\,.
   \end{eqnarray}
 
 \medskip
  In fact, Rule 2 is  the ``translation'' of condition (ii) in terms of $\mu$ and Rule 1 is required to fix the ordering of the terms $\left( {\mathcal{V}}_{\al'_{ \tau_{i-1},\tau_{i } } } f^1_{\al'_{ \tau_{i}} } \right)$, $i=1,\dots,p$, in \eref{eq 4.2i}.

   The ``translation'' of condition (i) in terms of $\mu$ is:

\medskip
\noindent
{\bf Rule 3.} 
   ${\mu}( l_{i-1})< {\mu}(y)  $ \emph{if} ~$ l_{i-1}  < y<l_{i} $\,.

\medskip
  This rule is implied by Rule 1 and 2:
  \begin{lemma}\label{lem4.1}
    Assume that $\mu$ is a permutation of $\{1,\dots, r\}$ that satisfies Rule 1 and 2. Then $\mu$ also satisfies  Rule 3.
      \end{lemma}
  \begin{proof}
  We take $y$ such that $ l_{i-1}<y $. Since $\mu$ is a bijection, we have $\mu(l_{i-1}) \neq \mu(y)$.
  Suppose that ${\mu}(l_{i-1})>{\mu}(y) $. Then there exists $j$ such that ${\mu}(l_j) , {\mu}(y) \in I_j $ and $ {\mu}(l_{i-1})\geq {\mu}(l_j)>{\mu}(y)$. By   Rule 2 we have $l_j>y$. On the other hand,   Rule  1 implies that $l_{i-1}\geq l_j$. So we obtain $l_{i-1}>y$, which contradicts   the assumption.
  \end{proof}

  \medskip

 We are ready to define the shuffle-type permutations.
\begin{Def}\label{def2.2}
Take  $ \vec{l}= (l_1,\dots, l_p) $ such that $ 1\leq l_1<\cdots<l_p = r$.
We define $\Theta_r(\vec{l})$ as the collection of all permutations  of $\{1,\dots, r\}$ that satisfy Rule 1 and 2.
Take  $\vec{\tau} = (\tau_1,\dots, \tau_p)   \, $     such that $1\leq \tau_1 < \cdots < \tau_p     = r $. We define  $\Theta_r(\vec{l}; \vec{\tau})$ as the collection of permutations in $\Theta_r(\vec{l})$ such that ${\mu}(l_i) =\tau_i$, $i=1,\dots, p$. \end{Def}
Note   that  the set $\Theta_r (\vec{l} \,; \vec{\tau}  ) $     could be   empty   for some   $l_1,
\dots, l_p$ and $\tau_1, \dots, \tau_p$\,.
	
	\medskip
	
  According to the above discussion,    we have the following result:
    \begin{lemma}\label{lem2.3}
Take   $\al=(\al_{1},\dots, \al_{r}) \in \Gamma_r$, $\vec{l} = (l_{1},\dots, l_{p})$ such that $ 1\leq l_{1} <\cdots<l_{p}=r$ and $f^{i}_{j} \in C^{r-p}$, $i=1,\dots, p$, $j=1,\dots, m$. Then the quantity \eref{eq 4.1} is equal to
  \begin{eqnarray*}
     \sum_{
 {    {\mu} \in \Theta_r ( \vec{l} \, )
   }
    }
\left(    {\mathcal{V}}_{ \al \circ {\mu}^{-1} ( 0 , {\mu}(l_1 ) ) } f^1_{ \al_{  l_1 }  } \right)
    \dots
  \left(     {\mathcal{V}}_{ \al \circ {\mu}^{-1}({\mu}(l_{p-1 } ) , {\mu}(l_p) ) } f^p_{ \al_{  l_p  }  } \right)
      \end{eqnarray*}
     or
      \begin{eqnarray*}
\sum_{
      1\leq \tau_1 < \cdots < \tau_p   = r
    }
      \sum_{
    {\mu} \in \Theta_r ( \vec{l}  ; \, \vec{\tau}\,)
     }
      \left(    {\mathcal{V}}_{\al \circ {\mu}^{-1} ( 0,  \tau_1 ) } f^{1}_{\al \circ {\mu}^{-1} (   \tau_1 )  } \right)
    \dots
  \left(     {\mathcal{V}}_{\al \circ {\mu}^{-1}( \tau_{p -1} , \tau_p ) } f^{p}_{\al \circ {\mu}^{-1} (    \tau_p  )  } \right).
\end{eqnarray*}
\end{lemma}
The proof of the lemma is omitted. Please note that in the above summation,  when $ \Theta_r ( \vec{l}  ; \, \vec{\tau}\,)=\emptyset $, we follow the convention that a summation over the empty set is  $0$.

\medskip

We introduce   another type of  permutations, which will   be the inverses of those in $\Theta_{r}(\vec{l})$.  Let $\tau_0=0$ and $\vec{\tau} = (\tau_{1},\dots, \tau_{p})$ be the same as in Definition \ref{def2.2}, and    define the partition on $\{1,\dots, r\}$ given by
\begin{eqnarray}\label{e.3.2}
  I_i =
\{\tau_{i -1} +1 ,\dots, \tau_{i } \} \,, ~~i= 1,\dots, p \,.
\end{eqnarray}
 \begin{Def}\label{def2.4}
 Let $\vec{l}$ and $\vec{\tau}$ be   as in Definition \ref{def2.2}.
  We define
$ \Xi_r ( \vec{\tau}  ) $ as
the collection of   permutations $\rho $ on $
    \{
1,2,\dots,r
\}  $ such that $\rho$ keeps the ordering of $\tau_1$, $\dots$, $\tau_p$\,, i.e. the last elements   of $I_1,\dots, I_p$\,,
  and   the order  of  the elements in each $I_i$.
In other words, $\rho \in  \Xi_r ( \vec{\tau}  ) $ iff $\rho$ satisfies

 \medskip
 \noindent
 {\bf Rule 4.}
    $\rho(\tau_i ) < \rho(\tau_j  ) $ if \,$i <j$;
    
    \smallskip
    \noindent
    {\bf Rule 5.}
    $\rho ( y )  < \rho ( y')$  if $y, y' \in I_i $ and $y < y'$\,.

\medskip
We define
$ \Xi_r (\vec{l} \,; \vec{\tau}  )
$ as
the collection of permutations $ \rho $ in $ \Xi_r ( \vec{\tau}  ) $ such that $  \rho ( \tau_i )= l_i\,, $ $i=1, \dots, p$.
\end{Def}
  Note that for $\rho \in \Xi_r ( \vec{\tau}  )  $     we always have $ \rho(\tau_p) = \rho (r)= r$.

\medskip

The following proposition
 shows that the permutations introduced in Definitions \ref{def2.2} and \ref{def2.4}   are inverses of each other.

\begin{prop}\label{lem 2.1}
Let $\vec{l} $ and $\vec{\tau} $ be   as in Definition \ref{def2.2}.  Suppose that   at least one of the two sets
   $
   \Xi_r (\vec{l} \,;\vec{\tau} )$ and  $\Theta_r (\vec{l} \,;\vec{\tau}  )
$
   is not empty. Then the following   holds,
\begin{eqnarray}\label{e 2.8}
 \Xi_r (\vec{l} \,;\vec{\tau}  ) = \{   \rho: \rho^{-1}
\in  \Theta_r ( \vec{l} \,;\vec{\tau}  )   \}.
\end{eqnarray}
\end{prop}
\begin{remark}
It follows from Proposition \ref{lem 2.1} that $\Xi_r (\vec{l} \,;\vec{\tau}  ) =\emptyset$ if and only if $\Theta_r (\vec{l} \,;\vec{\tau}  ) =\emptyset$.
\end{remark}

\begin{remark}
Equation \eref{e 2.8}   is equivalent to the following,
\begin{eqnarray*}
 \Theta_r (\vec{l} \,;\vec{\tau}  ) = \{ {\mu}: {\mu}^{-1}  \in  \Xi_r (\vec{l} \,;\vec{\tau}  )   \}.
\end{eqnarray*}
\end{remark}
\noindent \textit{Proof of Proposition \ref{lem 2.1}:}\quad
We first note that  the partitions $(I_i,\, i=1,\dots, p)$ defined in \eref{e.3.2}   and in \eref{e.3.3} are the same.
Take $\rho \in  \Xi_r (\vec{l}; \vec{\tau} ) $ and denote ${\mu}:= \rho^{-1}$. We show that ${\mu} \in \Theta_r(\vec{l}; \vec{\tau})$.
  It is clear that ${\mu}$ satisfies Rule 1. Take $y$, $y'$ such that $y<y'$ and $ {\mu} (y)$, $ {\mu} (y') \in I_i $\,.  We have
$$     \rho ( {\mu} (y) )=  y < y'  =  \rho ( {\mu} (y' ) )     . $$
So  Rule 5 in the definition of $\Xi_r( \vec{\tau})$ implies that  $ {\mu}(y) <  {\mu}(y') $. This shows that ${\mu}$ satisfies Rule 2. We conclude that ${\mu}$ belongs to the right-hand side of \eref{e 2.8}.
We take $\rho$ such that $\rho^{-1}=: {\mu} \in  \Theta_r (\vec{l} \,;  \vec{\tau} )$.
Since
 $$
 \rho(\tau_i) ={\mu}^{-1}(\tau_i) = l_i <l_j ={\mu}^{-1}(\tau_j )= \rho(\tau_j)
 $$
  for $i<j$, $\rho$ satisfies Rule 4.
  Take $y,y' \in I_i$ such that $y < y'$. From Rule 2, it is easy to see that
   $$
   {\mu}^{-1}(y) <{\mu}^{-1} (y' ) ,
   $$
   that is, $\rho(y) < \rho (y')$.
  So $\rho$ satisfies Rule 5. We conclude that $\rho \in \Xi_r(\vec{l}; \vec{\tau})$.
\hfill $\square$

\medskip

 \subsection{ Multiple integrals}\label{subsection 4.2}
 Let $g = (g^{1}, \dots, g^{m})$ be a H\"older continuous function  on $[0,T ]$ of order $\beta>1/2$ with values in $\RR^m$. Take $\al = (\al_{1},\dots, \al_{r}) \in \Gamma_r$. Recall that we denote by
 $\Gamma_r$ the collection of multi-indices of length $r$  with elements in $\{1,\dots, m\}$. We also denote
  $\Gamma=\cup_{r=1}^\infty
 \Gamma_r$ the collection of all multi-indices  with elements in $\{1,\dots, m\}$.  Recall that $|\gamma|$  is the length of the multi-index $\gamma$.
 Given a permutation $\rho$ on $\{1,\dots, r\}$, denote $\al \circ \rho = (\al_{\rho(1)},\dots, \al_{\rho(r)}) \in \Gamma_r$.

 In this subsection,  we  study        multiple  integrals, defined as iterated Riemann-Stieltjes integrals, of the form
\begin{eqnarray} \label{e2.6}
  g_{ \tau,  s  }^{   \alpha} &:=& \int_{ \tau }^{s } \int_{  \tau  }^{s_r}       \cdots  \int_{  \tau  }^{s_{2}} dg^{ \al_1  }_{s_{1 }} \cdots dg^{  \al_{ r-1}  }_{s_{r-1} } dg^{ \al_r }_{s_r },
\end{eqnarray}
where $0\le  \tau \le   s  \le T $.
We define the differential of $g_{   \tau, s  }^{ \alpha} $ by
\begin{equation}  \label{eq2}
d g_{   \tau, s  }^{  \alpha}  =  \left( \int_{  \tau  }^{s }       \cdots  \int_{  \tau  }^{s_{2}} dg^{ \al_{1} }_{s_{1 }} \cdots dg^{  \al_{ r-1} }_{s_{r-1} } \right)
dg^{ \al_r  }_{s  }
\, .
 \end{equation}

The following lemma gives a formula for the product of two such  multiple integrals.
\begin{lemma}\label{lem 0}
Let $\gamma' $, $\gamma'' $ be multi-indices in $\Gamma$ and denote $r'=|\gamma'|$, $r''=|\gamma''|$ and $r=|\gamma' | + |\gamma''  |$. Denote $\gamma= (\gamma', \gamma'' ) \in \Gamma$.  Then
\begin{eqnarray*}
g^{\gamma' }_{\tau, s} g^{\gamma'' }_{\tau, s} = \sum_{\rho \in Sh(\gamma', \gamma'' ) } g^{\gamma \circ \rho^{-1} }_{\tau, s},
\end{eqnarray*}
 where $Sh(\gamma', \gamma'')$ is the collection of permutations $\rho$ on $\{1,\dots, r\}$ such that $\rho$ does not change the orderings of $(1,\dots, r' )$ and the orderings of $( r' +1,\dots, r)$, that is, if $y,y' \in \{1,\dots , r'  \}$  or  $y,y' \in \{ r' +1, \dots, r \}$, and $y<y'   $, then we have
$  \rho(y)<\rho(y')$.
\end{lemma}
This result can be shown by the properties of shuffle product of   words (see, for example \cite{Reu}) and   Fubini's theorem.

 The following  is an immediate  corollary  of the above result.
 \begin{lemma}\label{lem 1}
Let $\gamma' $, $\gamma'' $, $r' $, $r$ and $\gamma $ be   as in Lemma \ref{lem 0}.  Then
\begin{eqnarray}\label{eqn 4}
 \int^{s}_{\tau} \int^{s''}_{\tau} d g^{  \gamma'  }_{ {\tau,  } s' }  dg^{ \gamma'' }_{ {\tau,  }s'' }
 =\int^{s}_{\tau}   g^{  \gamma'  }_{ {\tau,  } s'' }  dg^{ \gamma'' }_{ {\tau,  }s'' }
   &=&
  \sum_{ \rho \in  \Xi_{ r   } (  r', r ) } g^{ \gamma \circ {\rho^{-1}}   }_{  {\tau,  }s}\,.
\end{eqnarray}
 \end{lemma}

\begin{proof}
Using  (\ref{eq2}), we can rewrite the left-hand side of \eref{eqn 4} as
$$
\int_{\tau}^s g^{\gamma' }_{{\tau,  }s'' } g^{\gamma''-}_{{\tau,  }s''} dg^{\gamma ({r })  }_{s'' },
 $$
 where   we denote by  $\gamma -$  the multi-index obtained by removing the last element of $\gamma$, that is, $ \gamma  - = (\gamma_1, \dots, \gamma_{r-1} )$, and recall that   $\gamma(i) =\gamma_i$ denotes  the $i$th element of $\gamma$.
 Applying Lemma \ref{lem 0} to $ g^{\gamma'}_{{\tau,  }s'' } g^{\gamma'' -}_{{\tau,  }s'' }   $ yields
 \begin{eqnarray}\label{e2.10ii}
 \int^{s}_{\tau} \int^{s''}_{\tau} d g^{  \gamma'  }_{ \tau,   s' }  dg^{ \gamma'' }_{ {\tau,  }s'' }
   &=& \int_{\tau}^s
    \sum_{\rho \in Sh(\gamma', \gamma''- ) } g^{(\gamma-) \circ \rho^{-1} }_{\tau,  s''}
   dg^{\gamma({r})  }_{s'' }.
\end{eqnarray}
Denote by   $\wt{\Xi}_{r}(r',r)$ the collection of permutations $\rho$ on $\{1,\dots, r\}$ such that there exists $\rho' \in Sh(\gamma',\gamma''-)$ such that
\begin{eqnarray*}
\rho(j)= \begin{cases}
\rho'(j), & j=1,\dots, r-1,
\\
r, & j=r.
\end{cases}
\end{eqnarray*}
Then \eref{e2.10ii} becomes
\begin{eqnarray*}
 \sum_{ \rho \in \wt{ \Xi}_{ r   } (  r', r ) } g^{ \gamma \circ {\rho^{-1}}   }_{ \tau,   s}.
\end{eqnarray*}
 Equation  \eref{eqn 4} then follows by noticing  that $\wt{\Xi}_{r}(r',r) = {\Xi}_{r}(r',r) $.
\hfill
 $\square$

   \medskip

 The following  result is a generalization of Lemma \ref{lem 1}.
 \begin{prop}\label{prop 4}
 Let $\gamma^1$, $\dots$, $\gamma^p $ be multi-indices in $\Gamma$, and denote  $r= |\gamma^1| +\dots+|\gamma^p|$ and $\tau_i = |\gamma^1| +\dots+|\gamma^i| $, $i=1,\dots, p$. Denote $\gamma = (\gamma^1, \dots, \gamma^p) \in \Gamma$.  Then
\begin{eqnarray}\label{e2.10i}
\int^{s }_{ \tau }   \int^{s_p }_{\tau}  \cdots   \int^{s_{2 } }_{\tau}d g^{\gamma^1}_{\tau,  s_1}  \cdots dg^{\gamma^{p-1}}_{\tau,  s_{p-1}}  dg^{\gamma^p }_{\tau,  s_p}
          &=&
           \sum_{ \rho \in \Xi_r ( \tau_1,\dots, \tau_p )
 } g^{\gamma \circ \rho^{-1} }_{\tau,  s }\,.
\end{eqnarray}
 \end{prop}

\begin{proof}    We prove the proposition by induction on $p$. The proposition is clearly true when    $p=1$, and  by    Lemma \ref{lem 1} it is   true when $p=2$. Take $p \geq 3$.  Assuming that \eref{e2.10i} holds for $p-1$, we can write
\begin{eqnarray}\label{e2.10}
 \int^{s }_{ \tau }   \int^{s_p }_{\tau}  \cdots   \int^{s_{2 } }_{\tau}d g^{\gamma^1}_{\tau,  s_1}  \cdots dg^{\gamma^{p-1}}_{\tau,  s_{p-1}}  dg^{\gamma^p }_{\tau,  s_p}
          &=&
         \sum_{  \wt{\rho}  \in \Xi_{\wt{r} } ( \tau_1,\dots, \tau_{p-1} )
 } \int^{s }_{ \tau }
 g^{\wt{\gamma} \circ \wt{\rho}^{-1} }_{\tau,  s_p }
   dg^{\gamma^p}_{\tau,  s_p},
\end{eqnarray}
where $\wt{\gamma} = (\gamma^1, \dots, \gamma^{p-1 })$ and $\wt{r} = \tau_{p-1}$. Applying
 Lemma \ref{lem 1}  to the right-hand side of \eref{e2.10} we have
\begin{eqnarray}\label{e.3.10}
 \int^{s }_{ \tau }   \int^{s_p }_{\tau}  \cdots   \int^{s_{2 } }_{\tau}d g^{\gamma^1}_{\tau,  s_1}  \cdots dg^{\gamma^{p-1}}_{\tau,  s_{p-1}}  dg^{\gamma^p }_{\tau,  s_p}
          &=&
    \sum_{  \wt{\rho} \in \Xi_{\wt{r} } ( \tau_1 , \dots, \tau_{p-1} )
 }
 \sum_{\widehat{\rho} \in \Xi_r(  \tau_{p-1}, \tau_p ) } g_{\tau,  s}^{( \wt{\gamma} \circ \wt{\rho}^{-1}, \gamma^p ) \circ \widehat{\rho}^{-1} }
      .
\end{eqnarray}
For each
 $ \wt{\rho} \in \Xi_{\wt{r} } ( \tau_1 , \dots, \tau_{p-1} )$
 and $\widehat{\rho} \in \Xi_r(  \tau_{p-1}, \tau_p )   $,
we define a permutation $\rho= \rho_{\wt{\rho}, \widehat{\rho}}$ on $\{1,\dots, r\}$ such that
\begin{eqnarray*}
\rho (j) &=&
\begin{cases}
\widehat{\rho}(j), ~~~~~~j=\tau_{p-1}+1,\dots, \tau_p,
 \\
\widehat{\rho} ( \wt{\rho}(j)  ),
~~j =1,\dots, \tau_{p-1} .
\end{cases}
\end{eqnarray*}
 Then \eref{e.3.10} is equal to
\begin{eqnarray}\label{eq 4.12i}
  \sum_{  \wt{\rho} \in \Xi_{\wt{r} } ( \tau_1 , \dots, \tau_{p-1} )
 }
 \sum_{\widehat{\rho} \in \Xi_r(  \tau_{p-1}, \tau_p ) } g_{\tau,  s}^{\gamma \circ \rho_{\wt{\rho}, \widehat{\rho}}^{-1} }.
 \end{eqnarray}
 We show that
 \begin{eqnarray}\label{eq 4.12}
  \Xi_r(\vec{\tau}) &=& \left\{\rho_{\wt{\rho}, \widehat{\rho} } : \wt{\rho} \in \Xi_{\wt{r} } ( \tau_1 , \dots, \tau_{p-1} ), \widehat{\rho} \in \Xi_r (  \tau_{p-1}, \tau_p )
 \right\} .
 \end{eqnarray}
  It is easy to see that $ \Xi_r(\vec{\tau}) $ is included on the right-hand side of \eref{eq 4.12}, that is, for each $\rho \in \Xi_r(\vec{\tau})$, we can find $   \wt{\rho} \in \Xi_{\wt{r} } ( \tau_1 , \dots, \tau_{p-1} )$ and $ \widehat{\rho} \in \Xi_r (  \tau_{p-1}, \tau_p )
 $ such that $\rho = \rho_{\wt{\rho},  \widehat{\rho}   }$.

In the following, we show the other inclusion.  We take $\rho = \rho_{ \wt{\rho},  \widehat{\rho}   }$ from the right-hand side of \eref{eq 4.12}. Rule 4 for $\wt{\rho} $ implies that $\wt{\rho} (\tau_i) < \wt{\rho} (\tau_j)  $ for $i,j=1,\dots, p-1$ and $i<j$. This fact and  Rule 5 for $\widehat{\rho}$   imply that
 $$
  \rho(\tau_i)= (\widehat{\rho} \circ\wt{\rho}) (\tau_i) < (\widehat{\rho} \circ \wt{\rho}) (\tau_j) = \rho(\tau_j)
  $$
   for $i,j=1,\dots, p-1$ and $i<j$.  On the other hand, Rule 4 for $\widehat{\rho}$ implies that
   $$ \rho(\tau_{p-1} ) =\widehat{\rho}(\wt{\rho} (\tau_{p-1})) =\widehat{\rho}(\tau_{p-1}) < \widehat{\rho}(\tau_p)={\rho}(\tau_p). $$
   Therefore, $\rho$ satisfies Rule 4 in the definition of  $\Xi_r(\vec{\tau} )$.
 Take now $y<y'$ and $y,y' \in I_i$, $i=1,\dots, p-1$. By Rule 5 for $\wt{\rho}$,  we have $\wt{\rho} (y) < \wt{\rho} (y')$, and thus $\rho (y) = \widehat{\rho}(\wt{\rho} (y)   ) < \widehat{\rho}(\wt{\rho} (y')   ) = \rho (y') $. On the other hand, if $y<y'$ and $y,y' \in I_p$, then by Rule 5 in the definition of $\widehat{\rho}$, we have $\rho(y) =\widehat{\rho}(y)< \widehat{\rho}(y') = \rho(y')$. We conclude that $\rho $ satisfies Rule 5 in the definition of   $\Xi_r(\vec{\tau} )$.
 In summary, we have shown that $\rho \in \Xi_r(\vec{\tau})$. This proves   identity \eref{eq 4.12}.

  It is easy to show that there is no duplicated element in the set on the right-hand side of \eref{eq 4.12}, that is, whenever $\wt{\rho} \neq \wt{\rho}'$ or $\widehat{\rho} \neq \widehat{\rho}'$,
 we have $\rho_{ \wt{\rho}, \widehat{\rho} } \neq\rho_{ \wt{\rho}', \widehat{\rho}' }$. This fact, together with   identity \eref{eq 4.12}, imply that \eref{eq 4.12i} is equal to
 $$ \sum_{\rho \in \Xi_r(\tau_1,\dots, \tau_p)} g_{\tau,  s}^{\gamma \circ \rho^{-1}}.  $$
 This completes the proof. \end{proof}

 \subsection{A generalized Leibniz rule}\label{subsection 4.3}
 Using the  permutation set     $\Xi_r(\vec{\tau})$, we can state the following Leibniz rule. Recall that given a multi-index $\al=(\al_{1},\dots, \al_{r})$ we denote $\al_{i,j} = (\al(i+1),\dots, \al(j-1))$.  We also use
    ${\mathcal{V}}_{j_1,\dots, j_k} := {\mathcal{V}}_{j_1}\cdots {\mathcal{V}}_{j_k}$,  where
    ${\mathcal{V}}_j$ is   defined in \eref{eq 1.3i}.
    \begin{prop}\label{lem 2.5}
 Take  $\alpha = (\al_{1},\dots, \al_{r}) \in \Gamma_r$, $\vec{l}=(l_1,\dots, l_p)$
  such that $ 1\leq l_1<\cdots<l_p = r$,   $  p\leq r$.
 Assume that $  f^i_j  \in C^{r-p} (\RR)$, $i=1,\dots, p$, $j=1,\dots, m$.    Then the following Leibniz rule holds:
   \begin{eqnarray*}
 &&  {\mathcal{V}}_{\alpha_{ 0, l_1  } }
       \left( f^{1 }_{\alpha_{l_1 }  } \cdots
              {\mathcal{V}}_{\alpha_{ l_{p-2}, l_{p-1} }  } \left(  f^{ p-1 }_{\alpha_{l_{p-1} }  }   {\mathcal{V}}_{\alpha_{ l_{p-1}, l_p }  }  f^{ p}_{\alpha_{ l_p }  }  \right)
  \right)
  \nonumber
 \\
  &&=
  \sum_{
      1\leq \tau_1 < \cdots < \tau_p   = r
    }
      \sum_{
    \rho \in \Xi_r ( \vec{l}  ; \, \vec{\tau}\,)
     }
      \left(    {\mathcal{V}}_{\al \circ \rho ( 0,  \tau_1 ) } f^{1}_{\al \circ \rho (   \tau_1 )  } \right)
    \dots
  \left(    {\mathcal{V}}_{\al \circ \rho ( \tau_{p -1} , \tau_p ) } f^{p}_{\al \circ \rho (    \tau_p  )  } \right)
                     \,.
      \end{eqnarray*}
 \end{prop}
 \begin{proof}
  The above  formula      follows from Lemma \ref{lem2.3} and Proposition \ref{lem 2.1}.  \end{proof}

\section{The error function}\label{section 5}
The objective of this section is to derive   an explicit expression for   the remainder
in the incomplete  Taylor scheme \eref{eq1.7}.
Let $y$ be the solution of the   differential equation
\begin{eqnarray}\label{eqn5.10}
dy_{t} &=&   V(y_{t} ) dx_{t}  , \quad y_{0} \in \RR^{d},
\end{eqnarray}
on $ [ 0, T]$, where $x: [0,T] \rightarrow  \RR^{m}$ is H\"older continuous of order $\beta>1/2$ and  $V=(V^{i}_{j})_{1\leq i\leq d, 1\leq j\leq m} $ is a continuous mapping from $ \RR^{d}$ to $  \RR^{d\times m}$.
We consider  the Taylor scheme
  \begin{eqnarray}\label{e3.2i}
 y^{n}_{t}  &=&
y^n_{t_k}+ \mathcal{E}_{ t_k,   t }^{(N)}  ( y^n_{t_k   }   ),  \quad y^{n}_{0} = y_{0},
	 \end{eqnarray}
for $  t \in [t_k, t_{k+1}]$, $  k=0,1,\dots, n-1$, where
\begin{eqnarray}\label{eqn5.1}
\mathcal{E}_{s,t}^{(N)}(y  )
&:=&
 \sum_{r=1}^N \sum_{ \gamma \in \Gamma_r  }
    {\mathcal{V}}_{\gamma   } I(y)     x^{ \gamma }_{s,t}
  \,, \quad y \in \RR^{d}
  \end{eqnarray}
  is
    the  order-$N$  Taylor expansion and  ${\mathcal{V}}_\gamma := {\mathcal{V}}_{\gamma(1)} \cdots {\mathcal{V}}_{\gamma(r)}$ for $\gamma=(\gamma(1),\dots, \gamma(r)) \in \Gamma_r $ (the collection of multi-indices of length $r$  with elements in $\{1,\dots, m\}$). Recall that      $I$ is the identity function on $\RR^{d}$,
     the vector field ${\mathcal{V}}_{j}$ is defined in \eref{eq 1.3i} and   for $\gamma \in \Gamma_r$    we denote (see \eref{e2.6})
  \begin{eqnarray*}
x^{  \gamma}_{s,t} :  =\int_{s}^{t}\cdots \int_{s}^{t_{2}} d x^{\gamma(1)}_{t_{1}} \cdots d x^{\gamma(r)}_{t_{r}}.
\end{eqnarray*}

 By \eref{eqn5.10} we can write
  \begin{eqnarray}\label{e3.3}
  y_t-y^n_t &=& \int_0^t \left[ V(y_{s})  -  V(y^n_{s})\right] dx_{s} +R_{t},
\end{eqnarray}
where $R_{t}:=\int_0^t V(y^n_{s} ) dx_{s} - y^n_t$ is the remainder term.
 If we can find a function $ \dot{V}_{j}(\xi, \xi') $   on $\RR^{d}\times\RR^{d}$ (see \eref{eqn5.14}) such that
\begin{eqnarray*}
V_{j}(y_{t})-V_{j}(y^{n}_{t}) = \dot{V}_{j}(y_{t},y^{n}_{t}) (y_{t}-y^{n}_{t}),
\end{eqnarray*}
   $j=1,\dots, m$, then  equation \eref{e3.3} can be considered as a linear differential equation for the error function $y-y^{n}$.
     Our aim in this section is to derive an explicit expression for   the remainder $R_{t}=\int_0^t V(y^n_{s} ) dx_{s} - y^n_t$.

  \subsection{A linear   differential equation for the error function}\label{section3.1}
 We   first   consider  the differential operator ${\mathcal{V}}_{\gamma}$ appearing in \eref{eqn5.1}.
 We denote by  $  {\Upsilon} _p$   the collection of multi-indices of length $p$  with elements in $ \{   1, \dots,  d\}$
 and $\displaystyle  {\Upsilon}=\cup_{p=1}^\infty  {\Upsilon}_p$.
  For $\zeta \in  {\Upsilon}_p $, we  denote
\begin{eqnarray*}
 \partial_{\zeta  } &: =& \partial_{\zeta_1} \partial_{\zeta_2} \cdots \partial_{\zeta_{p} }  \,,
\end{eqnarray*}
 where $\zeta_j$ is the $j$th element of $\zeta$ and recall that $\partial_{i}=\frac{\partial}{\partial y^{i}}$.
For    $\al\in \Gamma_r $ such that   $p\leq r$,       $\vec{\tau} = (\tau_1, \dots, \tau_p)$ such that $ 1\leq \tau_1 <\cdots<\tau_p=r$,  and $y \in \RR^{d}$
 we introduce the function
\begin{eqnarray}\label{eqn 5.1}
H(  \al, \zeta,   \vec{\tau}   ) (y)
 &=&
 \left(    {\mathcal{V}}_{\al_{  0,  \tau_1 }  } V^{\zeta_1}_{\al_{ \tau_1 }  } (y) \right)
    \dots
  \left(     {\mathcal{V}}_{\al_{  \tau_{p -1} , \tau_p }  } V^{\zeta_p}_{\al_{     \tau_p }  } (y) \right)
\nonumber
  \\
  &=&
   \left(    {\mathcal{V}}_{\al_{  0,  \tau_1+1 }  } I_{\zeta_1}  (y) \right)
    \dots
  \left(     {\mathcal{V}}_{\al_{  \tau_{p -1} , \tau_p+1 }  } I_{\zeta_p}  (y) \right),
    \end{eqnarray}
   where $I_i (y) = y_i$, $i=1,\dots, d$ is the projection function, and recall that given a multi-index $\al=(\al_{1},\dots, \al_{r})$, we denote $\al_{i,j} = (\al_{i+1},\dots, \al_{j-1})$. Note that the second equation in \eref{eqn 5.1} follows from the identity $\mathcal{V}_j I_{i}(y) = V^i_j (y)$.

 \begin{lemma}\label{lem 5.1}
Let  $f \in C^{r}(\RR^{d} )$ and $\al \in \Gamma_r $.
Then the  following  identity holds true:
 \begin{eqnarray}\label{e3.4}
{\mathcal{V}}_\al   f
&=&
 \sum_{p = 1}^{ r }
 \sum_{ \zeta \in {\Upsilon}_p  }
   \sum_{\substack{\vec{\tau}=(\tau_1 , \dots , \tau_p  ):\\
    1\leq \tau_1 < \cdots < \tau_p   = r
  } }
    \sum_{  \rho \in \Xi_r (   \vec{\tau}\,) }
         H(   \al \circ \rho, \zeta,  \vec{\tau}    )
   \partial_{\zeta } f
  \, .
\end{eqnarray}
 \end{lemma}
\begin{proof}
It is easy to show by induction and by the  definition of the vector field ${\mathcal{V}}_j$ that
\begin{eqnarray*}
        {\mathcal{V}}_{\alpha   }f
    &=&
   \sum_{p = 1}^{ r }
 \sum_{ \zeta \in {\Upsilon}_p  }
   \sum_{1\leq l_1 < \cdots < l_p = r }
      {\mathcal{V}}_{\alpha_{ 0, l_1 }  }\left( V^{\zeta_1  }_{\alpha_{ l_1 }  } \cdots {\mathcal{V}}_{\alpha_{ l_{p-2}, l_{p-1} }  } \left(  V^{\zeta_{p-1} }_{\alpha_{ l_{p-1} }  }    {\mathcal{V}}_{\alpha_{ l_{p-1}, l_p }  }  V^{ \zeta_{p}  }_{\alpha_{ l_p }  }  \right)
  \right)   \partial_{\zeta } f
   .
 \end{eqnarray*}
Applying  Proposition \ref{lem 2.5} to the above expression yields
 \begin{eqnarray*}
{\mathcal{V}}_\al   f
&=&
 \sum_{p = 1}^{ r }
 \sum_{ \zeta \in {\Upsilon}_p  }
   \sum_{\substack{
   1 \leq l_1  < \cdots < l_p = r\\
   1\leq \tau_1 < \cdots < \tau_p   = r
  } }
    \sum_{  \rho \in \Xi_r ( \vec{l}  ; \, \vec{\tau}\,) }
         H(   \al \circ \rho, \zeta,  \vec{\tau}    )
   \partial_{\zeta } f
  \, .
\end{eqnarray*}
Equation    \eref{e3.4} is then obtained by noticing  that the following two sets are identical:
\begin{eqnarray*}
\Xi_{r}(\vec{\tau}) &=& \bigcup_{ \substack{ \vec{l}=(l_{1},\dots, l_{p}) :\\
1\leq l_{1}<\cdots < l_{p}=r.
 } } \Xi_{r}(\vec{l}, \vec{\tau}) .
\end{eqnarray*}
This completes the proof.
\end{proof}

\medskip

Take $\zeta \in  {\Upsilon}_p$. We denote by  ${\mathcal{E} }^{ \zeta }_{s, t} $  the   multiple   integral  
\begin{eqnarray}\label{eqn 5.2}
  {\mathcal{E} }^{ \zeta }_{s, t} & = &    \int_{ s }^{t } \int_{  s  }^{t_r}       \cdots  \int_{  s  }^{t_{2}} d\mathcal{E}^{ \zeta_1 }_{s,t_{1 }} \cdots d\mathcal{E}^{  \zeta_{ p-1 }  }_{s,t_{r-1} } d\mathcal{E}^{ \zeta_{  p }  }_{s,t_r },
\end{eqnarray}
where
\begin{eqnarray}\label{e3.8ii}
\mathcal{E}_{s,t}^{i}(y  )
&=&
 \sum_{r=1}^N \sum_{ \gamma \in \Gamma_r }
    {\mathcal{V}}_{\gamma   } I_{i}(y)     x^{ \gamma }_{s,t}
  \,, \quad y \in \RR^{d}.
\end{eqnarray}

     According to Proposition \ref{prop 4}, the multiple integral ${\mathcal{E} }^{ \zeta}_{s, t}$ can be expressed as a linear combination of  elements in  $\{ x^{\al}_{s,t}\,, \al \in \Gamma \} $. Recall that  $ \Gamma = \cup_{r=1} ^\infty \Gamma_r$ is the collection of multi-indices with elements in $\{1,\dots, m\}$.
 The following lemma provides an explicit  formula   for this  linear combination. Recall that for a permutation $\rho$ on $\{1,\dots, r\}$, we denote $\al \circ \rho = (\al_{\rho(1)}, \dots, \al_{\rho(r)}) $.
  \begin{lemma}\label{lem 5.2}
 Take $p \in \NN$  and  $\zeta \in  {\Upsilon}_p $\,. Assume that $V\in C^{N-1}$.
 Then for each $s,t \in [0, T]$ we have
 \begin{eqnarray}\label{e3.7}
     \mathcal{E}^{  \zeta   }_{ s,  t }
   &=&
        \sum_{r=p }^{N   }
        \sum_{\substack{
\vec{\tau}=(\tau_1,\dots,\tau_p):
\\
  1\leq \tau_1<\cdots<\tau_p=r}
 }
   \sum_{\al \in \Gamma_r}
    \sum_{ \rho \in \Xi_r ( \vec{\tau} \, )
 }
     H(   \al \circ \rho, \zeta,   \vec{\tau}   )
       x^{\al   }_{s, t }
          +Q(N,p,\zeta)_{ s,t}
   ,
  \end{eqnarray}
 where
 \begin{eqnarray*}
 Q (N,p,\zeta)_{ s,t}& =&
    \sum_{\substack{
    \gamma^1 ,\dots, \gamma^p \in \Gamma:
    \\
      |\gamma^1|+\cdots+|\gamma^p| > N \\
      |\gamma^1|, \dots, |\gamma^p| \leq N
      }
      }
       \sum_{ \rho \in \Xi_{|\gamma|} ( \vec{\tau}(\gamma) \, )
 }
     H(   \gamma, \zeta,   \vec{\tau} (\gamma)    )
       x^{\gamma \circ \rho^{-1} }_{s, t },
\end{eqnarray*}
 $\gamma= (\gamma^1, \dots, \gamma^p)$, $\vec{\tau} (\gamma) = ( \tau_1(\gamma), \dots, \tau_p(\gamma) )$   and $\tau_i (\gamma) = |\gamma^1|+\cdots+|\gamma^i| $, $i=1,\dots, p$.
\end{lemma}

 \begin{proof}  It follows from \eref{eqn 5.2} and \eref{e3.8ii}  that
 \begin{eqnarray*}
     \mathcal{E}^{  \zeta   }_{ s,  t } (y)
   &=&
      \sum_{\substack{\gamma^1 ,\dots, \gamma^p \in \Gamma:
  \\
1\leq   |\gamma^1| , \dots, |\gamma^p| \leq N }
   }
  \mathcal{V}_{\gamma^1} I_{\zeta_1} (y)    \cdots      \mathcal{V}_{\gamma^p} I_{\zeta_p} (y)
    \int^{t }_{s}   \int^{t_p }_{s} \cdots  \int^{t_{2 } }_{s} d x^{\gamma^1}_{t_1}   \cdots d x^{\gamma^{p-1}}_{t_{p-1}}    d x^{\gamma^p }_{t_p}
     .
  \end{eqnarray*}

 We recall    the notation
 $\gamma= (\gamma^1, \dots, \gamma^p)$,    and $\tau_i (\gamma) = |\gamma^1|+\cdots+|\gamma^i| $, $i=1,\dots, p$.
It follows from   Proposition \ref{prop 4} and the definition of the function $H$ in \eref{eqn 5.1} that
  \begin{eqnarray*} 
    \mathcal{E}^{  \zeta   }_{   s,t } (y)
   &=&
   \sum_{\substack{\gamma^1 ,\dots, \gamma^p \in \Gamma:
  \\
1\leq  |\gamma^1| , \dots, |\gamma^p| \leq N }
   }
     H\left(   \gamma, \zeta,   \vec{\tau} (\gamma)   \right)(y)
      \sum_{ \rho \in \Xi_{|\gamma|} ( \vec{\tau}(\gamma)\, )
 }   x^{\gamma \circ \rho^{-1} }_{s, t },
  \end{eqnarray*}
 or
 \begin{eqnarray}\label{e3.8i}
     \mathcal{E}^{  \zeta   }_{ s,  t } (y)
   &=&
    \left(
 \sum_{r=p }^{N   }
       \sum_{ \substack{ |\gamma^1|+\cdots+|\gamma^p|=r}
   }
   +
    \sum_{\substack{
     |\gamma^1|+\cdots+|\gamma^p| > N \\
      |\gamma^1|, \dots, |\gamma^p| \leq N
      }
      }
   \right)
    \sum_{ \rho \in \Xi_{|\gamma|} ( \vec{\tau}(\gamma)\, )
 }
     H(   \gamma, \zeta,   \vec{\tau} (\gamma)   )
     (y)   x^{\gamma \circ \rho^{-1} }_{s, t }
   .
  \end{eqnarray}
  The second term in the above summation  is exactly $Q  (N,p,\zeta )_{s,t}(y)$.
 On the other hand, since
 \begin{eqnarray*}
\{(\gamma^{1},\dots, \gamma^{p}): \sum_{i=1}^{p} |\gamma^{i}| =r\} &=& \bigcup_{\substack{\tau_{1},\dots, \tau_{p}:\\ 1\leq \tau_{1}<\cdots<\tau_{p} =r} } \{ (\gamma^{1},\dots, \gamma^{p}):  |\gamma^{i}|=\tau_{i}-\tau_{i-1}, i=1,\dots, p  \},
\end{eqnarray*}
where $\tau_{0}=0$, we have
    \begin{eqnarray}\label{eqn 5.3ii}
  &&
     \sum_{ \substack{ |\gamma^1|+\cdots+|\gamma^p|=r}
   }
    \sum_{ \rho \in \Xi_r ( \vec{\tau}(\gamma)\, )
 }
     H(   \gamma, \zeta,   \vec{\tau} (\gamma)   )
       x^{\gamma \circ \rho^{-1} }_{s, t }
 \nonumber
      \\
     && =
         \sum_{\substack{
\vec{\tau}=(\tau_1,\dots,\tau_p):
\\
  1\leq \tau_1<\cdots<\tau_p=r}
 }
     \sum_{ \substack{\gamma^i:  |\gamma^i|=\tau_i - \tau_{i-1} \\ i=1,\dots, p }
   } \,
    \sum_{ \rho \in \Xi_r ( \vec{\tau} \, )
 }
     H(   \gamma, \zeta,   \vec{\tau}     )
       x^{\gamma \circ \rho^{-1} }_{s, t }
     .
 \end{eqnarray}
 Notice that for fixed $\vec{\tau} = (\tau_1,\dots, \tau_p)$ such that $  1\leq \tau_1<\cdots<\tau_p=r$ and $\rho \in \Xi_r(\vec{\tau})$,  we have
    \begin{eqnarray*}
     \sum_{ \substack{\gamma^i:  |\gamma^i|=\tau_i - \tau_{i-1} \\ i=1,\dots, p }
   } \,
    H(   \gamma, \zeta,   \vec{\tau}     )
       x^{\gamma \circ \rho^{-1} }_{s, t }
      &=&
      \sum_{\gamma\in \Ga_r}
     H(   \gamma, \zeta,   \vec{\tau}   )
       x^{\gamma \circ \rho^{-1} }_{s, t }
     ,
 \end{eqnarray*}
so   the quantity in \eref{eqn 5.3ii} is equal to
   \begin{eqnarray*}
   \sum_{\substack{
\tau_1,\dots,\tau_p:
\\
  1\leq \tau_1<\cdots<\tau_p=r}
 }
   \sum_{ \rho \in \Xi_r ( \vec{\tau} \, )
 }
 \sum_{\gamma:|\gamma|=r}
     H(   \gamma, \zeta,   \vec{\tau}   )
       x^{\gamma \circ \rho^{-1} }_{s, t }.
 \end{eqnarray*}
 Since $\rho$ is a bijection on $\{1,\dots,r\}$,  by replacing $\gamma \circ \rho^{-1}$ by $\al$,  the above expression becomes
   \begin{eqnarray*}
   \sum_{\substack{
\tau_1,\dots,\tau_p:
\\
  1\leq \tau_1<\cdots<\tau_p=r}
 }
   \sum_{ \rho \in \Xi_r ( \vec{\tau} \, )
 }
 \sum_{\al \circ \rho :|\al|=r}
     H(   \al \circ \rho, \zeta,   \vec{\tau}   )
       x^{\al }_{s, t }.
 \end{eqnarray*}
 Substituting the above expression into \eref{e3.8i}, we obtain   identity \eref{e3.7}. \end{proof}

      \bigskip

 Let  $f \in C^{N}(\RR)$ and $s, t \in [0.T]$. It follows from \eref{e3.4} in Lemma \ref{lem 5.1} that
  \begin{eqnarray}
  \sum_{r=1}^N \sum_{\al \in \Gamma_r  }   x^{\al}_{s,t} \mathcal{V}_\al f
&&=
  \sum_{r=1}^N \sum_{\al \in \Gamma_r  }    x^{\al}_{s,t}
  \sum_{p = 1}^{ r }
 \sum_{ \zeta \in {\Upsilon}_p  }
   \sum_{\substack{
    1\leq \tau_1 < \cdots < \tau_p   = r
  } }
    \sum_{  \rho \in \Xi_{r} (   \vec{\tau}\,) }
         H(   \al \circ \rho, \zeta,  \vec{\tau}    )
   \partial_{\zeta } f
\nonumber
   \\
   &&=
  \sum_{1\leq p \leq r \leq N}
 \sum_{ \zeta \in {\Upsilon}_p  }
   \sum_{\al \in \Gamma_r }    \sum_{\substack{
    1\leq \tau_1 < \cdots < \tau_p   = r
  } }
    \sum_{  \rho \in \Xi_r (   \vec{\tau}\,) }
      x^{\al}_{s,t}    H(   \al \circ \rho, \zeta,  \vec{\tau}    )
   \partial_{\zeta } f.
   \label{e3.9}
\end{eqnarray}
On the other hand, it follows from \eref{e3.7} in Lemma \ref{lem 5.2} that
 \begin{eqnarray*}
 &&
\sum_{p=1}^N
 \sum_{ \zeta \in {\Upsilon}_p  } \left( \mathcal{E}^{  \zeta   }_{ s,  t }  -   Q(N,|\zeta|,\zeta)_{ s,t} \right) \partial_{\zeta} f
 \\
 &&=
\sum_{p=1}^{N}  \sum_{\zeta\in  {\Upsilon}_p }  \partial_{\zeta} f
    \sum_{r=p }^{N   }
        \sum_{\substack{
\tau_1,\dots,\tau_p:
\\
  1\leq \tau_1<\cdots<\tau_p=r}
 }
   \sum_{\al\in \Ga_r}
    \sum_{ \rho \in \Xi_r ( \vec{\tau} \, )
 }
     H(   \al \circ \rho, \zeta,   \vec{\tau}   )
       x^{\al   }_{s, t } ,
\end{eqnarray*}
which is equal to the right-hand side of \eref{e3.9}. Therefore, we obtain the following identity
\begin{eqnarray}\label{eqn5.7ii}
      \sum_{\al \in \Gamma : 1\leq |\al| \leq N}   x^{\al}_{s,t} \mathcal{V}_\al f
   &=&
 \sum_{\zeta\in  {\Upsilon} :1\leq |\zeta| \leq N }  \mathcal{E}^{  \zeta   }_{ s,  t }  \partial_{\zeta} f -  \sum_{\zeta\in  {\Upsilon} : 1\leq |\zeta| \leq N} Q(N,|\zeta|,\zeta)_{ s,t} \partial_{\zeta} f
      .
  \end{eqnarray}

Notice that
\begin{eqnarray}
\mathcal{E}_{s,t}^{(N)}(y)
  &=&
   \sum_{j=1}^m
     \int_s^t    \sum_{ \al \in \Gamma: 0 \leq   |\al| \leq N-1 }
 x_{s,u}^{ \al  }   { \mathcal{V} }_{\al   } V_{j}(y)     d x^{j}_u
\nonumber
  \\
  &=&
   \sum_{j=1}^m
     \int_s^t
     \left(
     V_j(y) +   \sum_{ \al \in \Gamma: 1 \leq   |\al| \leq N-1 }
  x_{s,u}^{ \al  }   { \mathcal{V} }_{\al   } V_{j}(y)  \right)
     d x^{j}_u  \,.
     \label{eq 5.7}
  \end{eqnarray}
Then,   applying \eref{eqn5.7ii} to the second term on  the right-hand side of \eref{eq 5.7} with $f=V_j$, we obtain the following result.

\begin{prop}\label{prop5.1}
Let   $ \mathcal{E}^{(N)}_{s, t}    $ be the   {order-$N$} Taylor expansion defined in  \eref{eqn5.1}. Assume that $V \in C^{N}$. Then   the following equation holds true,
\begin{eqnarray}\label{eqn5.7i}
   \int_s^t \left(  V (y) +   \sum_{\zeta \in {\Upsilon} : 1 \leq |\zeta| \leq N  }   \mathcal{E}_{s,u}^{  \zeta} (y) \partial_{\zeta} V (y)  \right) dx_u  - \mathcal{E}^{(N)}_{ s,t} (y) & =&
   R^1_{s,t} (y)  ,
   \end{eqnarray}
   where
\begin{eqnarray}
  R^1_{s,t} (y)  &= &
 \sum_{\zeta \in {\Upsilon}_N  } \int_{s}^{t}  \mathcal{E}_{s,u}^{  \zeta} (y) \partial_{\zeta} V (y)  d x_{u}
 \nonumber \\
 &&+ \sum_{\zeta\in {\Upsilon} : 1\leq |\zeta| \leq N-1} \int_s^t Q(N-1,|\zeta|,\zeta)_{ s,u}(y)   \partial_{\zeta} V (y) d x_u\,.
\label{eqn5.15}
   \end{eqnarray}
\end{prop}

\medskip

 Applying the chain rule   repeatedly  we obtain
\begin{align}\label{eqn5.8}
V ( y+ \mathcal{E}^{(N)}_{s, t}(y) )
 =&
 V(y) +   \sum_{\zeta \in {\Upsilon} : 1\leq |\zeta| \leq N  }   \mathcal{E}_{s, t}^{ \zeta} (y) \partial_{\zeta} V(y)
+
\sum_{\zeta \in {\Upsilon}_{ N +1}  }
 \mathcal{E}^{  \zeta }_{s, t} \left( \partial_{\zeta} V (y+ \mathcal{E}^{(N)}_{ s, \cdot } (y) )
 \right) (y) ,
 \end{align}
 where for  $\zeta \in \Upsilon_p$, we denote
 \begin{eqnarray}\label{e3.15}
 \mathcal{E}_{s,t} ^\zeta (f) &=&   \int_{ s }^{t } \int_{  s  }^{t_r}       \cdots  \int_{  s  }^{t_{2}}
  f_{t_1}    d\mathcal{E}^{ \zeta_1 }_{s,t_{1 }} \cdots d\mathcal{E}^{  \zeta_{ p-1 }  }_{s,t_{r-1} } d\mathcal{E}^{ \zeta_{  p }  }_{s,t_r }.
\end{eqnarray}
 Note that  the first two terms on the right-hand side of \eref{eqn5.8} are   the integrands of \eref{eqn5.7i}, so  Proposition \ref{prop5.1} implies
   \begin{eqnarray*}
     \int_s^t V\left( y + \mathcal{E}^{(N)}_{s,u} (y)  \right) dx_u  - \mathcal{E}_{s,t}^{(N)}(y)
    &=&
\sum_{\zeta \in  {\Upsilon}_{N+1 } } \int_s^t
 \mathcal{E}^{  \zeta }_{s,u} \left( \partial_{\zeta} V \left( y + \mathcal{E}^{(N)}_{ s,\cdot } (y) \right)
 \right)(y) dx_u +   R^1_{s,t} (y)  .
 \end{eqnarray*}
 In particular, the difference $  \int_s^t V\left( y + \mathcal{E}^{(N)}_{s,u} (y)  \right) dx_u  - \mathcal{E}_{s,t}^{(N)}(y) $ is equal to the summation of multiple   integrals of order  higher than $N$.
 Our next result  is a generalization of   this property.
 We first introduce a modification of the  order-$N$ Taylor expansion.
\begin{Def}\label{def3.4}
Let $\wt{\Gamma}$ be a finite subset  of $\Gamma$ \emph{(}collection of multi-indices with elements in $\{1,\dots, m\}$\emph{)} and denote
$N=\max\left\{|\al|\,, \al\in \wt{\Ga}\right\}$.
 We define   the  \emph{incomplete Taylor expansion}
\begin{eqnarray}\label{eqn5.11}
 \wt{\mathcal{E} }_{ s,t}^{(N)}(y)
=
  \sum_{ \gamma \in \wt{\Gamma}  }
    {\mathcal{V}}_{\gamma   } I(y)     x^{ \gamma }_{s,t}.
  \end{eqnarray}
  \end{Def}
  If $\mathcal{E}_{s,t}^{(N)}(y)$ is  defined as  in \eref{eqn5.1}
 and
 if  $\wt{\Gamma} = \{\gamma \in \Gamma: |\gamma| \leq N\} $  we have
 $ \wt{\mathcal{E}}^{(N)}_{s, t} = \mathcal{E}_{s, t}^{(N)}  $.
In the following,     $  \wt{\mathcal{E}}_{ s,t}^\zeta (f)  $  is the multiple integral  defined as in \eref{e3.15}
 by replacing $    {\mathcal{E}}_{ s,t}^{\zeta_{j}}$ by $   \wt{\mathcal{E}}_{s, t}^{\zeta_{j}}$ in \eref{e3.15}.

  \begin{prop}\label{prop5.2}
Let   $ \wt{\mathcal{E}}^{(N)}_{s, t}$ and  $  {\mathcal{E}}_{s, t}^{(N)} $, $t\in [0,T]$ be the incomplete Taylor expansion and the order-$N$ Taylor expansion   in Definition \ref{def3.4} with $N=\max_{\gamma \in \wt{\Gamma} }|\gamma| $. Assume that $V \in C^{N+1}$. Then   the following equation holds true,
\begin{eqnarray}\label{eq 5.15}
   \int_s^t V(y + \wt{\mathcal{E}}^{(N)}_{s, u} (y) ) dx_u  - \wt{\mathcal{E}}^{(N)}_{s, t} (y) =
    \sum_{e=1}^4 R^e_{ s,t} (y),
   \end{eqnarray}
   where $R^1_{s,t}(y)$ is the same as in \eref{eqn5.15}, and
\begin{eqnarray}
R^2_{s,t} (y) &= &
 \sum_{\zeta \in {\Upsilon}_{N +1}  } \int_s^t
 \wt{\mathcal{E}}_{ s,u}^{ \zeta }  \left( \partial_{\zeta} V(y+ \wt{\mathcal{E}}^{(N)}_{ s,\cdot}  (y) )
   \right)  (y ) dx_u,
 \label{eqn5.15i}
\\
 R^3_{s,t} (y)  &=&
\sum_{\zeta \in {\Upsilon} : 1\leq |\zeta| \leq N  }  \int_s^t
 \left(  \wt{\mathcal{E}}_{s, u}^\zeta(y)
    -     {\mathcal{E}}_{ s,u}^\zeta(y)
 \right) \partial_{\zeta} V(y)
  d x_u,
\label{eqn 5.12}
   \\
R^4_{s,t} (y) &= &
 \mathcal{E}_{ s,t}^{(N)} (y)-  \wt{\mathcal{E}}^{(N)}_{s, t} (y)
  \,.
 \label{eqn5.14i}
  \end{eqnarray}
\end{prop}
 \begin{proof}
As in \eref{eqn5.8},  applying the chain rule several times we obtain
 \begin{align} \label{e3.20}
V (  y + \wt{\mathcal{E}}^{(N)}_{s,t}  (y) )
 =&
 V(y) +   \sum_{\zeta \in {\Upsilon} : 1\leq |\zeta| \leq N  }    \wt{\mathcal{E}} _{s, t}^\zeta(y)
   \partial_{\zeta}  V(y)
+
\sum_{\zeta \in {\Upsilon} : |\zeta| = N+1  }
  \wt{\mathcal{E}} _{s, t}^\zeta \left( \partial_{\zeta} V(y+  \wt{\mathcal{E}}^{(N)}_{s, \cdot} (y)   )
 \right) (y).
 \end{align}
 Integrating both sides of  the above equation with respect to $dx_u$ over $[s,t]$ and then subtracting $\wt{\mathcal{E}}^{(N)}_{s,t} (y) $, we obtain
 \begin{eqnarray*}
   &&
   \int_s^t V(y + \wt{\mathcal{E}}^{(N)}_{ s,u} (y) ) dx_u  - \wt{\mathcal{E}}^{(N)}_{s, t} (y)
   \\
   &&=
         \int_s^t   V (y) dx +   \int_s^t \left(   \sum_{\zeta \in {\Upsilon} : 1 \leq |\zeta| \leq N  }   \mathcal{E}_{s,u}^{  \zeta} \partial_{\zeta} V (y)  \right) dx_u  - \mathcal{E}_{ s,t}^{(N)}(y)
        +   \sum_{e=2}^4 R^e_{ s,t} (y) .
   \end{eqnarray*}
 Applying Proposition \ref{prop5.1} to the above equation we obtain     equation \eref{eq 5.15}.
  \end{proof}

\begin{Def}
Take $\al, \al'\in \Gamma $ such that $ |\al|=r$ and $|\al' |=r+1$.
We say that $\al $ is contained in $\al'$, denoted by $\al \Subset \al'$, if there is an injection $\rho$ from $\{1,\dots, r\}$ to $\{1,\dots, r+1\}$ such that $\al(i)= \al'(\rho(i))$, $i=1,\dots, r$.
\end{Def}

\begin{Def}  \label{hi}
We say that $\wt{\Gamma}\subset \Gamma$ has a  {\it hierarchical structure} if  for any $\al \in \Gamma\setminus\wt{\Gamma}$ and $\al \Subset \al'$, we have $\al' \in \Gamma\setminus\wt{\Gamma}$.
\end{Def}

The following result shows that the difference $   \int_s^t V(y + \wt{\mathcal{E}}^{(N)}_{s, u} (y) ) dx_u  - \wt{\mathcal{E}}^{(N)}_{s, t} (y) $ is equal to the summation of multiple integrals of order ``higher'' than those in   $\{x^{\al}: \al \in   \wt{\Gamma}\}$.
 \begin{prop}\label{prop4.7}
 Let the assumptions be as in Proposition \ref{prop5.2}. Then the following statements hold true:

\noindent \emph{(i)} $ R^{1}_{s,t}(y)$ is a linear combination of the multiple integrals in
$\left\{  x^{\al}_{s,t}: \al \in \Gamma,\, |\al| \geq N+1  \right\},
$ and the coefficients of this  combination are   the products of $ {\mathcal{V}}_{\gamma   } I_{i}(y) $ and $\partial_{\zeta} V_{j} (y) $ for $\gamma \in \Gamma$ such that $|\gamma|\leq N$ and $\zeta \in \Upsilon$ such that $|\zeta|\leq N$, $i=1,\dots, d$, $j=1,\dots, m$.

\noindent \emph{(ii)} $ R^{4}_{s,t}(y)$ is a linear combination of the multiple integrals in
$
\left\{  x^{\al}_{s,t}: \al \in \Gamma \setminus \wt{\Gamma} \right\},
$
and the coefficients   are
$
   {\mathcal{V}}_{\gamma   } I(y)
$ for $\gamma \in \Gamma \setminus \wt{\Gamma}$ such that $|\gamma|\leq N$.

\noindent \emph{(iii)} Assume that $\wt{\Gamma}$ has the {\it hierarchical structure} introduced in Definition \ref{hi}.   Then $ R^{3}_{s,t}(y)$ is a linear combination of the multiple integrals in
$ \left\{  x^{\al}_{s,t}: \al \in \Gamma \setminus \wt{\Gamma}  \right\},
$ and the coefficients   are products of
$    {\mathcal{V}}_{\gamma   } I_{i}(y)
$ and $ \partial_{\zeta} V_{j}(y)  $  for $\gamma \in \Gamma$ such that $|\gamma|\leq N$ and $\zeta \in \Upsilon$ such that $|\zeta|\leq N$, $i=1,\dots, d$, $j=1,\dots, m$.
 \end{prop}

With the help of Proposition \ref{prop5.2} we can now derive an equation for the  global  error function of the  incomplete Taylor  scheme  \eref{eq1.7}   associated with the incomplete Taylor expansion $\wt{\mathcal{E}}^{(N)}_{s,t} (y)$ in \eref{e1.8i} or \eref{eqn5.11}; that is, for the
 global error function   of the numerical scheme
\begin{eqnarray}\label{e3.21}
y^{n}_{t}  &=&
y^n_{t_k}+ \wt{\mathcal{E}}^{(N)}_{ t_k,   t } ( y^n_{t_k   }   ), \quad t \in [t_{k}, t_{k+1}], \quad k=0,1,\dots, n-1.
\end{eqnarray}
Recall that   $t_{k}=kT/n$, $k=0,1,\dots, n-1$, $\lfloor a\rfloor$ is the integer part of $a$ and $a \wedge b$ is the small of $a$ and $b$ for $a, b\in \RR$.
\begin{prop}\label{prop 4.1}
Let $y$ and $y^n$ be the solutions of equation \eref{eqn5.10}  and \eref{e3.21}, respectively.  Let $R_{s,t}^{e} (y)$, $e=1,2,3,4$, be the functions defined   in  \eref{eqn5.15}, \eref{eqn5.15i}, \eref{eqn 5.12} and \eref{eqn5.14i}. 
Assume that $V\in C^{N+1}$.
Then  the error function $y - y^n $ satisfies the equation
\begin{eqnarray}\label{eqn5.21}
y_t - y^n_t  =    \int^t_0\left( V(y_s)- V(y^n_s) \right)       dx_s
 +  \sum_{e=1}^4 \sum_{k=0}^{\lfloor \frac{nt}{T} \rfloor} R^e_{t_k , t_{k+1} \wedge t} (y^n_{t_k})  ,
  \end{eqnarray}
  for $t\in [0,T]$.
  \end{prop}
  \begin{remark}
  Denote  $\epsilon:=y - y^n  $, and  set
  \begin{eqnarray}
  \dot{V}_{j  }(  \xi , \xi')    :=
 \int_0^1 \partial V_{j } (  \theta \xi+ (1-\theta ) \xi' ) d \theta,
\label{eqn5.14}
\end{eqnarray}
for $  \xi, \xi' \in \RR^{d}$, $j=1,\dots, m$.
 The following linear   differential equation   for $\epsilon$ can be     easily  derived from \eref{eqn5.21},
\begin{eqnarray} \label{eq 5.20}
\epsilon_t = \sum_{j=1}^m  \int^t_0 \dot{V}_j ( y_u, y^n_u )   \epsilon_u    dx_u^j
 +  \sum_{e=1}^4 \sum_{k=0}^{\lfloor \frac{nt}{T} \rfloor} R^e_{t_k,   t_{k+1} \wedge t} (y^n_{t_k}) \, .
  \end{eqnarray}
\end{remark}

\noindent \textit{Proof of Proposition \ref{prop 4.1}:}\quad
 By taking $s=t_k$, $t \in [t_k, t_{k+1}]$ and $y=y^n_{t_k}$ in \eref{eq 5.15} we obtain

 \begin{eqnarray*}
   \int_{t_k}^t V(y^n_u ) dx_u  - \wt{\mathcal{E}}^{(N)}_{t_k, t} ( y^n_{t_k} ) =
    \sum_{e=1}^4 R^e_{ t_k ,t} (y^n_{t_k})\,.
   \end{eqnarray*}
This  implies that
    \begin{eqnarray*}
 \sum_{i=0}^k  \int_{t_i}^{t_{i+1} \wedge t} V(y^n_u ) dx_u  -  \sum_{i=0}^k  \wt{\mathcal{E}}^{(N)}_{t_i , t_{i+1} \wedge t} ( y^n_{t_i } ) =
    \sum_{i=0}^k  \sum_{e=1}^4 R^e_{ t_i ,  t_{i+1} \wedge t} (y^n_{t_i}),  ~~t \in [t_k, t_{k+1}],
   \end{eqnarray*}
 or
  \begin{eqnarray*}
    \int_{0}^{t } V(y^n_u ) dx _u -  y^n_t =
    \sum_{i=1}^{\lfloor \frac{nt}{T} \rfloor}  \sum_{e=1}^4 R^e_{ t_i ,  t_{i+1} \wedge t} (y^n_{t_i}).   \end{eqnarray*}
  Equation \eref{eqn5.21}  then follows   by noticing equation \eref{e3.3}.
    \end{proof}

  \medskip

  \subsection{The explicit expression for the error function}\label{section3.2}
  In this subsection we   derive an explicit expression of  the error function $y-y^n$, where $y$ and $y^n$ are   solutions of equation \eref{eqn5.10}  and \eref{e3.21}, respectively,  with $\wt{\mathcal{E}}^{(N)}_{s,t}(y)$ being the incomplete Taylor expansion \eref{eqn5.11}.

 We define the \emph{fundamental equation} of \eref{eq 5.20},
\begin{eqnarray}\label{eqn5.24}
 {\Phi}_{  t} = I + \sum_{j  = 1}^m   \int_0^t
\dot{V}_j (y_s, y^n_s)
 {\Phi}_{s}    d x^{j }_s \,,
\end{eqnarray}
and  its inverse,
\begin{eqnarray}\label{eqn5.23}
 {\Psi}_{t}  = I - \sum_{j  = 1}^m  \int_0^t   {\Psi}_{s}   \dot{V}_j ( y_s, y^n_s) d x^{j }_s\,,
\end{eqnarray}
for  $   t \in [0, T]$. Recall that $\dot{V}$ is defined in \eref{eqn5.14} and $I$ is the $d\times d$ identity matrix.
 The fact that  $ {\Psi}$ is the inverse of the function $ {\Phi}$, i.e.    $  {\Psi}  {\Phi} \equiv I$, can be shown by applying the product rule to $\Psi \Phi$ and taking into account the identity\begin{eqnarray*}
\int_{0}^{t} \Psi_{s} \cdot d\Phi_{s} + \int_{0}^{t} d\Psi_{s} \cdot \Phi_{s}=0.
\end{eqnarray*}

The following result provides an explicit expression for the error function $y-y^{n}$ under the above assumptions. We denote $\eta(t)=t_{k}$ for $t  \in [t_{k},t_{k+1})$.
\begin{theorem}\label{cor 5.9}
Let assumptions be as in Proposition \ref{prop 4.1}. The following expression of $y-y^n$ holds true for   $t \in [0,T]$,
\begin{eqnarray}\label{eq3.24}
y_t-y^n_t &=& \sum_{e=1}^4  {\Phi}_t \sum_{k=0}^{\lfloor \frac{nt}{T}  \rfloor}
 \int_{t_k  }^{t_{k+1} \wedge t}    {\Psi}_s d     R^e_{t_k, s} (y^n_{t_k})
 \\
 &=& \sum_{e=1}^{4} \Phi_{t} \int_{0}^{t} \Psi_{s} dR^{e}_{\eta(s),s} (y^{n}_{\eta(s)})  \,.\nonumber
\end{eqnarray}
\end{theorem}
    \begin{proof}  By applying the product rule to the quantity on the right-hand side of   equation \eref{eq3.24} and taking into account   identities \eref{eqn5.24} and $ {\Phi}   {\Psi}   \equiv I$, we can show that this quantity satisfies   equation \eref{eq 5.20}, and by  the uniqueness of the solution of   equation \eref{eq 5.20}, we conclude that it is equal to $y_{t}-y^{n}_{t}$.
    \end{proof}

    \medskip

\section{The incomplete Taylor  scheme}\label{section 6}
Let $y$ be the solution of the differential equation \eref{eqn5.10} and let  $x^{j}$ be  H\"older continuous of order $\beta_{j}
>1/2$.
Given any finite set $\wt{\Ga}$ of $\Ga$   {(}collection of multi-indices with elements in $\{1,\dots, m\}${)}
let $y^{n}$ be the approximation solution defined by     \eref{e3.21},  where  $\wt{\mathcal{E}}^{(N)}_{s,t} (y)$ is the incomplete Taylor expansion   in Definition \ref{def3.4}.  In this section, we study the convergence rate of $y^{n}$ to $y$. Denote $\beta:=\min_{j}\beta_{j} $.

 Let $a, b \in [0,T]$ with $a < b$ and   $\delta \in ( 0, 1)$.
  For a function $z:[0,T] \rightarrow \mathbb{R}$, $\|z\|_{a, b, \delta}  $ denotes the $\delta$-H\"older seminorm of $z$ on $[a, b]$,
that is,
\begin{eqnarray*}
   \|z\|_{a,b, \delta} = \sup \left\{\frac{|z_u- z_v|}{(v-u)^{\delta}}:\, a \leq u < v \leq b\right\}.
\end{eqnarray*}
 We will denote the uniform norm of $z$ on the interval $[a, b]$ by $\| z\|_{a, b, \infty}$. When $a=0$ and $b=T$, we will simply write $\|z\|_\infty$ for $\|z\|_{0,T, \infty}$ and $\|z\|_\delta$ for $\|z\|_{0,T, \delta}$.

The following lemma   provides some upper bounds of $y^{n}$, $\Phi$ and $\Psi$.
 \begin{lemma}\label{lem 6.4}
Let  $\wt \Ga$ be a finite subset of $\Ga$ and assume  $V \in C^{ {N+1}}_{b}$.  Let   $y^{n}$   be the solution of  \eref{e3.21}. We have the following  estimate
 \begin{eqnarray}\label{eqn6.1}
 \|y^n\|_\be \leq  C \|x\|_\be \vee \|x\|_\be^{1/\be +N-1} ,
  \end{eqnarray}
  where $C$ is a constant independent of $n$.
  Furthermore,
  the following estimate holds true for   the functions ${\Phi}$ and ${\Psi}$   defined in \eref{eqn5.24} and \eref{eqn5.23},
   \begin{eqnarray*}
  \|  {\Phi}  \|_{\be} \vee  \|  {\Phi}  \|_{\infty} \vee \|  {\Psi}  \|_{\be} \vee  \|  {\Psi}  \|_{\infty}  \leq  C \exp (  C\|x\|_\be^{1/\be^{2 } +(N-1)/\beta } ).
\end{eqnarray*}
 \end{lemma}

\begin{proof}  The upper bound estimate for $y^{n}$ follows immediately from Lemma \ref{prop 4.4}.
 We turn to the function $\Phi$. Consider the   following  system of equations
  \begin{eqnarray*}
  y^n_t &=&y_{0}+\int_{0}^{t} dy^n_s\,,
  \\
  y_t&=& y_0+ \int_0^t V(y_s) dx_s\,,
  \\
  {\Phi} _t &=&
   I + \sum_{j  = 1}^m   \int_0^t
\dot{V}_j (y_u, y^n_u)
{\Phi} _u  d x^{j }_u \,.
  \end{eqnarray*}
  Applying   Lemma 3.1 in \cite{HLN} to the above   system  we obtain
 \begin{eqnarray*}
 \|  {\Phi} \|_{\be} &\leq&  C \exp (C \|y^{n}\|_{\beta}^{1/\beta}  + C\|x\|_\be^{1/\beta} ).
\end{eqnarray*}
  Applying the estimate \eref{eqn6.1} to the right-hand side  of the above inequality, we obtain the upper bound for $
  \|  {\Phi} \|_{\be} $.  The upper bound for $  \|  {\Phi}  \|_{\infty} $ follows from the estimate of $
  \|  {\Phi} \|_{\be} $\,.
 The  upper bounds for ${\Psi} $ can be shown similarly.   \end{proof}

 \medskip
 We also need the following upper bound on  $\wt{\mathcal{E} }^{(N)}_{s,\cdot}(y)$.
\begin{lemma}\label{lem4.2}
Take $  s,s'  \in [0,T] $ such that $ s<s' $ and $y \in \RR^d$.  Assume that $V \in C^{N-1}_{b}$.
Then  for the incomplete Taylor expansion  $\wt{\mathcal{E} }^{(N)}_{s,t}(y)$, $t\in [s,s']$ we have
\begin{eqnarray*}\label{eqn 6.1i}
\|\wt{\mathcal{E} }^{(N)}_{s,\cdot}(y) \|_{s, s', \be}    \leq  K \exp\left( K\sum_{j=1}^{m}\|x^{j}\|_{s,s', \beta_{j}} \right),
\end{eqnarray*}
for some constant $C$ independent of $n$.
\end{lemma}
\begin{proof}   By  Lemma \ref{lem 8.1i}, we have, for any   $\al \in \Gamma_r$
\begin{eqnarray*}
\|x^{\al}_{s,\cdot}\|_{s, s',\beta} &\leq& \prod_{i=1}^{r}\|x^{\al_{i}}\|_{s,s',\beta_{\al_{i}}} \leq   \exp\left( K\sum_{j=1}^{m}\|x^{j}\|_{s,s', \beta_{j}} \right).
\end{eqnarray*}
The  desired estimate follows  immediately by noticing that    $\wt{\mathcal{E} }^{(N)}_{s,t} (y)$  is a  linear combination of  multiple   integrals in $\{x^{\al}_{s,t}: \al \in \wt{\Gamma} \}$.   \end{proof}

  \medskip

For a given finite  subset $\wt{\Ga}$ of $\Ga$, we define
\begin{equation}
\theta=\theta_{\wt{\Ga}}:=\min\left\{ \beta_{\al(1)} + \cdots+ \be_{\al(|\al | ) } -1 \,: \al \in \Ga \setminus  \wt{\Ga}
\right\}\,. \label{e.def-theta}
\end{equation}

\begin{lemma}\label{lem5.1}
 Assume that    $\al $ belongs to $ \Gamma \setminus \wt{\Gamma}  $. Assume that $f$ is a H\"older continuous function of order $\beta$. Then there exists a constant $K$ such that for $t,t' \in [t_{k},t_{k+1}]  $,    $k=0,1,\dots, n-1$, we have
\begin{eqnarray*}
\left| \int_{t}^{t'} f_{u} d x^{\al}_{t_{k},u} \right| &\leq& K  (\|f\|_{\beta}+\|f\|_{\infty}) \exp \left(K \sum_{j=1}^{m}\|x^{j}\|_{ \beta_{j}} \right)n^{ - \theta-1}.
\end{eqnarray*}
\end{lemma}
\begin{proof}
Take $\al \in \Gamma $ such that $ |\al|=r$. Applying Lemma \ref{lem 7.2} to the integral $\int_{t}^{t'} f_{u} d x^{\al}_{t_{k},u}$,
we obtain
\begin{eqnarray*}
\left| \int_{t}^{t'} f_{u} d x^{\al}_{t_{k},u} \right| &\leq& K (\|f\|_{\beta}+\|f\|_{\infty})  \left(\prod_{j=1}^{r} \|x^{\al_{j}}\|_{\beta_{\al_{j}}} \right)n^{-\sum_{j=1}^{r}\beta_{\al_{j}}}
\nonumber
\\
&\leq&
K (\|f\|_{\beta}+\|f\|_{\infty}) \exp
\left( K \sum_{j=1}^{m}\|x^{j}\|_{ \beta_{j}} \right)
 n^{-\sum_{j=1}^{r}\beta_{\al_{j}}} .
 \end{eqnarray*}
Since  $\al \notin \wt{\Gamma}$,   
we see   that $ \sum_{j=1}^{r} \beta_{\al_{j}} \geq \theta+1 $,  proving  the desired   estimate. \end{proof}

In the following,  we   consider the incomplete Taylor scheme   \eref{e3.21}  defined by
 any finite set $\wt{\Gamma}$.

\begin{lemma}\label{lem5.3}
Assume that $\wt{\Gamma}$ has the hierarchical structure  introduced in Definition \ref{hi}. Assume that $f$ is a H\"older continuous function of order $\beta$, and $R_{s,t}^{e} (y)$, $e=1,2,3,4$, are the functions defined   by  \eref{eqn5.15}, \eref{eqn5.15i}, \eref{eqn 5.12} and \eref{eqn5.14i}. Assume that $V \in C^{N+1}_{b}$.  Then there exists a constant $K$ such that for $t,t' \in [t_{k},t_{k+1}] \subset [0,T]$, we have
\begin{eqnarray}\label{eq5.3i}
\left| \int_{t}^{t'} f_{u} d R^{2}_{t_{k},u} (y) \right|  &\leq& C \left( \|f\|_{\beta} + \|f\|_{\infty}  \right) \exp \left( K  \sum_{j=1}^m \|x^{j}\|_{\beta_{j}}   \right) n^{-(N+2)\beta},
\end{eqnarray}
and
\begin{eqnarray}\label{eq5.3}
\left| \int_{t}^{t'} f_{u} d R^{e}_{t_{k},u} (y) \right|  &\leq& C \left( \|f\|_{\beta} + \|f\|_{\infty}  \right)  \exp  \left( K \sum_{j=1}^{m}\|x^{j}\|_{ \beta_{j}}  \right)n^{-\theta-1}, \quad  e=1,2,3,4 .
\end{eqnarray}

\end{lemma}

\begin{proof}
According to Proposition \ref{prop4.7}, for $e=1,3,4$, the integral $ \int_{t}^{t'} f_{u} d R^{e}_{t_{k},u} (y)$ is a linear combination of integrals of the form
\begin{eqnarray*}
\int_{t}^{t'} f_{u} d x^{\al}_{t_{k},u}, \quad \al \in \Gamma\setminus \wt{\Gamma}\,.
\end{eqnarray*}
 So    inequality \eref{eq5.3} for $e= 1,3,4$ follows from   Lemma \ref{lem5.1}.

 Inequality \eref{eq5.3i}   can be shown by applying Lemma \ref{lem 7.2} to the integral
 \begin{eqnarray*}
\int_{t}^{t'} f_{u} d R^{2}_{t_{k},u} (y) &=& \sum_{\zeta\in\Upsilon: |\zeta|=N+1 } \int_{t}^{t'}  \wt{\mathcal{E}}_{ s,u}^{ \zeta }  \left( \partial_{\zeta} V(y+ \wt{\mathcal{E}}^{(N)}_{ s,\cdot}  (y) )
   \right)  (y )
  f_{u  } d x_u
\end{eqnarray*}
  and taking into account the estimates in Lemma \ref{lem4.2}.
Finally,  inequality \eref{eq5.3}  holds for $e=2$ because it is easy to
verify from the definition of $\theta$ that    $(N+1)\beta \geq \theta+1$.
\end{proof}

The following theorem is the    main result in this section.
 \begin{thm}\label{thm4.4}
Let  $\wt{\Gamma} $ be a finite subset of $  \Gamma $ and let $\theta $ be defined by \eref{e.def-theta}. Assume that $\wt{\Gamma}$  has the hierarchical structure  introduced in Definition \ref{hi}. Let $y$   be the solution of equation \eref{eqn5.10}   and  let
 $y^n$ be the solution to  \eref{e3.21}. Assume that $V \in C^{N+1}_{b} $. Then
\begin{eqnarray*}
\sup_{t\in [0,T]} \left| y_{t}-y^n_{t} \right|  & \leq&   G n^{-\theta},
\end{eqnarray*}
where \begin{eqnarray*}
G&=&C  \exp \left(  C \|x\|_\be^{1/\be^{2 } +(N-1)/\beta } + C\sum_{j=1}^{m}\|x^{j}\|_{ \beta_{j}}   \right).
\end{eqnarray*}

\end{thm}
\begin{proof}    Because of   identity  \eref{eq3.24} and  the  estimate of $\|\Phi\|_{\infty}$ in Lemma \ref{lem 6.4},  we only need to show that the quantity
\begin{eqnarray*}
\sum_{e=1}^{4} \sum_{k=0}^{\lfloor \frac{nt}{T}  \rfloor}
\left|
 \int_{t_k  }^{t_{k+1} \wedge t}   {\Psi}_s d     R^e_{t_k, s} (y^n_{t_k})\right|
\end{eqnarray*}
 is bounded by  $G n^{-\theta}$ for $t \in [0,T]$.
Inequality \eref{eq5.3} in Lemma \ref{lem5.3} shows that the above quantity is bounded by
\[
C \sum_{k=0}^{\lfloor \frac{nt}{T} \rfloor}  \left( \|\Psi\|_{\beta} + \|\Psi\|_{\infty}  \right)   \exp \left(K \sum_{j=1}^{m}\|x^{j}\|_{ \beta_{j}}  \right)  n^{-\theta-1},
\]
which is less than 
\begin{eqnarray*}
C  \left( \|\Psi\|_{\beta} + \|\Psi\|_{\infty}  \right)   \exp \left(K \sum_{j=1}^{m}\|x^{j}\|_{ \beta_{j}}  \right)  n^{-\theta}.
\end{eqnarray*}

 Applying the estimates on
  $ \Psi $   given in Lemma \ref{lem 6.4} to the above expression, we obtain the  desired estimate.
\end{proof}

Now we can apply this theorem to obtain the best Taylor scheme. 
It is clear that the possible rates of convergence  are of the form $n^{-\theta}$, where $\theta$ is a nonnegative integer linear combination of $\beta_i$, $i=1,\dots, m$ subtracting one:
\begin{equation}
\theta= \sum_{j=1}^{m}k_j \beta_{j}  -1, \,\,    k_j =0,1,2,\dots \text{ and } j=1,\dots,m  \,.
\label{e.theta-form}
\end{equation}
Given a rate of the above  form we define %
%
%
%
\begin{eqnarray}
      \Gamma(\theta) : = \{ \al \in \Gamma:   \beta_{\al(1)} + \cdots+ \be_{\al(|\al | ) } -1 < \theta \}.
      \label{e.gamma-theta}
 \end{eqnarray}
\begin{lemma}\label{l.4.6}  If $\theta$ has the form \eref{e.theta-form},
then $\theta_{\Ga(\th)}=\theta$, where   $\theta_{\Ga(\th)}$ is defined by \eref{e.def-theta}.
\end{lemma}
\begin{proof}Let  $\theta=\sum_{j=1}^{m}k_j \beta_{j}  -1 $ for some $k_j$, $j=1, \cdots, m$.
Consider
\[
\al=(\overbrace{1,\dots, 1}^{k_1}, \dots, \overbrace{m, \dots, m}^{k_m})\,.
\]
Then $ \beta_{\al(1)} + \cdots+ \be_{\al(|\al | ) }-1=\theta$ and hence $\al\not\in \Ga(\theta)$.
This  shows that $\theta_{\Ga(\th)}\le \th$.  If $\theta_{\Ga(\th)}<\theta$,
then, by the definition of $\theta_{\Ga(\th)}$,  there is an $\al=(\al(1),\dots, \al(r))\in \Ga\setminus \Ga(\theta)$ such that
$\beta_{\al(1)} + \cdots+ \be_{\al(|\al | ) }-1<\theta$.
On the other hand, by our definition of $\Ga(\th)$, $\al\in \Ga(\th)$. This is
a contradiction. Thus $\theta_{\Ga(\th)}=\theta$.
\end{proof}

\begin{remark}\label{remark4.5}
\begin{description}
\item{(i)}\ From Lemma  \ref{l.4.6} and from   Theorem \ref{thm4.4}, we see that a possible rate
has the form $n  ^{-\theta}$, where   $\theta$  is of the form
\eref{e.theta-form}, and for a rate of this form,  the best choice of the incomplete Taylor scheme \eref{e3.21}  is
$\wt{\Ga}=\Ga(\th)$.
 
\item{(ii)}\
  When $\beta_{i} = \beta$, $i=1,\dots, m$ for $\beta >1/2$, 
   $\theta =(N+1)\beta-1 $ and $\Gamma(\theta)$ becomes
  \begin{eqnarray*}
\Gamma(\theta) &=& \{\al \in \Gamma: |\al|\leq N \}.
\end{eqnarray*}
So in this case, the  best Taylor scheme is the complete Taylor scheme:
\begin{eqnarray*}
 y^{n}_{t}  &=&
y^n_{t_k}+ \mathcal{E}_{ t_k,   t }^{(N)}  ( y^n_{t_k   }   ),  \quad y^{n}_{0} = y_{0},
\end{eqnarray*}
for $  t \in [t_k, t_{k+1}]$, $  k=0,1,\dots, n-1$. According to Theorem \ref{thm4.4}, its convergence rate is $n^{1-(N+1)\beta}$,  which coincides with the result obtained in \cite{FV}.
\end{description}
  \end{remark}

  \section{$L_p$-estimates of weighted random sums and multiple integrals}\label{section5}
  In the first subsection, we recall some definitions on  fractional integrals and derivatives
  to fix  the notation we are going to use.
  In the subsequent  three subsections, we derive some $L_p$-estimates of weighted random sums and multiple integrals, which are  needed
  to obtain the rate of convergence for the modified Taylor scheme \eref{eq1.7-modify}.
\subsection{Elements of fractional calculus}
Take $f \in L_{1}  ([0,T])$ and $\delta>0$. The left-sided and right-sided fractional Riemann-Liouville integrals of $f$ of order $\delta$ are defined, for almost all $t \in (a,b)$, by
\begin{eqnarray*}
I^{\delta}_{a+} f(t) &=& \frac{1}{\Gamma(\delta)} \int_{a}^{t} (t-s)^{\delta-1} f(s) ds
\end{eqnarray*}
and
\begin{eqnarray*}
I^{\delta}_{b-} f(t) &=& \frac{1}{\Gamma(\delta)} \int_{t}^{b} (s-t)^{\delta-1} f(s) ds,
\end{eqnarray*}
respectively, where   $\Gamma (\delta) = \int_{0}^{\infty} r^{\delta-1} e^{-r} dr$ is the Gamma function. For $p\geq 1$,  let $I^{\delta}_{a+}(L_{p}  ([0,T]) )$ (respectively  $I^{\delta}_{b-}(L_{p}  ([0,T]) )$) be the class of functions $f$ which may be represented as an $I^{\delta}_{a+}$- ($I^{\delta}_{b-}$-) integral of some $L_{p}$-function $\varphi$. If $f \in I^{\delta}_{a+} (L_{p}  ([0,T])) $ (respectively $f \in I^{\delta}_{b-} (L_{p}  ([0,T])) $) and $0<\delta<1$ then the fractional Weyl derivative is defined as
\begin{eqnarray*}
D^{\delta}_{a+} f(t) &=& \frac{1}{\Gamma(1-\delta)} \left( \frac{f(t)}{(t-a)^{\delta}} + \delta \int_{a}^{t} \frac{f(t) -f(s) }{(t-s)^{\delta+1}} ds \right) \mathbf{1}_{(a,b)}(t)
\end{eqnarray*}
 \begin{eqnarray*}
\bigg(\text{resp. } D^{\delta}_{b-} f(t) &=& \frac{ 1 }{\Gamma(1-\delta)} \left( \frac{f(t)}{(b-t)^{\delta}} + \delta \int_{t}^{b} \frac{f(t) -f(s) }{( s-t)^{\delta+1}} ds \right) \mathbf{1}_{(a,b)}(t)
 \bigg),
\end{eqnarray*}
where $a<t<b$.

\subsection{$L_{p}$-estimate of weighted random sums}
 Let $\zeta= \{ \zeta_{k, n} , n \in \mathbb{N}, k=0, 1, \dots, n\}$ be a double sequence of random variables.  The aim of this subsection is to    provide an $L_p $-estimate of the weighted summations of this sequence, which we need for the rate of convergence of the modified Taylor scheme.
 We first introduce the space of H\"older continuous functions in $L_{p}$.
\begin{Def}\label{def5.1}
Let $f$ be a stochastic process   on $[0,T]$ such that $f(t) \in L_p$ for each $t$. We say that $f$ is \emph{H\"older  continuous of order $\beta$ in $L_p$} if
\begin{eqnarray*}
\| f(t) - f(s) \|_{p} &\leq& K |t-s|^{\beta}, \quad s,t \in [0,T]
\end{eqnarray*}
for  $\beta>0$. We define the seminorm
\begin{eqnarray*}
\|f\|_{\beta, p} &=& \sup \left\{ \frac{\| f(t) - f(s) \|_{p}   }{(t-s)^{\beta} } :\, 0\leq s< t\leq T \right\}.
\end{eqnarray*}
\end{Def}
In the following, we denote   $t_{k}=kT/n$, $k=0,1,\dots, n$ and $\eta(t) = t_{k}$ for $t \in [t_{k},t_{k+1})$.
   \begin{prop}\label{lem 2.4}
Let $p \geq 1$, $q,q' >p$ such that $  \frac 1p= \frac{1}{q} + \frac{1}{q'} $ and let $\beta,  \beta'$ be in $
 (0,1)$ such that $ \beta+\beta'>1 $.
  Let $\zeta= \{ \zeta_{k, n} , n \in \mathbb{N}, k=0, 1, \dots, n\}$
  satisfy
\begin{eqnarray}\label{eq 2.7}
 \mE \left(\left|  \sum_{k=j+1}^{ i } \zeta_{k , n}   \right|^{q} \right) &\leq&
  L \left( \frac{i-j}{n}\right)^{ \beta' q}\,
    \end{eqnarray}
for all $ i,j =0, 1, \dots, n $, $i>j $  and  for some constant $L>0$.
Let $f$ be a continuous process and   assume that $f$ is   H\"older continuous of order $\beta$ in $L_{q'}$.
   Then for $i,j=0, 1,\dots, n-1$, $ i>j$,
\begin{eqnarray*}
\left\| \sum_{k= j +1   }^{ i  } f(t_k) \zeta_{k, n} \right\|_p &\leq&  cL   \|f\|_{\beta, q'} \left(\frac{i-j}{n}\right)^{\beta+\beta'} + cL \|f(t_{j})\|_{q'} \left(\frac{i-j}{n}\right)^{\beta'} ,
\end{eqnarray*}
where $c$ is a constant depending on $T$ and the parameters $p, q, q', \beta, \beta'$.
\end{prop}

\begin{proof}
For each $t \in [0, T]$ we denote
\begin{eqnarray*}
g_n(t): = \sum_{k=0}^{   \lfloor \frac {nt}T  \rfloor } \zeta_{k , n} \,.
\end{eqnarray*}
Then we can  write
\begin{eqnarray*}
\sum_{k= j +1   }^{ i  } f(t_k) \zeta_{k, n} &=&
\int_{ (t_{j},t_{i+1 }) } f(t) d g_n(t).
\end{eqnarray*}
We shall use    the following  fractional integration by parts formula to deal with the above integral
(see   Theorem 2.4 in \cite{Zahle}).
 \begin{equation}\label{e5.3i}
\int_{(a,b)}  fdg_{n}=  \int_{a}^{b} D^{\delta}_{a+} f_{a} (t) D^{1-\delta}_{b-} g_{n,b}(t) dt
 +f(a)(g_{n}(b-)-g_{n}(a+)),
\end{equation}
where   we denote $f_{a}(t) = \mathbf{1}_{(a,b)}(t) (f(t)-f(a) )$, $g_{n,b}(t) = \mathbf{1}_{(a,b)}(t) (g_{n}(t) -g_{n}(b-))$ and  $\delta \in [0,1]$.

 We denote $a:= t_j$ and $b:=t_{i+1}$.
Let $\delta$ be such that $1-\beta' <\delta<\beta$. By the definition of the fractional derivative it is easy to show that
\begin{eqnarray} \label{eqn 2.5i}
\left\| D^{\delta}_{a+} f_{a+}(t) \right\|_{q'}
&\leq&
   \frac{\|f\|_{  \beta, q'} }{\Gamma (1-\delta)}\frac{\beta}{\beta-\delta} (t-a)^{\beta-\delta}.
 \end{eqnarray}
On the other hand,
\begin{eqnarray}\label{eqn 2.7i}
\left\| D^{1-\delta}_{b-} g_{n,b}(t) \right\|_q
&\leq &
 \frac{1 }{\Gamma ( \delta)} \left( \left\|  \frac{ g_n(t)  -g_n(b-)  }{(b-t)^{1-\delta}}  \right\|_q + (1-\delta)  \left\|  \int^b_t \frac{g_n(t) - g_n(s)}{(s-t)^{2-\delta }} ds \right\|_q \right)\,.
   \end{eqnarray}
   We first consider the second term on the right-hand side of the above inequality. When  $t \geq b-\frac Tn$, we have $g_n(t) - g_n(s)=0$ and thus the second term is equal to zero. When $t< b-\frac Tn$,  we have
   \begin{eqnarray}\label{eqn 2.6}
   &&
    \left\|
    \int^b_t \frac{g_n(t) - g_n(s)}{(s-t)^{2-\delta }} ds
   \right\|_q
   \nonumber
   \\
   && =
    \left\|
    \int^b_{\eta(t) + \frac Tn} \frac{g_n(t) - g_n(s)}{(s-t)^{2-\delta }} ds
   \right\|_q
\leq
    L\int^b_{\eta(t) + \frac Tn} \frac{ [ \eta(s) - \eta(t)]^{\beta'}}{(s-t)^{2-\delta }} ds
    \nonumber
    \\
    &&\leq
    2^{\be'} L \int^b_{\eta(t) + 2 \frac Tn}  (s-t)^{\beta' + \delta -2 }  ds
    +
    L\int^{ \eta(t) + 2 \frac Tn }_{\eta(t) + \frac Tn} \frac{ [ \eta(s) - \eta(t)]^{\beta'}}{(s-t)^{2-\delta }} ds
   \nonumber
    \\
   && \leq
   \frac{ 2^{\be'} L}{   \beta'+\delta-1 } (b-t)^{ \beta'+\delta-1 }
    +
    \frac{L T^{\beta' } }{n^{\beta' } (\delta-1) } [ ( \eta(t) + 2 \frac Tn -t )^{\delta-1} - ( \eta(t) + \frac Tn -t )^{\delta-1} ]
   \nonumber
    \\
   && \leq
   \frac{ 2^{\be'} L}{   \beta'+\delta-1 } (b-t)^{ \beta'+\delta-1 }
    +
    \frac{L T^{\beta' } }{n^{\beta' } ( 1 - \delta ) } ( \eta(t) + \frac Tn -t )^{\delta-1}
    ,
   \end{eqnarray}
where in the first inequality we used the assumption \eref{eq 2.7}.
For the first term in \eref{eqn 2.7i}, since
\begin{eqnarray*}
\|g_n(t)  -g_n(b-)\|_q \leq L | b-\frac{T}{n} - \eta(t) |^{\beta'}
\leq
L|b-t|^{\beta'}
\end{eqnarray*}
 for $t \in (a, b)$,
we have
\begin{eqnarray*}
\left\|  \frac{ g_n(t)  -g_n(b-)  }{(b-t)^{1-\delta}}  \right\|_q
\leq
L
(b-t)^{\beta'+\delta-1}.
\end{eqnarray*}
Substituting  the above inequality and \eref{eqn 2.6}  into  \eref{eqn 2.7i} we obtain
\begin{eqnarray}\label{eqn 2.8i}
\left\| D^{1-\delta}_{b-} g_{n,b}(t) \right\|_q
&\leq &
  cL  (b-t)^{\beta'+\delta-1}
  +cL  n^{-\beta'} ( \eta(t) + \frac Tn -t )^{\delta-1}  .
   \end{eqnarray}
   By the fractional integration by parts formula \eref{e5.3i} and by \eref{eqn 2.5i} and \eref{eqn 2.8i} we obtain
   \begin{eqnarray*}
\left\|
\int^b_a f (t) dg_{n} (t)
\right\|_p
 &\leq&
      cL \|f\|_{\beta, q'}
      \int^b_a
      (t-a)^{\beta-\delta}
      \left[
       (b-t)^{\beta'+\delta-1}
  +  n^{-\beta'} ( \eta(t) + \frac Tn -t )^{\delta-1}
      \right]
      dt
      \\
&&\quad\quad   + \|f(a)\|_{q'} (b-a)^{\beta'}
  \\
 & \leq&
   cL \|f\|_{\beta,q'} (b-a)^{\beta+\beta'}
   +
   cL \|f\|_{\beta,q'}
    \int^b_a
      (t-a)^{\beta-\delta}
         n^{-\beta'} ( \eta(t) + \frac Tn -t )^{\delta-1}
        dt
         \\
&&\quad\quad
 +
        \|f(a)\|_{q'} (b-a)^{\beta'}
  .
\end{eqnarray*}
The lemma then follows from the following calculation,
 \begin{eqnarray*}
 &&
 \int^b_a
      (t-a)^{\beta-\delta}
         n^{-\beta'} ( \eta(t) + \frac Tn -t )^{\delta-1}
        dt
       \\
     &&  \leq
       n^{-\beta'}
       \sum_{k=j}^{i -1  }
       (t_{k+1}-a)^{\beta-\delta}
         \int^{ t_{k+1} }_{ t_k}
     ( \eta(t) + \frac Tn -t )^{\delta-1}
        dt
        \\
    &&    \leq
       n^{-\beta'}
       \sum_{k=0}^{i-j }
       \left(\frac{k+1}{n}\right)^{\beta-\delta}
      \frac{1}{ \delta} \left(  \frac Tn \right)^{\delta}
 =
        cn^{-\beta-\beta'} \sum_{k=0}^{i-j} (k+1)^{\beta-\delta}
        \\
      &&  \leq
         cn^{-\beta-\beta'} (i-j)^{1+\beta-\delta}
\leq
        c\left( \frac{i-j}{n} \right)^{\beta+\beta'}.
       \end{eqnarray*}
\hfill$\square$

\subsection{ Monotonicity in the $L_{2}$-norm of multiple   integrals } \label{section5.3}
In this subsection we derive a  monotonicity result on the $L_{2}$-norm of multiple   integrals with  respect to the fractional Brownian motion. 
We recall that a standard one-dimensional fractional Brownian motion (fBm) is a centered Gaussian process $B=\{B_t, t\ge 0\}$ with covariance given by
\[
\mE[ B_tB_s ] =\frac 12 \left( t^{2H} + s^{2H} -|t-s|^{2H} \right),
\]
where $H\in (0,1)$ is the Hurst parameter.
Our proof is based on the approximation of   multiple   integrals by   sums of products of fBm increments. Throughout this subsection,  we
assume that, for $j=1,\dots, N$,   $B^{j}$ is either a  fBm with Hurst parameter $H_j>1/2$ or the identity function.  Assume in addition that for $j,j'=1,\dots, N$,  the two processes  $B^j$ and $B^{j'}$ are either   mutually independent or equal.

 We first recall  a formula for the expectation of a   product of increments.
Take an even number $r$. There are $( r -1 ) !!$ ways to arrange the elements of $\{1,\dots, r\}$ into pairs. We denote by $ \mathcal{R}_{r} $ the collection of these ways.  Assume that each $\tau \in \mathcal{R}_{r} $ is of the form $\tau = \{(\tau_{1}, \tau_{2}),\dots, (\tau_{r-1}, \tau_{r})\}$, where $\tau_{i} \in \{1,\dots, r\}$.
    For a subinterval $I=(s,t)$ of $[0,T]$, we denote $B_{I}^j = B^j_{s,t}=B^j_{t}-B^j_{s}$. The following is a consequence
    of the Feynman diagram formula (see \cite[page 16]{janson}).

 \begin{lemma}\label{lem5.2}
 Let $I_{i}$, $i=1,\dots, r$, be subintervals of $[0,T]$.

\noindent \emph{(i)} If  $B^{j }$, $j=1,\dots, r$ are fBms, then the following identity holds true,
\begin{eqnarray*}
\mE[B_{I_{1}}^{ {1}} \cdots B_{I_{r}}^{ {r}}] & =& \begin{cases} \sum\limits_{\tau \in \mathcal{R}_{r}}  \mE[B^{ {\tau_{1}}}_{I_{\tau_{1}}} B^{ {\tau_{2}}}_{I_{\tau_{2}}}  ]\cdots \mE[ B^{ {\tau_{r-1}}}_{I_{\tau_{r-1}}}  B^{ {\tau_{r}}}_{I_{\tau_{r}}} ],  \quad& r \text{ is even}.
\\
0, & r \text{ is odd}.
\end{cases}
\end{eqnarray*}

\noindent \emph{(ii)} The following inequality holds true,
 \begin{eqnarray*}
\mE[B_{I_{1}}^{ {1}} \cdots B_{I_{r}}^{ {r}}] & \geq&0  .
\end{eqnarray*}
  \end{lemma}
  \begin{proof}  The first result  is   the Feynman diagram  formula for products of
  Gaussian random variables. The second result follows   from (i)  and   the fact that the increments of a fBm with Hurst parameter $H>1/2$ have positive correlation.
\end{proof}

  \smallskip

  We recall  an $L_{p}$-convergence result in Proposition 2.2    \cite{HLN}.
  \begin{prop}\label{prop5.4}
 Assume that $f $ and $g $ are stochastic processes which are H\"older continuous  of orders $\mu$ and $\lambda$ in $L_{p}$ (see Definition \ref{def5.1}) for any $p \geq 1$, respectively,  such that $\lambda +\mu >1$. Assume   that $f_{0} \in L_{p}$. Then, the  integral 
$ \int_0^T fdg$ exists  as a Riemann-Stieltjes integral of $L_p$-valued functions, and, as a consequence,  we have the 
  following convergence  in $L_p$:
  \begin{eqnarray*}
\lim_{n \rightarrow \infty} \sum_{k=0}^{n-1} f_{t_{k}} g_{t_{k},t_{k+1}}
= \int_{0}^{T} f dg\,, \quad \hbox{
where $t_{k}= kT/n$, $k=0,1,\dots, n$}\,.
\end{eqnarray*}
\end{prop}

We are ready to prove the  main result of this subsection. We will make use of the notation
\begin{eqnarray*}
J_{r} (\mathcal{A}) &:=& \int_{0}^{T}  \cdots \int_{0}^{T} \mathbf{1}_{\mathcal{A}}   d B^{ {1}}_{s_{1}} \cdots dB^{ {r}}_{s_{r}}
\end{eqnarray*}
 for any Borel subset $\mathcal{A} \subset [0,T]^{r}$ such that the above multiple integral exists as an iterated  Riemann-Stieltjes integral  defined using  $L_p$-convergence. For any $0\le s<t \le T$,  we define
 \[
 [s,t]^r_<=\{ (s_{1},\dots, s_{r}) \in [0,T]^r:  s \leq s_{1} \leq \cdots \leq s_{r} \leq t\}.
 \]

\begin{prop}\label{prop5.5}
Let   $ 0=t_{0} <t_{1}<\cdots < t_{n-1} < t_{n} =T$ be a partition of $[0,T]$.
Take
 \begin{equation}\label{e5.8i}
\mathcal{A}  =  \bigcup_{k=0}^{n-1} [t_{k},t_{k+1}]^{r}_{<}\,,
\end{equation}
and
 \begin{eqnarray*}
\mathcal{A}' = \bigcup_{k=0}^{n-1} [t_{k},t_{k+1}]^{r}.
\end{eqnarray*}
Then  $J_{r} (\mathcal{A})$ and $J_{r} (\mathcal{A}')$ are well defined  as iterated  Riemann-Stieltjes integrals, and we have
 \begin{eqnarray}\label{e5.10}
 \mE \left(\left|   J_{r} (\mathcal{A}) \right|^{2} \right)
 &\leq&
  \mE \left(\left|  J_{r} (\mathcal{A}' ) \right|^{2} \right) .
\end{eqnarray}
\end{prop}

  \begin{proof}
  We denote  $t^{l}_{k}=Tk/l$ and $I^{l}_{k} = [ t^{l}_{k},  t^{l}_{k+1}]$ for $0\leq k\leq l$. By Proposition \ref{prop5.4},  it is easy to see that
  \begin{eqnarray*}
 \mE \left(\left|   J_{r} (\mathcal{A}) \right|^{2} \right) &=&
\lim_{n_{1}\rightarrow \infty} \mE  \left(\left|  \sum_{k=0}^{n_{1}-1} B^{ {r}}_{I^{n_{1}}_{k}} \int_{0}^{T} \cdots \int_{0}^{T} \mathbf{1}_{\mathcal{A} }  (s_{1},\dots,s_{r-1}, t^{n_{1}}_{k} )   d B^{ {1}}_{s_{1}} \cdots dB^{ {r-1}}_{s_{r-1}}\right|^{2} \right)
\end{eqnarray*}
  Applying Proposition \ref{prop5.4} several times we obtain
  \begin{eqnarray}
 \lim_{n_{r}\rightarrow \infty}\cdots \lim_{n_{1}\rightarrow \infty} \mE  \left(\left|  \sum_{k_{r}=0}^{n_{r}-1} \cdots \sum_{k_{1}=0}^{n_{1}-1}  B^{ {r}}_{I^{n_{r}}_{k_{r}}} \cdots B^{ {1}}_{I^{n_{1}}_{k_{1}}}      \mathbf{1}_{\mathcal{A} }  ( t_{k_{1}}^{n_{1}},\dots,  t_{k_{r} }^{n_{r}} )   \right|^{2} \right)
 &=& \mE \left(\left|   J_{r} (\mathcal{A}) \right|^{2} \right).
 \label{e5.9}
\end{eqnarray}
It is clear that this identity still holds when we replace $\mathcal{A}$ by $\mathcal{A}'$.
On the other hand, since $\mathcal{A} \subset\mathcal{A}'$, it follows from Lemma \ref{lem5.2}(ii) that
\begin{eqnarray*}
&&
\mE  \left(\left|  \sum_{k_{r}=0}^{n_{r}-1} \cdots \sum_{k_{1}=0}^{n_{1}-1}  B^{ {r}}_{I^{n_{r}}_{k_{r}}} \cdots B^{ {1}}_{I^{n_{1}}_{k_{1}}}      \mathbf{1}_{\mathcal{A} }  ( t_{k_{1}}^{n_{1}},\dots,  t_{k_{r} }^{n_{r}} )    \right|^{2} \right)
\\
&&\leq
\mE  \left(\left|  \sum_{k_{r}=0}^{n_{r}-1} \cdots \sum_{k_{1}=0}^{n_{1}-1}  B^{ {r}}_{I^{n_{r}}_{k_{r}}} \cdots B^{ {1}}_{I^{n_{1}}_{k_{1}}}      \mathbf{1}_{\mathcal{A}' }  ( t_{k_{1}}^{n_{1}},\dots,  t_{k_{r} }^{n_{r}} )    \right|^{2} \right).
\end{eqnarray*}
By taking limits in both sides of the above inequality and taking into account
   \eref{e5.9} we obtain   inequality \eref{e5.10}.
  \end{proof}

\begin{remark}
The monotonicity property in Proposition \ref{prop5.5} can be generalized to any $\mathcal{A},\mathcal{A}' \subset [0,T]^{r} $ such that $  \mathcal{A} \subset \mathcal{A}' $ as long as the multiple integrals $J_{r}(\mathcal{A})$ and  $J_{r}(\mathcal{A}')$ are well defined as iterated   Riemann-Stieltjes integrals. The same generalization holds true for the   monotonicity property established in Proposition \ref{prop5.8} below.
\end{remark}

\begin{remark}\label{remark5.7}
In the same way as in the proof of Proposition \ref{prop5.5}, we can show that
\begin{eqnarray}
 \mE \left(   J_{r} (\mathcal{A}) \right)&=&
\lim_{n_{r}\rightarrow \infty}\cdots \lim_{n_{1}\rightarrow \infty} \mE  \left(   \sum_{k_{r}=0}^{n_{r}-1} \cdots \sum_{k_{1}=0}^{n_{1}-1}  B^{ {r}}_{I^{n_{r}}_{k_{r}}} \cdots B^{ {1}}_{I^{n_{1}}_{k_{1}}}      \mathbf{1}_{\mathcal{A} }   ( t_{k_{1}}^{n_{1}},\dots,  t_{k_{r} }^{n_{r}} )  \right).
\label{e5.11ii}
\end{eqnarray}
So the expectation of the  multiple integral $     J_{r} (\mathcal{A})  $ is always zero when the number of fBms in $\{B^1,\dots, B^r\}$ is odd.
\end{remark}

\begin{lemma}\label{lem5.7}
 Let $I_{i}$, $i=1,\dots, 2r$, be subintervals of $[0,T]$, where $r$ is an even number.    If $B^j$, $j=1,\dots, 2r$ are fBms, then the following identity holds true,
\begin{eqnarray*}
\mathrm{Cov} ( B^{1}_{I_{1}}\cdots B^{r}_{I_{r}},  B^{r+1}_{I_{r+1}}\cdots B^{2r}_{I_{2r}} )
&=& \sum_{\tau \in \mathcal{R}'_{2r}   }   \mE[B^{ {\tau_{1}}}_{I_{\tau_{1}}} B^{ {\tau_{2}}}_{I_{\tau_{2}}}  ]\cdots \mE[ B^{ {\tau_{2r-1}}}_{I_{\tau_{2r-1}}}  B^{ {\tau_{2r}}}_{I_{\tau_{2r}}} ],
\end{eqnarray*}
where $ \mathcal{R}'_{2r}   $ is a subset of $  \mathcal{R}_{2r}  $ such that   for  $\tau \in \mathcal{R}_{2r} \setminus  \mathcal{R}'_{2r}   $,  either $\tau_{i}, \tau_{i+1} \in \{1,\dots, r\}$ or $\tau_{i}, \tau_{i+1} \in \{r+1,\dots, 2r\}$, $i=1,3,\dots, 2r-1$. In particular, the covariance of $B^{1}_{I_{1}}\cdots B^{r}_{I_{r}}$ and $ B^{r+1}_{I_{r+1}}\cdots B^{2r}_{I_{2r}}$ is nonnegative.
\end{lemma}
\begin{proof}  Note that
\begin{eqnarray*}
\mathrm{Cov}( B^{1}_{I_{1}}\cdots B^{r}_{I_{r}},  B^{r+1}_{I_{r+1}}\cdots B^{2r}_{I_{2r}} )&=&\mE[B^{1}_{I_{1}}\cdots B^{2r}_{I_{2r}} ] - \mE[B^{1}_{I_{1}}\cdots B^{r}_{I_{r}} ] \mE[B^{r+1}_{I_{r+1}}\cdots B^{2r}_{I_{2r}} ] .
\end{eqnarray*}
  The lemma then follows immediately from Lemma \ref{lem5.2} (i). \end{proof}

Recall that we define the centered   multiple integral as
 \begin{eqnarray*}
\wt{J}_{r} ( \mathcal{A} ) &:=&  {J}_{r} ( \mathcal{A} )  - \mE \left[ {J}_{r} ( \mathcal{A} )   \right].
\end{eqnarray*}
Following is the monotonicity result on the $L_{2}$-norm of this multiple integral.
\begin{prop} \label{prop5.8}
Let $\mathcal{A}$ and $\mathcal{A}'$ be   as in Proposition \ref{prop5.5}. Then   we have
 \begin{eqnarray} \label{e5.13iiii}
 \mE \left(\left|   \wt{J}_{r} (\mathcal{A}) \right|^{2} \right)
 &\leq&
  \mE \left(\left|  \wt{J}_{r} (\mathcal{A}' ) \right|^{2} \right) .
\end{eqnarray}
\end{prop}
\begin{proof}  We first notice that
\begin{eqnarray*}
 \mE \left(\left|   \wt{J}_{r} (\mathcal{A}) \right|^{2} \right)  &=&   \mE \left(\left|   {J}_{r} (\mathcal{A}) \right|^{2} \right)  -  \mE \left(      {J}_{r} (\mathcal{A})  \right)^{2} .
\end{eqnarray*}
By applying \eref{e5.9} and \eref{e5.11ii} to the above equation, we have
\begin{eqnarray}
 \mE \left(\left|   \wt{J}_{r} (\mathcal{A}) \right|^{2} \right)
 &=&
\lim_{n_{r}\rightarrow \infty}\cdots \lim_{n_{1}\rightarrow \infty} \biggl\{
\mE  \left(  \left| \sum_{k_{r}=0}^{n_{r}-1} \cdots \sum_{k_{1}=0}^{n_{1}-1}  B^{ {r}}_{I^{n_{r}}_{k_{r}}} \cdots B^{ {1}}_{I^{n_{1}}_{k_{1}}}      \mathbf{1}_{\mathcal{A} }  ( t_{k_{1}},\dots,  t_{k_{r} } ) \right|^{2}   \right)
\nonumber
\\
&&  \quad\quad-
\mE  \left(   \sum_{k_{r}=0}^{n_{r}-1} \cdots \sum_{k_{1}=0}^{n_{1}-1}  B^{ {r}}_{I^{n_{r}}_{k_{r}}} \cdots B^{ {1}}_{I^{n_{1}}_{k_{1}}}      \mathbf{1}_{\mathcal{A} }  ( t_{k_{1}},\dots,  t_{k_{r} } )    \right)^{2
}
\biggr\}
\nonumber
\\
&=& \lim_{n_{r}\rightarrow \infty}\cdots \lim_{n_{1}\rightarrow \infty}  \sum_{k_{r}, k_{r}'=0}^{n_{r}-1} \cdots \sum_{k_{1},k_{1}'=0}^{n_{1}-1}
  \mathrm{Cov}[
 B^{ {r}}_{I^{n_{r}}_{k_{r}}} \cdots B^{ {1}}_{I^{n_{1}}_{k_{1}}}   , B^{ {r}}_{I^{n_{r}}_{k_{r}'}} \cdots B^{ {1}}_{I^{n_{1}}_{k_{1}'}}
 ]
 \nonumber
  \\
 && \quad\quad \times   \mathbf{1}_{\mathcal{A} }  ( t_{k_{1}},\dots,  t_{k_{r} } )   \mathbf{1}_{\mathcal{A} }  ( t_{k_{1}'},\dots,  t_{k_{r} '} ) .
     \label{e5.13iii}
\end{eqnarray}
Since $\mathcal{A}  \subset \mathcal{A}' $ and $\mathrm{Cov}[
 B^{ {r}}_{I^{n_{r}}_{k_{r}}} \cdots B^{ {1}}_{I^{n_{1}}_{k_{1}}}   , B^{ {r}}_{I^{n_{r}}_{k_{r}'}} \cdots B^{ {1}}_{I^{n_{1}}_{k_{1}'}}
 ]
\geq 0$ by Lemma \ref{lem5.7}, we have
\begin{eqnarray*}
 &&
 \mathrm{Cov}[
 B^{ {r}}_{I^{n_{r}}_{k_{r}}} \cdots B^{ {1}}_{I^{n_{1}}_{k_{1}}}   , B^{ {r}}_{I^{n_{r}}_{k_{r}'}} \cdots B^{ {1}}_{I^{n_{1}}_{k_{1}'}}
 ]
     \mathbf{1}_{\mathcal{A} }  ( t_{k_{1}},\dots,  t_{k_{r} } )   \mathbf{1}_{\mathcal{A} }  ( t_{k_{1}'},\dots,  t_{k_{r} '} )
   \\
   &  &\leq
      \mathrm{Cov}[
 B^{ {r}}_{I^{n_{r}}_{k_{r}}} \cdots B^{ {1}}_{I^{n_{1}}_{k_{1}}}   , B^{ {r}}_{I^{n_{r}}_{k_{r}'}} \cdots B^{ {1}}_{I^{n_{1}}_{k_{1}'}}
 ]
     \mathbf{1}_{\mathcal{A}' }  ( t_{k_{1}},\dots,  t_{k_{r} } )   \mathbf{1}_{\mathcal{A}' }  ( t_{k_{1}'},\dots,  t_{k_{r} '} ) .
\end{eqnarray*}
By summing  over $k_{1},k_{1}',\dots, k_{r},k_{r}'$ and taking   limits on both sides of the above inequality, and taking into account the identity \eref{e5.13iii}, we obtain the inequality \eref{e5.13iiii}. \end{proof}

\subsection{$L_{p}$-estimates of    multiple   integrals}\label{section5.4}
  In this subsection, we assume that $B^1$ is the identity function and   denote $H_1 =1$, and  $B^{j}$  is   a    fBm of Hurst parameter $H_j>1/2$, $j=2,\dots, m$.  We assume in addition that $B^2$, $\dots, $ $B^m$ are mutually independent.   For $\al = (\al_1, \dots, \al_r) \in \Gamma_r$ (the collection of multi-indices of length $r$ with elements in $\{1,\dots, m\}$),  and $\mathcal{A} \subset [0,T]^r$, we write
\begin{eqnarray*}
J_{r}^{\al} (\mathcal{A}) : = \int_{0}^{T} \cdots \int_{0}^{T} \mathbf{1}_{\mathcal{A} } (s_{1},\dots, s_{r})  d B^{\al_{1}}_{s_{1}} \cdots dB^{\al_{r}}_{s_{r}},
\end{eqnarray*}
provided that this multiple   integral exists, as an integrated Riemann-Stieltjes integral in $L_p$, for all $p\ge 1$.
Recall that   $D= \{t_{k} = kT/n, k=0,1,\dots, n\}$.
\begin{prop}\label{prop 2}
Denote   \emph{(}see \emph{\eref{e5.8i})}  $ \mathcal{A}_{k} = [t_{k},t_{k+1}]^{r}_{<}$, $k=0,1,\dots, n-1$. Take $s,t \in D$, $\al\in \Gamma_r$, and denote $r' = \#\{i:  { {\alpha_i}}\neq 1 \}$.
Then we have
 \begin{eqnarray}\label{e5.11i}
  \left\|
\sum_{k =  ns/T  }^{   nt/T -1  } J_{r}^{\al} ( \mathcal{A}_{k} )
\right\|_{ 2}
&\leq&
 C\nu_{n}
(\al)(t-s)^{1/2},
 \end{eqnarray}
 where $C$ is a constant that depends on $T$, $m$ and $\alpha$, and 
\begin{eqnarray}\label{e5.14i}
  \nu_{n}(\al)&=&
\begin{cases}
       n^{ 1  - H_\al     }
  &\qquad \hbox{if}\quad r' \text{  is even}\, ,\\
  n^{     H - H_\al      }
   &\qquad \hbox{if}\quad  r'  \text{  is odd}\, .\\
\end{cases}
\end{eqnarray}
Here  $H_\al := H_{\al_1} + \dots+H_{\al_r}$ and $H=\max\limits_{i:\al_i \neq 1} H_{\al_i}$.
\end{prop}
\begin{proof}
According to Proposition \ref{prop5.5}, it suffices to show that the $L_{p}$-estimate holds for $\mathcal{A}_{k} = [t_{k},t_{k+1}]^{r}$.
In this case,
we have
\begin{eqnarray}\label{e5.11}
J^{\al}_{r} (\mathcal{A}_{k}) &=& \prod_{i=1}^{r} B^{\al_{i}}_{t_{k},t_{k+1}}
=
  \prod_{j=1}^{m}( B^{j}_{t_{k},t_{k+1}}  )^{r_{j}},
 \end{eqnarray}
where $r_{j} =\#\{i: \al_{i} = j \}$, $j=1,\dots, m$ and $r_{1} +\cdots+r_{m} = r$.

\medskip

We first consider the case when $r_{1}=0$.
Denote by $\mu_{q}$   the $q$th moment of  a standard Gaussian random variable  $G \sim \mathcal{N} (0,1)$, that is, $\mu_{q} = (q-1)!!$ when $q$ is even and $\mu_{q} = 0$ when $q$ is odd.
It is well-known that we have the following Hermite decomposition:
\begin{eqnarray*}
x^{q} &=&  \sum_{p=0}^{q} \frac{q!}{(q-p)!} \mu_{q-p} H_{p}(x) , \quad x\in \RR,
\end{eqnarray*}
where
\begin{eqnarray*}
H_{q}(x) &=& \frac{(-1)^{q}}{q!} e^{\frac{x^{2}}{2}} \frac{d^{q}}{dx^{q}} \left( e^{-\frac{x^{2}}{2}}  \right).
\end{eqnarray*}
If
we denote $G^{j}_{k} = (\frac nT)^{H_j} B^{j}_{t_{k},t_{k+1}}$, then,  applying the above decomposition to \eref{e5.11},  we have
\begin{eqnarray}
J^{\al}_{r} (\mathcal{A}_{k})  &=& (\frac nT)^{-\sum_{j=2}^m r_jH_j} \prod_{j=2}^{m} ( G^{j}_{k})^{r_{j}}
\nonumber
 \\
 &=&  (\frac nT)^{-\sum_{j=2}^m r_jH_j} \prod_{j=2}^{m}   \sum_{p_j=0}^{r_{j}} \frac{r_{j}!}{(r_{j}-p_j)!} \mu_{r_{j}-p_j} H_{p_j}(G^{j}_{k}).
 \label{e5.13ii}
\end{eqnarray}
Therefore, for any $0\le k_ 1 \le k_2 \le n-1$,
\begin{eqnarray}
\left\| \sum_{k=k_{1}}^{k_{2}} J^{\al}_{r} (\mathcal{A}_{k}) \right\|_{2}^{2}
 &=& \left\| \sum_{k=k_{1}}^{k_{2}}  (\frac nT)^{-\sum_{j=2}^m r_jH_j} \prod_{j=2}^{m}   \sum_{p_{j}=0}^{r_{j}} \frac{r_{j}!}{(r_{j}-p_{j})!} \mu_{r_{j}-p_{j}} H_{p_{j}}(G^{j}_{k}) \right\|_{2}^{2}
 \nonumber
 \\
 &\leq& C n^{-2\sum_{j=2}^m r_jH_j} \sum_{p_{2}=0}^{r_{1}} \cdots \sum_{p_{m}=0}^{r_{m}}  \left\| \sum_{k=k_{1}}^{k_{2}}    \prod_{j=2}^{m}  \mu_{r_{j}-p_{j}}    H_{p_{j}}(G^{j}_{k}) \right\|_{2}^{2}  .
 \label{e5.13}
 \end{eqnarray}
 Expanding the right-hand side of the above inequality, we have
 \begin{eqnarray}
&&
    \left\| \sum_{k=k_{1}}^{k_{2}}    \prod_{j=2}^{m}  \mu_{r_{j}-p_{j}}    H_{p_{j}}(G^{j}_{k}) \right\|_{2}^{2}
\nonumber
\\
&&=   \sum_{k,k'=k_{1}}^{k_{2}}  \mE\left[      \prod_{j=2}^{m}  \mu_{r_{j}-p_{j}}    H_{p_{j}}(G^{j}_{k})      \prod_{j'=2}^{m} \mu_{r_{j'}-p_{j'}}     H_{p_{j'}}(G^{j'}_{k'}) \right]
\nonumber
\\
&&=
 \sum_{k,k'=k_{1}}^{k_{2}}   \prod_{j=2}^{m}  \mu_{r_{j}-p_{j}}^{2}   \mE\left[     H_{p_{j}}(G^{j}_{k})    H_{p_{j}}(G^{j}_{k'})   \right]
\nonumber
\\
&&=
 \left( \sum_{ | k-k'|\leq 2}  + 2 \sum_{k > k' +2 }   \right)  \prod_{j=2}^{m}  \mu_{r_{j}-p_{j}}^{2}   \mE\left[     H_{p_{j}}(G^{j}_{k})    H_{p_{j}}(G^{j}_{k'})   \right]
 \nonumber
\\
&&: = E_{1} +E_{2}.
\label{e5.15ii}
\end{eqnarray}
We take $k_{1} =   \frac{ns}{T}  $ and $k_{2} =   \frac{nt}{T}   -1$. Since
\begin{eqnarray}\label{e5.17}
 \mE\left[     H_{p_{j}}(G^{j}_{k})    H_{p_{j}}(G^{j}_{k'})   \right] &=& \left( \mE[ G^{j}_{k}  G^{j}_{k'} ]\right)^{p_{j}}
 \nonumber
 \\
 &=&  \left( \al_{H_j}\int_{k}^{k+1} \int_{k'}^{k'+1} |u-v|^{2H_j-2} dudv\right)^{p_{j}},
\end{eqnarray}
where $\al_{H_j} =H_j(2H_j-1)$, it is easy to see that
\begin{eqnarray}\label{e5.13i}
E_{1} \leq C n (t-s).
\end{eqnarray}
 On the other hand, from \eref{e5.17} we have
 \begin{eqnarray*}
\mE\left[     H_{p_{j}}(G^{j}_{k})    H_{p_{j}}(G^{j}_{k'})   \right] &\leq& C (k-k')^{(2H_j-2)p_j}, \quad k,k': |k-k'|\geq2,
\end{eqnarray*}
and thus
 \begin{eqnarray}
E_{2}&\leq&
   C \sum_{k > k' +2 }     \prod_{j=2}^{m}  \mu_{r_{j}-p_{j}}^{2}
    (k-k')^{ (2H_j-2)   p_{j}}
  \nonumber
    \\
    &=&
      C    \sum_{k > k' +2 }
    (k-k')^{  \sum_{j=2}^{m}(2H_j-2)  p_{j}}  \prod_{j=2}^{m}  \mu_{r_{j}-p_{j}}^{2}  .
    \label{e5.14}
\end{eqnarray}
Since $2H_j-2 <0$, the quantity $  (k-k')^{ (2H_j-2) \sum_{j=2}^{m}  p_{j}}   $ reaches maximum when $p_{2}=\cdots=p_{m}=0$. So
\begin{equation}  \label{eq3}
E_{2 } \leq  C    \sum_{k > k' +2 }     1  \le Cn^{2 } (t-s)^{2}.
\end{equation}
 Substituting (\ref{e5.13i}) and (\ref{eq3}) into  to \eref{e5.15ii} and taking into account \eref{e5.13} and the identity
 \begin{eqnarray*}
\sum_{j=2}^m r_jH_j = H_{\al_1} +\dots+ H_{\al_r},
\end{eqnarray*}
 we obtain   inequality \eref{e5.11i} for the case when $r$ is even.

We turn to the case when  $r_1=0$ and $r=r'$ is odd.
Note that
\begin{eqnarray*}
 \prod_{j=2}^{m}  \mu_{r_{j}-p_{j}}   &\neq&0
\end{eqnarray*}
only if   $r_{2}-p_{2},\dots, r_{m}-p_{m}$ are all even, so for $p_{j}$, $j=2,\dots, m$ such that $E_{2} \neq 0$,
\begin{eqnarray*}
 \sum_{j=2}^{m}r_{j} - \sum_{j=2}^{m}p_{j} = r- \sum_{j=2}^{m}p_{j}
\end{eqnarray*}
must be even, and so $\sum_{j=2}^{m}p_{j}$ must be odd. Therefore,  for \eref{e5.14} we have
 \begin{eqnarray}
 E_{2} &\leq& C     \sum_{k > k' +2 }
    (k-k')^{ (2H-2) }
   \nonumber
    \\
    &\leq& Cn^{2H } (t-s) .
    \label{e5.15}
\end{eqnarray}
 Substituting  the estimates  \eref{e5.13i} and \eref{e5.15} into \eref{e5.15ii} and taking into account \eref{e5.13}, we obtain
  the estimate \eref{e5.11i}    when $r$ is odd.

In the case $r_{1}>0$, it follows from the identity
  \eref{e5.11}   that
\begin{eqnarray*}
  \left\|
\sum_{k =  ns/T  }^{   nt/T-1   } J^{\al}_{r} ( \mathcal{A}_{k} )
\right\|_{ 2}
&\leq&
\left(\frac{T}{n}\right)^{ r_{1}  }  \left\|
\sum_{k =  ns/T  }^{   nt/T -1  }    \prod_{j =2 }^{m} ( B^{j}_{t_{k},t_{k+1}}  )^{r_{j}}
\right\|_{ 2}
 .
\end{eqnarray*}
Now we can apply the  inequality \eref{e5.11i} for the case $r_{1}=0$ to the right-hand side of the above equation to obtain the inequality \eref{e5.11i} in   the general case.
\end{proof}

\begin{remark}\label{remark5.12}
By the monotonicity property in Proposition \ref{prop5.5} we can also show that the rate
$1 -  H_\al$   or $H -  H_\al $   obtained in Proposition \ref{prop 2} is
optimal; that is, there exists a constant $C$ such that the right-hand side
of inequality \eref{e5.11i} is   the lower bound for the quantity $ \left\|
\sum_{k =  ns/T  }^{   nt/T-1   } J_{r}^{\al} ( \mathcal{A}_{k} )
\right\|_{ 2}$ with  $ {\mathcal{A}}_{k}=[t_{k},t_{k+1}]^{r}_{<}  $\,.  This follows by taking
 $ \wt{\mathcal{A}}_{k}$ of the  following  form:
 \begin{eqnarray*}
 \wt{\mathcal{A}}_{k} :=  [t_{k}+\frac{h}{2^{r}}, t_{k}+\frac{h}{2^{r-1}}] \times\cdots [t_{k}+\frac{h}{4}, t_{k}+\frac{h}{2}] \times [t_{k}+\frac{h}{2},t_{k}+h],
\end{eqnarray*}
where $h=\frac{T}{n}$ and $k=0,1,\dots, n-1$. Since $\wt{\mathcal{A}}_k \subset \mathcal{A}_k$, by the monotonicity property, it suffices to find a  the lower bound   for the quantity
 \begin{eqnarray*}
 \left\|
\sum_{k =  ns/T  }^{   nt/T -1  } J_{r}^{\al} ( \wt{\mathcal{A}}_{k} )
\right\|_{ 2}\,,
\end{eqnarray*}
 which  can be done  in a  similar way as   in the proof of Proposition \ref{prop 2}.
 \end{remark}

\medskip

\begin{remark}\label{remark5.12i}
Denote $r_j = \# \{i: \al_i =j\}$, $j=1,\dots, m$. From Remark \ref{remark5.7} we see that $\EE[J^\al_r (\mathcal{A})]=0$ if $(r-r_1)$ is odd.  In fact, in a similar way, we can show that $\EE[J^\al_r (\mathcal{A})] \neq 0$ iff $r_j$, $j=2,\dots, m$ are all even numbers.
\end{remark}
We turn to the $L_{2}$-estimate of
centered multiple  integral
$\wt{J}^{\al}_{r} (\mathcal{A}):= {J}^{\al}_{r} (\mathcal{A})-\mE[{J}^{\al}_{r} (\mathcal{A})]$.
We shall prove that the rate in \eref{e5.11i} will be improved and this is the basis
for the introduction of the modified Taylor scheme. Recall that $H=\max\limits_{i:\al_i \neq 1} H_{\al_i}$ and $H_\al = H_{\al_1}+\dots+H_{\al_r}$.
\begin{prop}\label{prop5.10}
 Let  $ \mathcal{A}_{k}  $, $k=0,1,\dots, n-1$ be   as in Proposition \ref{prop 2}. Take $s,t \in D$ and $\al \in \Gamma_r $, $r \geq 2$. Denote $r_j = \# \{i: \al_i =j\}$, $j=1,\dots, m$.  If $r_j$, $j=2,\dots, m$ are even,
then we have
 \begin{eqnarray}\label{eq7.1}
  \left\| \sum_{k  =  ns/T   }^{  nt/T -1}
     \wt{J}^{\al}_{r} ( \mathcal{A}_{k} ) \right\|_{p}
&\leq&
 C(t-s)^{1/2} \omega_{n}(\al)
 \end{eqnarray}
 where
 \begin{eqnarray*}
\omega_{n}(\al) &=& \begin{cases}
        n^{  1/2-H_\al  }
  &\qquad \hbox{if}\quad    { \frac12< H < \frac34 }\, ,\\
      n^{  1/2-H_\al   } \sqrt{\log  n}
  &\qquad \hbox{if}\quad    {  H = \frac34 }\, ,\\
   n^{        2H-1-H_\al   }
   &\qquad \hbox{if}\quad      { \frac34  < H   < 1   }\, .\\
\end{cases}
\end{eqnarray*}

\end{prop}

\begin{proof}
According to Proposition \ref{prop5.8}, it suffices to consider the case when $\mathcal{A}_{k} = [t_{k}, t_{k+1}]^{r}$, $k=0,1,\dots, n-1$.  We first assume that $r_{1} = 0$. By \eref{e5.13ii}, we have
\begin{eqnarray*}
\mE[{J}^{\al}_{r} (\mathcal{A}_{k})  ] &=&  (\frac nT)^{-\sum_{j=2}^m r_jH_j} \prod_{j=2}^{m}     \mu_{r_{j} }.
\end{eqnarray*}
So by denoting $\mathcal{P} = \{ (p_{2},\dots, p_{m}): p_{j}=0,1,\dots, r_{j},\,j=2,\dots, m \}$ and ${\bf{0}} = (0,\dots, 0)$ we can write
\begin{eqnarray*}
\wt{J}^{\al}_{r} (\mathcal{A}_{k})  &=&
(\frac nT)^{-\sum_{j=2}^m r_jH_j} \sum_{(p_{2},\dots, p_{m}) \in \mathcal{P}\setminus {\bf{0}} }\,  \prod_{j=2}^{m}    \frac{r_{j}!}{(r_{j}-p_{j})!} \mu_{r_{j}-p_{j}} H_{p_{j}}(G^{j}_{k})   .
\end{eqnarray*}
 As in \eref{e5.15ii}, we can write
 \begin{eqnarray}
\left\| \sum_{k=k_{1}}^{k_{2}} \wt{J}^{\al}_{r} (\mathcal{A}_{k})  \right\|_{2}^{2}
&\leq&
C n^{-2\sum_{j=2}^m r_jH_j} \sum_{(p_{2},\dots, p_{m}) \in \mathcal{P}\setminus {\bf{0}} }  \left\|  \sum_{k=k_{1}}^{k_{2}} \prod_{j=2}^{m}     \mu_{r_{j}-p_{j}} H_{p_{j}}(G^{j}_{k})    \right\|_{2}^{2}
\nonumber
\\
&\leq&
C n^{-2\sum_{j=2}^m r_jH_j} \sum_{(p_{2},\dots, p_{m}) \in \mathcal{P}\setminus {\bf{0}} } (E_{1} +E_{2}).
\label{e5.19}
\end{eqnarray}

  Since $r$ is even,  by the same argument as in the   proof of the  Proposition  \ref{prop 2}, we see that for   $p_{j}$, $j=2,\dots, m$ such that $E_{2} \neq 0$, $\sum_{j=2}^{m}p_{j}$ must be even.  So, for $(p_{2},\dots, p_{m}) \in \mathcal{P}\setminus {\bf{0}} $, we have $\sum_{j=2}^{m}p_{j} \geq 2$. This implies
\begin{eqnarray*}
E_{2 } &\leq &   C    \sum_{k > k' +2 }
    (k-k')^{ 2(2H-2)  }
  \leq
    \begin{cases}
     C  n (t-s)
  &\qquad \hbox{if}\quad    { \frac12< H < \frac34 }\, ,\\
     C n  {\log n}  (t-s)
  &\qquad \hbox{if}\quad    { H = \frac34 }\, ,\\
   C n^{        4H-2      }  (t-s)
   &\qquad \hbox{if}\quad      { \frac34  < H   < 1   }\, .\\
    \end{cases}
\end{eqnarray*}
Applying \eref{e5.13i} and  the above estimate to  \eref{e5.19}, we obtain the upper bound estimate in \eref{eq7.1}.

Finally,  the estimate \eref{eq7.1} for the case $r_{1} >0$ follows immediately from the estimate in  the case $r_{1}=0$.
\hfill$\Box$
\begin{remark}
As in Remark \ref{remark5.12}, we can show that the upper bound in Proposition \ref{prop5.10} is optimal.
\end{remark}
\begin{remark}
In terms of $\nu_{n}$ defined in \eref{e5.14i}, inequality \eref{eq7.1} becomes
 \begin{eqnarray*}
  \left\| \sum_{k  =  ns/T   }^{  nt/T -1 }
     \wt{J}^{\al}_{r} ( \mathcal{A}_{k} ) \right\|_{p}
&\leq&
 C(t-s)^{1/2}\nu_{n}(\al)\sigma_{n} ,
 \end{eqnarray*}
 where
 \begin{eqnarray}\label{e5.25}
\sigma_{n}&=&
\begin{cases}
        n^{  -1/2     }
  &\qquad \hbox{if}\quad    { \frac12< H < \frac34 }\, ,\\
      n^{  -1/2    } \sqrt{\log n}
  &\qquad \hbox{if}\quad    {  H = \frac34 }\, ,\\
   n^{        2H-2    }
   &\qquad \hbox{if}\quad      { \frac34  < H   < 1   }\, .
\end{cases}
\end{eqnarray}
\end{remark}

\begin{remark}\label{remark5.15}
For     $t\in [0, T]$ and $\gamma \in \Gamma_r$, we define the function 
\[
D_{\gamma}(t):=  \mE\left[B^{\gamma}_{0,t}\right] = \mE[J^{\gamma}_r ([0,t]^r_{<})].
\]
From Remark \ref{remark5.12i}, $D_\gamma (t) = 0$ if some $r_j = \# \{i: \gamma_i=j\}$, $j=2,\dots, m$ is odd. In the following we   derive an explicit formula for $D_\gamma (t)$ when all $r_j$, $j=2,\dots, m$ are even. 

Recall that when $r$ is an even number,  the set $\mathcal{R}_{r}$  is defined in Section \ref{section5.3}. When $r$ is an odd number, we    define $ \mathcal{R}_{r}$ to be the collection of ways to arrange  $(1,\dots, r)$ into $(\frac{r-1}{2})$ pairs and one element and we write  $\tau = \{ (\tau_{1},\tau_{2}),\dots, (\tau_{r-2}, \tau_{r-1}), \tau_{r} \}$, where $\tau_{i}$, $i=1,\dots, r $, are   elements in $\{ 1,\dots, r \}$.

 For $r\in \NN$,  we denote by $  {\mathcal{R}}({\gamma})  $ the subset of $ {\mathcal{R}}_{r}$ such that for $\tau \in   {\mathcal{R}}({\gamma} ) $ we have $\gamma_{\tau_{i}}=\gamma_{\tau_{i+1}}$ for $i=1,3,5,\dots$ and $i<r$, and when $r$ is odd  it satisfies an additional condition that $ \tau_{r} =1$.
We denote $\tau^{*} = \{ i=1,3,5,\dots: \tau_{i} \neq 1, i<r \}$.

We denote by $\dot{B}^{j}_{t}$, $t\in [0,T]$ the fractional white noise associated with the fBm $B^{j}$ (see \cite{HO}), $j=1,\dots, m$. Since 
\begin{eqnarray*}
\mE[\dot{B}^i_t \dot{B}^j_s ] =\al_{H_j} |t-s|^{2H_j-2}  \delta_{i,j},
\end{eqnarray*}
where  $\al_{H_j}= H_j(2H_j-1)$,
 we have
\begin{eqnarray*}
 \mE\left[B^{\gamma}_{0,t}\right] &=& \mE\left[\int_{0}^{t}\cdots\int_{0}^{t_{2}} dB^{\gamma_{1}}_{t_{1}}\cdots dB^{\gamma_{r}}_{t_{r}}\right]
 \\
 &=& \mE\left[\int_{0}^{t}\cdots\int_{0}^{t_{2}}  \dot{B}^{\gamma_{1}}_{t_{1}} \cdots \dot{B}^{\gamma_{r}}_{t_{r}} d {t_{1}}\cdots d{t_{r}}\right]
 \\
 &=& \sum_{\tau \in \mathcal{R}(\gamma)}    \int_{0}^{t} \cdots \int_{0}^{t_{2}}  \prod_{\substack{i\in \tau^{*}
  }}  \al_{H_{\tau_i} }  |t_{\tau_{i}} -t_{\tau_{i+1}}|^{2H-2}
 dt_{1} \cdots dt_{r}.
\end{eqnarray*}
\end{remark}
\section{$L_p$-rates for incomplete  and modified Taylor schemes}
In this section, we consider the numerical approximation of the solution for
the SDE
\begin{eqnarray}\label{e6.1i}
dy_{t} &=& V(y_{t}) dB_{t}, \quad y_{0}\in \RR^{d},\quad t\in [0,T],
\end{eqnarray}
 where  $B=(B^{1},\dots, B^{m})$ and $B^{j}$ is   a standard fractional Brownian motion (fBm) with Hurst parameter $H_j>1/2$ for  $j=2,\dots, m$ and  $B^{1}  $ is the identity function. Assume in addition that $B^2,\dots, B^m$   are mutually independent.

\subsection{The incomplete Taylor scheme for SDE driven by fBm}\label{section6}
 
  In this subsection, we consider the incomplete Taylor scheme \eref{e3.21}   of the following form
\begin{eqnarray}\label{e6.1}
y^{n}_{t}  &=&
y^n_{t_k}+ \wt{\mathcal{E}}^{(N)}_{ t_k,   t } ( y^n_{t_k   }   ), \quad t \in [t_{k}, t_{k+1}], \quad k=0,1,\dots, n-1, \quad  y^n_0=y_0,\nonumber \\
\wt{\mathcal{E}}^{(N)}_{s,t} (y) &=&  \sum_{ \gamma \in \wt{\Gamma}  }
    {\mathcal{V}}_{\gamma   } I(y)     B^{ \gamma }_{s,t}\,.
\end{eqnarray}
where    $\wt{\Gamma} $  is a  finite subset of $\Gamma$. Recall that  $\Gamma=\cup_{r=1} ^\infty \Gamma_r$, and $\Gamma_r$ is the collection of multi-indices of length $r$ with elements in $\{1,\dots, m\}$. We denote   $N = \max_{\gamma \in \wt{\Gamma} } |\gamma|$. 

 Take $\al \in \Gamma $. Denote $r'(\al) = \#\{i: \al_i \neq 1\}$ and 
 \[
  \vartheta(\al)=
   \begin{cases}
   1&\hbox{when  $r' (\al)$ is even } \\
   \max\limits_{i:\al_i\neq 1 } H_{\al_i} &\hbox{when  $r' (\al)$ is odd }
   \end{cases}
.   \]
 We define
\begin{equation}\label{e.def-theta-lp}
\rho=\rho_{\wt{\Gamma}}  =\min\left\{H_\al -\vartheta(\al) \,, \ \al\in \Gamma\setminus \wt{\Gamma}\
\right\}\,,
\end{equation}
where recall that $H_\al = H_{\al_1}+\dots+H_{\al_r}$ for $\al \in \Gamma_r$.

  We first derive two auxiliary results. We take   $\beta  $ such that $1/2<\beta < \min_j H_j$.
   \begin{lemma}\label{lem6.1}
 Let $ R^{e}_{s,t}$, $e=1,3,4$ be defined by  \eref{eqn5.15},
 \eref{eqn 5.12} and \eref{eqn5.14i}. Assume that $\wt{\Gamma}$ has the hierarchical structure introduced in Definition \ref{hi}.  Assume that $V \in C^{N+2}_{b}$.
 Then   the following  estimate    holds for any $i>j$
\begin{eqnarray}\label{eq6.2}
 \left\| \sum_{k=j}^{i-1} R^{e}_{t_{k}, t_{k+1}} (y^{n}_{t_{k}}) \right\|_{p} &\leq& C n^{-\rho } \left( \frac{i-j}{n} \right)^{1/2} \, \quad \hbox{   $e=1, 3, 4$}\,.
\end{eqnarray}
 \end{lemma}
\begin{proof}
According to Proposition \ref{prop4.7},    $R^{e}_{t_{k}, t_{k+1}} (y^{n}_{t_{k}})$, $e=1,3,4$, are  the summations of  quantities of the form $U(y^{n}_{t_{k}}) B^{\al}_{t_{k},t_{k+1}}$ for $\al \in \Gamma\setminus \wt{\Gamma} $.
Since $V\in C^{N+2}_{b}$, it is easy to see from Proposition \ref{prop4.7} that $U \in C^{1}_{b}$. To prove the lemma, it suffices to show that the $L_{p}$-estimate \eref{eq6.2} holds true for
\begin{eqnarray}\label{eq6.3}
\sum_{k=j}^{i-1} U(y^{n}_{t_{k}}) B^{\al}_{t_{k},t_{k+1}}, \quad \al \in \Gamma\setminus  \wt{\Gamma} \,.
\end{eqnarray}
By Proposition \ref{prop 2} and the definition of $\rho $ 
we have
\begin{eqnarray*}
\left\|	\sum_{k=j}^{i-1}   B^{\al}_{t_{k},t_{k+1}} \right\|_{p}&\leq& Cn^{-\rho } \left( \frac{i-j}{n} \right)^{1/2}\,.
\end{eqnarray*}
On the other hand, Lemma \ref{lem 6.4} implies that $\|y^{n}\|_{\beta, p} \leq C$. So
we can apply   Proposition \ref{lem 2.4} to the quantity \eref{eq6.3} to obtain the   estimate
\begin{eqnarray*}
\left\| \sum_{k=j}^{i-1} U(y^{n}_{t_{k}}) B^{\al}_{t_{k},t_{k+1}} \right\|_{p} &\leq& C n^{-\rho } \left( \frac{i-j}{n} \right)^{1/2} .
\end{eqnarray*}
This completes the proof.
\end{proof}

   \begin{lemma}\label{lem6.2}
   Let $f$ be a stochastic process on $[0,T]$, such that $\mE(\|f\|_{\beta}^{p})$ and $\mE(\|f\|_{\infty}^{p})$ are finite for  all $p \geq 1$. Let $\wt{\Gamma}$, $R^e_{s,t}(y)$, $e=1,3,4$ and $V  $ be as in Lemma \ref{lem6.1} and  let $R^e_{s,t} (y)$ be defined in \eref{eqn5.15i}. Then the following estimate is true for $e=1,2,3,4$ and $l=0,1,\dots, n-1$,
\begin{eqnarray}\label{eq6.4}
\left\| \sum_{k=0}^{l} \int_{t_{k}}^{t_{k+1}} f_{u} d R^{e}_{t_{k}, u } (y^{n}_{t_{k}})  \right\|_{p} &\leq& K n^{-\rho },
\end{eqnarray}
where $K$ depends on  $\mE(\|f\|_{\beta}^{p})$,  $\mE(\|f\|_{\infty}^{p})$  and the vector fields $V_j$.
 \end{lemma}

 \noindent \textit{Proof:} \quad
 According to the estimate \eref{eq5.3i} in Lemma \ref{lem5.3} and taking into account the assumption that $\mE(\|f\|_{\beta}^{p})$ and $\mE(\|f\|_{\infty}^{p})$ are finite, we have
 \begin{eqnarray*}
 \left\| \int_{t_{k}}^{t_{k+1}} f_{u} d R^{2}_{t_{k}, u } (y^{n}_{t_{k}})  \right\|_{p} &\leq& K n^{-(N+2)\beta},
\end{eqnarray*}
 where we  recall that $N= \max_{\gamma\in \wt{\Gamma}} |\gamma|$.  Therefore,
\begin{eqnarray*}
 \left\| \sum_{k=0}^{l} \int_{t_{k}}^{t_{k+1}} f_{u} d R^{2}_{t_{k}, u } (y^{n}_{t_{k}})  \right\|_{p} &\leq& \sum_{k=0}^{n-1}  \left\| \int_{t_{k}}^{t_{k+1}} f_{u} d R^{2}_{t_{k}, u } (y^{n}_{t_{k}})  \right\|_{p}
 \\
 &\leq& K n^{1-(N+2)\beta}
 \\
 &\leq&K n^{-\rho },
\end{eqnarray*}
where the last inequality   follows by taking $\beta$ sufficiently close to $\min_jH_j$ and the fact that we can find $\al \in \Gamma_{N+1} $ such that $H_\al - \vartheta (\al) < (N+2)\beta-1$.

 We turn to the case   $e=1,3,4$. We write
 \begin{eqnarray*}
 \int_{t_{k}}^{t_{k+1}} f_{u} d R^{e}_{t_{k}, u } (y^{n}_{t_{k}})&=&   \int_{t_{k}}^{t_{k+1}} \int_{t_{k}}^{u } df_{v} d R^{e}_{t_{k}, u } (y^{n}_{t_{k}})+ f_{t_{k}}     R^{e}_{t_{k}, t_{k+1} } (y^{n}_{t_{k}})
\\
&=:& R^{e,1}_{k}+R^{e,2}_{k}.
\end{eqnarray*}
Applying Proposition \ref{lem 2.4} to $\sum_{k=0}^{l} R^{e,2}_{k}$ and taking into account the estimate in Lemma \ref{lem6.1} and the assumption that $\|f\|_{\beta, p}$  is finite, we obtain the inequality
\begin{eqnarray}\label{eq6.6}
\left\| \sum_{k=0}^{l} R^{e,2}_{k} \right\|_{p} &\leq& K n^{-\rho }.
\end{eqnarray}
According to Proposition \ref{prop4.7}, the quantity $R^{e,1}_{k}$ is a linear combination of the terms
\begin{eqnarray*}
 \int_{t_{k}}^{t_{k+1}} \int_{t_{k}}^{u} df_{v} d B^{\al}_{t_{k}, u} , \quad \al \notin \wt{\Gamma}.
\end{eqnarray*}
Take $\beta_{j}<H_j$ for $j=1,\dots, m$.   Then by Lemma \ref{lem 7.2} we have
\[
\left\|  \int_{t_{k}}^{t_{k+1}} \int_{t_{k}}^{u} df_{v} d B^{\al}_{t_{k}, u} \right\|_{p} \le   K n^{-\beta_{\al_{1}}-\cdots-\beta_{\al_{p}}-\beta }
\le  K n^{-\rho -1},
\]
 where the last inequality follows by taking $\beta_{j}$, $j=1,\dots, m$ sufficiently close to $H_j$ and $\beta$ to $\min_j H_j$.
Therefore,
\begin{equation}
\left\| \sum_{k=0}^{l} R^{e,1}_{k} \right\|_{p} \leq \sum_{k=0}^{n-1} \left\|  R^{e,1}_{k} \right\|_{p}
\le \sum_{k=0}^{n-1}  Kn^{-\rho -1}
=K n^{-\rho }.
\label{eq6.7}
\end{equation}
Combining  \eref{eq6.6} and \eref{eq6.7}, we obtain the inequality \eref{eq6.4}   for $e=1,3,4$.
 \end{proof}

 \begin{theorem}\label{thm6.3}
 Let    $\wt{\Gamma}$ be a finite subset of $\Gamma$. Assume that $\wt{\Gamma}$ satisfies the hierarchical structure defined in Definition \ref{hi}.
 Let $y$   be the solution of equation \eref{e6.1i}   and   $y^n$ be the solution of numerical equation   \eref{e6.1}
with $  \wt{\mathcal{E}}^{(N)}_{s,t}$   defined in \eref{e6.1}. Take $M>0$. Assume that $V \in C^{N+2}_{b}$. Then
 \begin{eqnarray}\label{e6.9ii}
\sup_{t\in [0,T]}\left( \mE|\mathbf{1}_{\{\|B\|_{\beta} <M\} }( y_{t}-y^{n}_{t}) |^{p}\right)^{1/p} &\leq& C_M n^{-\rho },
\end{eqnarray}
where $C_M$ is a constant depending on   $M$.
 \end{theorem}
 \begin{proof}   By Theorem \ref{cor 5.9}, we have
 \begin{eqnarray*} 
 &&
 \mathbf{1}_{\{\|B\|_{\beta} <M\} }( y_{t}-y^{n}_{t})
 \nonumber
 \\
 &&
 = \sum_{e=1}^4  \left[ \mathbf{1}_{\{\|B\|_{\beta} <M\} } {\Phi}_t \right] \sum_{k=0}^{\lfloor \frac{nt}{T}  \rfloor}
 \int_{t_k}^{t_{k+1} \wedge t}
  \left[ \mathbf{1}_{\{\|B\|_{\beta} <M\} }  {\Psi}_s\right]
   d     R^e_{t_k, s} (y^n_{t_k})\,,
\end{eqnarray*}
where $\Phi$ and $\Psi$ are solutions of equations \eref{eqn5.24} and \eref{eqn5.23}.
According to the estimate of $ \|\Phi\|_{\infty}$ in Lemma \ref{lem 6.4}, the $L_{p}$-norm of the quantity $\mathbf{1}_{\{\|B\|_{\beta} <M\} } {\Phi}_t$ is less than a constant $C_M$  is independent of $n$. So to prove the theorem, it suffices to show that the $L_{p}$-norm of
\begin{eqnarray}\label{eq6.8}
 \sum_{k=0}^{\lfloor \frac{nt}{T}  \rfloor}
 \int_{t_k}^{t_{k+1} \wedge t}
  \left[ \mathbf{1}_{\{\|B\|_{\beta} <M\} }  {\Psi}_s\right]
   d     R^e_{t_k, s} (y^n_{t_k})
\end{eqnarray}
is less than $C_Mn^{-\rho }$.  We take $f_{s}= \mathbf{1}_{\{\|B\|_{\beta} <M\} }  {\Psi}_s$, then it follows again from Lemma \ref{lem 6.4}  that $\mE[\|f\|_{\beta}^{p}]$ and $\mE[\|f\|_{\infty}^{p}]$ are bounded by a constant independent of $n$. So   applying   Lemma \ref{lem6.2} to \eref{eq6.8} we obtain the upper bound $C_Mn^{-\rho }$. This completes the proof.
 \end{proof}
 \begin{remark}\label{remark6.4}
 With more careful estimates in Lemmas \ref{lem6.1} and
 \ref{lem6.2} and with the help of Remark \ref{remark5.12}, we can   show that    the convergence rate of the incomplete Taylor scheme obtained in Theorem \ref{thm6.3} is optimal, that is, we can find a constant $C$ such that the left-hand side of \eref{e6.9ii} is greater than its right-hand side.
 \end{remark}

 The following result follows immediately from Theorem \ref{thm6.3}.
\begin{cor}\label{cor6.4}
Let the assumption be   as in Theorem \ref{thm6.3}.
  The scaled error  $n^{\rho }(y_{t}-y^{n}_{t})$,  $t\in [0,T]$ of the numerical scheme   is bounded in probability (or tight), that is, for every $\varepsilon>0$, there exists   $C>0$ such that
\begin{eqnarray*}
P(n^{\rho }|y_{t}-y^{n}_{t}| >C ) \leq \varepsilon \quad \text{for all } n.
\end{eqnarray*}
\end{cor}

To obtain the best choice of $\wt{\Ga}$  by Theorem \ref{thm6.3}, we will follow the same ideas as in  \eref{e.theta-form}
and \eref{e.gamma-theta}. 
Take  nonnegative integers  $r_1,\dots, r_m$ and denote $r' = \sum_{j=2}^m r_j$.
 First, we see that a possible $L_p$-rate has the form
\begin{equation}
\rho= \begin{cases}   \sum_{j=1}^m r_jH_j -1  &\qquad if ~ r'  ~\text{is even}   \\
   \sum_{j=1}^m r_jH_j  - \max\limits_{j>1: \,r_j>0} H_j &\qquad if ~  r'  ~ \text{is odd}\,.
  \end{cases}
  \label{e.rho-form}
\end{equation}

Given a $\rho$ of the above form we define
\begin{eqnarray}
      \widehat\Gamma(\rho) &=& \left\{ \al \in \Gamma:  H_\al -1 < \rho , ~r'(\al) \text{ is even}
      \right\}
      \nonumber
      \\
      &&\bigcup  \left\{ \al \in \Gamma:  H_\al - \max\limits_{j: \al_j\neq 1}H_{\al_j} < \rho , ~r'(\al) \text{ is odd} \right\} \,.
      \label{e.gamma-rho}
 \end{eqnarray}
\begin{lemma}Given a $\rho$  of the form \eref{e.rho-form}, we define $\widehat\Gamma(\rho)$ by \eref{e.gamma-rho}.  Then
\begin{equation*}
\displaystyle \rho_{\,_{\widehat\Gamma(\rho)}}=\rho\,.
 \end{equation*}
\end{lemma}
\begin{proof} The proof of this lemma is similar to that of Lemma \ref{l.4.6}.
\end{proof}

\begin{remark} From this  lemma and Theorem \ref{thm6.3}, we see that all possible rates
for the $L_p$-convergence has the form \eref{e.rho-form}.  Given a rate of the form
\eref{e.rho-form} the best choice of  $\wt{\Ga}$  in \eref{e6.1}  for the $L_p$-convergence is
\eref{e.gamma-rho}.
\end{remark}

{ 
\begin{remark}\label{remark6.5}
We compare Theorem \ref{thm6.3} to the strong convergence results  in \cite{Hu,KP}.
We consider the $d$-dimensional Stratonovich SDE:
\begin{eqnarray*}
y_{t} =y_{0} +\int_{0}^{t} V(y_{s}) dW_{s},
\end{eqnarray*}
where $W=(W^{1},\dots, W^{m})$, $W^{j}$ is   a standard    Brownian motion for $j=2,\dots, m$,  $W^{1}_{t}\equiv t$ and $W^{j}$, $j=2,\dots, m$ are mutually independent, and the integral on the right-hand side is a Stratonovich integral. The numerical scheme studied in \cite{KP} coincides with the incomplete Taylor scheme \eref{e6.1} constructed   here. Indeed, by taking $H_j=1/2$, $j=2,\dots, m$ and $H_1 = 1$ in \eref{e.rho-form} we see that  the possible rates of convergence are $  \{1,2,3,\dots\}$, and for $\rho =1,2,\dots$, the set $\widehat\Gamma(\rho )$  becomes
\begin{eqnarray*}
\widehat\Gamma(\rho ) &=& \{\al \in \Gamma : r' (\al) +2r_{1}(\al) < 2\rho +1 \},
\end{eqnarray*}
or
\begin{eqnarray}\label{e6.9i}
\widehat\Gamma(\rho ) &=& \{\al \in \Gamma : |\al| +r_{1}(\al) \leq 2\rho  \},
\end{eqnarray}
where $r_1 (\al) =\#\{i: \al_i =1\}$, $r' (\al) =\#\{i: \al_i \neq 1\}$ and $|\al|$ is the length of $\al$. 
By taking
\begin{eqnarray*}
\wt{\mathcal{E}}^{(N)}_{s,t} (y) &=&  \sum_{ \gamma \in \widehat {\Gamma} (\rho ) }
    {\mathcal{V}}_{\gamma   } I(y)     W^{ \gamma }_{s,t}\,,
\end{eqnarray*}
 with $\widehat\Gamma(\rho )$   defined in \eref{e6.9i},  we obtain
the numerical scheme considered in \cite{KP}:
\begin{eqnarray*}
y^{n}_{t}  &=&
y^n_{t_k}+ \wt{\mathcal{E}}^{(N)}_{ t_k,   t } ( y^n_{t_k   }   ), \quad y^{n}_{0}=y_{0},
\end{eqnarray*}
for $ t\in [t_{k},t_{k+1}]$, $k=0,1,\dots, n-1$.
 The following strong convergence result is obtained in \cite{KP}:
 \begin{eqnarray*}
\mE(|y_{t}-y_{t}^{n}|^{2})^{1/2} &\leq& Cn^{-\rho }.
\end{eqnarray*}
In particular, the convergence rate $n^{-\rho }$ of the incomplete Taylor scheme in  the Brownian case coincides with the convergence rate in the fBm case. So we can consider Theorem \ref{thm6.3} as a generalization of \cite{KP} to the fBm case.
\end{remark}
}

\begin{example}\label{ex6.5}
We consider the scalar SDE
\begin{eqnarray}\label{e6.9}
y_{t} &=& y_{0}+\int_{0}^{t}V(y_{s}) dB_{s},\quad y_0 \in \RR,
\end{eqnarray}
where $B$ is a one-dimensional fBm with Hurst parameter $H>1/2$.
  Take $N\in \NN$. The  order-$N$  Taylor expansion in this case is
\begin{eqnarray*}
\mathcal{E}_{s,t}^{(N)}(y) &=& \frac{1}{r!} \sum_{ r=1 }^{N}
    {\mathcal{V}}_{r   } I(y)     (B_{t}-B_{s})^{r},
\end{eqnarray*}
where  $\mathcal{V}_{1} =\mathcal{V}$, and $\mathcal{V}_{r+1} = \mathcal{V}_{r} \mathcal{V}  $, r=1,\dots, N. So the global numerical scheme   associated with this Taylor expansion is
\begin{eqnarray}\label{eq6.9}
y^{n}_{t} &=& y^{n}_{t_{k}} +  \frac{1}{r!} \sum_{ r=1 }^{N}
    {\mathcal{V}}_{r   } I(y^{n}_{t_{k}})     (B_{t}-B_{t_{k}})^{r}, \quad t \in [t_{k}, t_{k+1}].
\end{eqnarray}
This is the numerical scheme studied in \cite{GN}. By taking $m=d=1$, we recover  from Theorem \ref{thm6.3}   the strong convergence result of \eref{eq6.9} obtained in \cite{GN}: the numerical scheme  $y^{n}$ defined in \eref{eq6.9} converges to the solution $y$ of \eref{e6.9} with rate $n^{1-(N+1)H}$ when $N$ is odd and with rate $n^{-NH}$ when $N$ is even.
\end{example}

  \subsection{Modified  Taylor scheme}\label{section 7}
 In this subsection, we briefly explain how to improve the convergence rate of    the numerical scheme studied in   Section \ref{section6} by a slight modification of the scheme. The proof of the main result in this subsection is similar to   the previous subsection, and the proof is  omitted.

By comparing Proposition \ref{prop 2} with Proposition \ref{prop5.10},  we see    that   the
centered multiple integral
 \begin{eqnarray*}
\wt{J}_{r} ( \mathcal{A} ) &=&  {J}_{r} ( \mathcal{A} )  - \mE \left[ {J}_{r} ( \mathcal{A} )   \right],
 \end{eqnarray*}
   usually   will have a smaller $L_2$-norm,
 where $\mathcal{A}$ is a subset of $[0,T]^{r}$ such that the multiple  integral ${J}_{r} ( \mathcal{A} ) $ is well defined.
This   leads us to consider the following modification of the Taylor expansion.

Take  nonnegative integers  $r_1,\dots, r_m$ and denote $r' = \sum_{j=2}^m r_j$. 
Let  $\rho$ be of the form
\begin{equation}
\rho= \begin{cases} \sum_{j=1}^m r_jH_j -1  &\qquad if ~ r' ~\text{is even}   \\
   \sum_{j=1}^m r_jH_j - \max\limits_{j>1: \, r_j>0 } H_{ j}  &\qquad if ~ r' ~ \text{is odd}
  \end{cases}, 
  \label{e.theta-form-modify}
\end{equation}
  and define
\begin{eqnarray*}  
      \widehat\Gamma(\rho) &=& \{ \al \in \Gamma:  H_\al -1 < \rho, ~r'(\al) \text{ is even} \}
      \nonumber
      \\
      &&\bigcup  \{ \al \in \Gamma:  H_\al- \max\limits_{j: \,  \al_j \neq 1} H_{\al_j}    < \rho, ~r'(\al) \text{ is odd} \} ,
 \end{eqnarray*}
 where recall   that   $H_\al = H_{\al_1}+\dots+H_{\al_r}$ and $r'(\al)= \# \{ i: \al_i\neq 1 \}$.
 
 We denote $\rho'=\min\{\delta: \delta \
 \hbox{is of the form \eref{e.theta-form-modify} and }  \delta>\rho\}$, and define $\widehat\Gamma(\rho)' := \widehat\Gamma(\rho') \setminus \widehat\Gamma(\rho)$. Recall that we define the function $ D_{\gamma}(t):=  \mE\left[B^{\gamma}_{0,t}\right] $; see Remark \ref{remark5.15} for an explicit expression.
  \begin{Def}\label{def6.7}
Let  $\rho $ be of the form   \eref{e.theta-form-modify}.
 We call
\begin{eqnarray*} 
 \widehat{\mathcal{E} }_{s, t}(y)
=
  \sum_{ \gamma \in \widehat\Gamma(\rho)    }
    {\mathcal{V}}_{\gamma   } I(y)     B^{ \gamma }_{s,t} + \sum_{ \gamma  \in \widehat\Gamma(\rho)'  }
    {\mathcal{V}}_{\gamma   } I(y)   D_{\gamma}(t-s)
  \end{eqnarray*}
   the modified     Taylor expansion.
   \end{Def}

 \begin{remark}\label{remark3.6}
 In    Proposition \ref{prop5.2} we have shown  that the identity \eref{eq 5.15}  holds for    the incomplete Taylor expansion $\wt{\mathcal{E}}^{(N)}_{s,t} $. In fact,      \eref{eq 5.15} also  holds true when   $\wt{\mathcal{E}}^{(N)}_{s,t}$ is replaced by the Taylor expansion $\widehat{\mathcal{E}}_{s,t}$.  In fact, the only properties of the incomplete Taylor expansion $\wt{\mathcal{E}}^{(N)}_{s,t} $ we used  in the proof of the proposition are:
 \begin{enumerate}
\item  The multiple integrals   appearing in the proof are well defined;
\item the chain rule used in   \eref{e3.20} holds true.
  \end{enumerate}
\end{remark}

 We consider the following modified Taylor scheme:
 \begin{eqnarray}\label{e6.13i}
y^{n}_{t}  &=&
y^n_{t_k}+ \widehat{\mathcal{E}}_{ t_k,   t } ( y^n_{t_k   }   ), \quad t \in [t_{k},t_{k+1}], \quad k=0,1,\dots, n-1.
\end{eqnarray}
 As in  Remark \ref{remark3.6}, it is easy to show  that   if  $R^{e}$, $e=1,2,3,4$ are defined in \eref{eqn5.15}, \eref{eqn5.15i}, \eref{eqn 5.12} and \eref{eqn5.14i} with $\wt{\mathcal{E}}^{(N)}_{s,t}$ replaced by $\widehat{\mathcal{E}}_{s,t}$ and $y^{n}$ is the modified Taylor scheme \eref{e6.13i}, then
 identity   \eref{eq3.24} still holds true.

Based on   estimate \eref{eq7.1} and   identity \eref{eq3.24}, we can prove the following stronger  convergence result. Recall that for $\al \in \Gamma $, we denote $r_j(\al) = \#\{i:  {\al_{i}}= j\}$, $j=1,\dots, m$.

\begin{theorem} \label{thm7.2}
Assume that $V\in C^{N+2}_{b}$. Let $y$   be the solution of equation \eref{e6.1i}   and   $y^n$ be the numerical scheme   \eref{e6.13i}. Take $M>0$. If $r_j(\al)$, $j=2,\dots, m$ is even for each $\al \in \widehat\Gamma(\rho)'$,  then
 \begin{eqnarray*}
\sup_{t\in [0,T]}\left( \mE|\mathbf{1}_{\{\|B\|_{\beta} <M\} }( y_{t}-y^{n}_{t}) |^{p}\right)^{1/p} &\leq&
C_Mn^{-\rho}\sigma_{n},
\end{eqnarray*}
where $\sigma_{n}$ is defined in \eref{e5.25}
and $C_M$ is a constant depending on $M$. In particular,   the scaled error $\sigma_{n}^{-1}n^{\rho}(y_{t}-y^{n}_{t})$  is bounded in probability for each $t\in [0,T]$.
 \end{theorem}
 \begin{remark}
As in Remark \ref{remark6.4}, we can show that   the convergence rate obtained in Theorem \ref{thm7.2} is optimal.
\end{remark}
 Following are two applications of the modified   Taylor scheme. For simplicity, we   assume from now on that $B=(B^1,\dots,B^m)$ is a $m$-dimensional standard fBm with Hurst parameter $H>1/2$.

\begin{example}\label{ex6.10}
Take $\rho=2H-1$. Then $\widehat\Gamma(\rho) = \{1,\dots, m  \}$ and $N=1$.  Take $\gamma \in \Gamma$ such that $|\gamma|=N+1=2$, and denote $\gamma=(j,j')$. We calculate $D_{\gamma}(t) $,
\begin{eqnarray*}
D_{\gamma}(t)=\mE[B^{\gamma}_{0,t}] &=& \mE\left[\int_{0}^{t} \int_{0}^{u} d B^{j}_{u'} dB^{j'}_{u}\right]
\\
&=& \frac12 t^{2H} \delta_{jj'},
\end{eqnarray*}
where $\delta_{jj'}$ is the Kronecker function such that $\delta_{jj'}=1$ when $j=j'$ and $\delta_{jj'}=0$ other wise.
 Then the modified  order-$2$  Taylor expansion is
\begin{eqnarray*}
\widehat{\mathcal{E}} (y) &=& V(y) B_{s,t} + \frac12 \sum_{j =1}^{m} \sum_{i=1}^{d} V^{i}_{j} \partial_{i} V_{j} (y) (t-s)^{2H} .
\end{eqnarray*}
 The modified Taylor scheme associated with this   scheme is
 \begin{eqnarray*}
y^{n}_{t_{k+1}} &=& y^{n}_{t_{k}} + V(y^{n}_{t_{k}} ) B_{t_{k}, t_{k+1}} + \frac12 \sum_{j =1}^{m}   ( \partial V_{j} V_{j} )(y^{n}_{t_{k}} ) (T/n)^{2H} ,
\end{eqnarray*}
for   $k=0,\dots, n-1$.
This is   the modified Euler scheme introduced in \cite{HLN}. By taking $\rho=2H-1$ we recover from Theorem \ref{thm7.2} the convergence result \eref{e1.10i}.
\end{example}

\begin{example}\label{ex6.11}
Let $N$ be an odd integer. We consider the model in Example \ref{ex6.5}.  The modified Taylor expansion for this model becomes
\begin{eqnarray*}
\widehat{\mathcal{E}}_{s,t}(y) &=& \sum_{ r=1 }^{N} \frac{1}{r!}
    {\mathcal{V}}_{r   } I(y)     (B_{t}-B_{s})^{r} + \frac{1}{(N+1)!!}    {\mathcal{V}}_{N+1   } I(y)   (t-s)^{(N+1)H }.
\end{eqnarray*}
The modified   Taylor scheme   associated with this   scheme is
\begin{eqnarray*}
y^{n}_{t_{k+1}} &=& y^{n}_{t_{k}} +  \sum_{ r=1 }^{N} \frac{1}{r!}
    {\mathcal{V}}_{r   } I(y^{n}_{t_{k}})     (B_{t_{k+1}}-B_{t_{k}})^{r}
    \\
    &&+ \frac{1}{(N+1)!!}    {\mathcal{V}}_{N+1   } I(y^{n}_{t_{k}})   (T/n)^{(N+1)H } .
\end{eqnarray*}
According to Theorem \ref{thm7.2},  the convergence rate of this scheme is $n^{1/2-H(N+1)}$ for $1/2<H<3/4$; $n^{1/2-3(N+1)/4} \sqrt{\log n}$ for $H=3/4$ and $n^{ -1-H(N-1)}$ for $H>3/4$, which  improves the numerical scheme  \eref{eq6.9}.
\end{example}

\section{Numerical approximation in the rough paths case}\label{section7}
In this section, we consider the numerical approximation for the $d$-dimensional    rough differential equation:
 \begin{eqnarray}\label{e7.1i}
dy_t
  &= &    V  (y_t )d x_t
   \end{eqnarray}
 on $ [0, T]$, where the control function $x \in C([0,T], \RR^{m})$ is not differentiable, but is enriched with a proper algebraic structure.   The theory of rough paths analysis has been developed from the seminal paper by Lyons \cite{Ly2}.  Our settings  in this section will follow closely  \cite{FV}.

As in \eref{e3.2i}, we can define the Taylor scheme  for the solution of \eref{e7.1i} based on the Taylor expansion:
\begin{eqnarray*}
\mathcal{E}_{s,t}^{(N)}(y  )
&=&
 \sum_{ \gamma \in \Gamma,  |\gamma| \leq N }
    {\mathcal{V}}_{\gamma   } I(y)     x^{ \gamma }_{s,t}
  \,, \quad y \in \RR^{d},
\end{eqnarray*}
where  $ x^{ \gamma }_{s,t}$ is a multiple rough integral that we will define later. Our aim in this section is to show that the expression for $y-y^n$ derived in \eref{eq3.24},  still holds true in the rough paths case.
Notice that   the results in   Section \ref{section3.1} are only based on the algebraic properties of the differential equation \eref{eqn5.10}, so to  show       \eref{eq3.24}, it suffices to derive a rough paths version of  Proposition \ref{prop 4}.

In the first subsection, we briefly review some concepts and results from the rough paths theory. In the second subsection, we generalize Proposition \ref{prop 4} to the rough paths case.

 \subsection{Elements of the rough paths theory}\label{section 3}

Denote by $C^{p\tvr}([s,t])$ the collection of continuous functions on $[s,t]$ with bounded $p$-variation.  We first define  the  step-$N$  signature.
\begin{Def}
The  step-$N$  signature
 of $\gamma \in C^{1\text{-var}} ([s, t]; \mathbb{R}^m)$ is given by
 \begin{eqnarray*}
 S_N(\gamma)_{s, t} & \equiv & \left( 1, \int_{s<u<t} d\gamma_u\,, \dots, \int_{s<u_1<\dots<u_N <t  } d\gamma_{u_1} \otimes \cdots \otimes d\gamma_{u_N } \right)
 \\
& \in& \oplus_{k=0}^N (\mathbb{R}^m)^{\otimes k}.
 \end{eqnarray*}
\end{Def}
We denote by $G^N(\mathbb{R}^m)  $ the so-called \emph{free nilpotent group of step $N$}    over $\RR^m$, that is,
 \begin{eqnarray*}
\ngs := \left\{
S_N( \gamma )_{0, 1} : \gamma \in \cm
\right\}.
\end{eqnarray*}
It is well-known that $\ngs$ is a   Lie group with respect to the tensor multiplication $\otimes$, and  for every $g \in \ngs$,  the  ``Carnot-Carath\'eodory norm"
\begin{eqnarray*}
\|g\| := \inf \left\{
\int_0^1 |d\gamma|: \gamma \in \cm \text{ and } S_N(\gamma)_{0,1} =g
\right  \}
\end{eqnarray*}
is finite and achieved at some minimizing path $\gamma^*$, i.e.
\begin{eqnarray*}
\|g\| = \int_0^1 |d\gamma^*| \text{ and }   S_N(\gamma^*)_{0,1} =g  \,.
\end{eqnarray*}
The norm $\| \cdot \|$
   leads to a metric
$d(g, h) :=   \|g^{-1} \otimes h \|
$
 on $\ngs$, called the \emph{Carnot-Carath\'eodory metric}.

 Consider a $\ngs$-valued path $ \bfx $ on the time interval $[0, T]$. For any $p \geq 1$,
  we denote by
 \begin{eqnarray*}
 \|\bfx \|_{p \text{-var} ;[s, t]}
 &=&
 \sup_{(t_i) \subset[s,t]} \left( \sum_i d(\bfx_{t_i}, \bfx_{t_{i+1 } } )^p \right)^{1/p}
  \end{eqnarray*}
  the $p$-variation norm of $\bfx$, and when $s=0$, $t=T$, we simply write $  \|\bfx \|_{p \text{-var}  } =  \|\bfx \|_{p \text{-var} ;[0, T]}    $.
 The following proposition shows that an abstract path $\bfx : [0, T] \rightarrow G^N(\mathbb{R}^m)$ of bounded $p$-variation can be approximated by a sequence of  step-$N$  signatures; see \cite{FV}. We  denote by $d_\infty$   the infinity distance, that is,
 \begin{eqnarray*}
 d_{\infty}( \bfx, \bfx' ) := \sup_{t\in [0, T] } d (\bfx_t, \bfx_t'),
\end{eqnarray*}
and  $C^{p\text{-var} } ([0, T]; G^{N} (\mathbb{R}^m ) )$   stands for
\begin{eqnarray*}
 \left\{ \bfx \in C([0, T]; G^{N}(\mathbb{R}^m ) ): \|\bfx\|_{p\text{-var} } <\infty \right\}.
\end{eqnarray*}

\begin{prop}\label{prop 3.2}
Let $\bfx \in C^{p\text{-var} } ([0, T]; G^N(\mathbb{R}^m))$, $p\geq 1$, with $\bfx_0 = (1, 0,\dots, 0)$. Then there exists $( x^{n})_n \subset C^{1  } ([0, T]; \mathbb{R}^m ) $, such that
\begin{eqnarray*}
d_{\infty} (\bfx, S_N(x^{n} ) ) \rightarrow 0 \text{ as } n \rightarrow \infty,\quad
\text{and}\quad
\sup_n \|S_N (x^{n}) \|_{p \text{-var} }  < \infty.
\end{eqnarray*}
\end{prop}

Consider the  \emph{   rough  differential equation} (RDE)
\begin{eqnarray}\label{eqn3.1}
d y_t = V( y_t) d \bfx_t, \quad   y_0 \in  \RR^d   ,
\end{eqnarray}
where $\bfx : [0, T] \rightarrow G^{\lfloor p \rfloor} (\RR^m ) $ is a \emph{weak geometric $p$-rough path}, i.e.  an element in
$
C^{p\text{-var} } ([0, T]; G^{\lfloor p \rfloor} (\mathbb{R}^m ) )
$.

\begin{Def}
We take $\bfx \in C^{p \tvr } ( [0,T]; G^{\lfloor p \rfloor} (\RR^m) ) $. We say that $  y \in C ( [0,T]; G^{\lfloor p \rfloor} (\RR^d) )  $ is a     solution  to   equation  \eref{eqn3.1} if for any sequence $(x^n)_n$ in $C^{1\tvr} ([0,T]; \RR^m)$ such that
\begin{eqnarray}\label{eqn3.1i}
d_{\infty} (\bfx , S_{ \lfloor p \rfloor }( x^{n} ) ) \rightarrow 0 \text{ as } n \rightarrow \infty,
\text{ and }
\sup_n \|S_{\lfloor p \rfloor} (x^{n}) \|_{p \tvr }  < \infty,
\end{eqnarray}
with $y^n$   the solution of the equation $dy^n = V (y^n) dx^n $, there exists a subsequence of $(x^n, y^n) $ $($which we still denote by $(x^n, y^n)$$)$ such that
$ y^n$ converges uniformly to $ y$ when $n\rightarrow \infty$.  
\end{Def}

\begin{thm}\label{thm 3.3}
Assume that  $V= (V_j)_{1\leq j\leq m}$ is a collection of  $C^{\lfloor p \rfloor+1}_{b}$-vector fields on $\RR^d$.
 Then, there exists a unique   RDE solution to the equation \eref{eqn3.1}.     The conclusion still holds   when   $V = (V_j)_{1\leq j\leq m} $ is a collection of linear vector fields.
 \end{thm}

 To define the rough path integral  $\int_0^{\cdot} V(x_t) d\bfx_t   $, we consider  the following RDE
 \begin{eqnarray*}
 d z_t &=& d \bfx_t, \\
 d y_t&=& V (z_t) d \bfx_t,
 \\
 ( z_0,  y_0 ) &=&(0, 0).
\end{eqnarray*}
   It follows from Theorem \ref{thm 3.3} that if $V \in C^{\lfloor p \rfloor+1}_{b} (\RR^d) $,  then the above equation
 has a unique
solution $(  z,  y )$. We call $ y$
 the \emph{rough integral}, denoted as  $   \int_0^{\cdot} V( x_t ) d\bfx_t$.

\subsection{Multiple rough integrals}
In this subsection,
   we consider some multiple rough integrals.
  We denote by $S_N(\bfx)$, $N \geq \lfloor p \rfloor $,    the so-called \emph{Lyons lift} of $\bfx$, which satisfies $\pi_i( S_N(\bfx) ) = \pi_i(\bfx)$ for $i=1,\dots, \lfloor p \rfloor$ and $S_N(\bfx) \in C^{p\tvr} ([0, T]; G^{N} (\RR^m))$.   We refer to   Section 9.1.2 in \cite{FV} for the proof of   the existence and uniqueness for the \emph{Lyons lift} of a weak geometric $p$-rough path. Following is a basic fact on weak geometric rough paths.  It shows that   if $\bfx$ is a weak geometric $p$-rough path,  $p\geq 1$,   then the multiple integral
	 \begin{eqnarray*}
 \int_s^t   \cdots \int_s^{u_2} d\bfx_{u_1} \otimes \cdots \otimes d\bfx_{u_i} , \quad s,t \in [0,T],
\end{eqnarray*}
coincides with the $i$th tensor level of the order-$N$ \emph{Lyons's lift}    of $\bfx$, where $N \geq i$.
\begin{lemma}\label{lem3.5}
 Let $\bfx$ be a weak geometric $p$-rough path and $N \in \NN$. Then for each $i=1,\dots, N$, we have
\begin{eqnarray*}
\pi_i(S_N(\bfx)_{s,t}) =   \int_s^t \int_s^{t_i} \cdots \int_s^{t_{2}}    d \bfx_{t_1} \otimes \cdots  \otimes  d \bfx_{t_{i-1}}  \otimes d \bfx_{t_i}  \, .
\end{eqnarray*}
\end{lemma}
 \begin{proof}  Without   loss of generality, we assume $s=0$.
 We consider the following linear   system of equations
  \begin{eqnarray*}
 d   z^1_t &=& d\bfx_t
 \nonumber
 \\
  d z^2_t &= & z^1_{t} \otimes d \bfx_t
  \nonumber
   \\
 & \cdots&
  \\
  d  z^{N}_t &= & z^{N-1}_{t}  \otimes  d\bfx_t,
\nonumber
 \end{eqnarray*}
 with the initial value $ (\bfz^1_0,\dots, \bfz^N_0) = 0$.
For convenience, we denote the equation system by
 \begin{eqnarray}\label{e4.12}
 d z_t = V(z_t)d\bfx_t
,~~ z_0=0,
 \end{eqnarray}
 where $V=(V_{1},\dots, V_{m})$ and $V_{j}(z) = A^{j}z+b^{j}$.
 It is easy to see that $A^{j}  $ and $b^{j}$ are $(\sum_{i=1}^N m^i ) \times (\sum_{i=1}^N m^i) $ and $  (\sum_{i=1}^N m^i )  \times 1$ matrices, respectively, whose entries    take values $0$ and $1$.
  According to Theorem \ref{thm 3.3},  the above equation system has a unique solution. In fact, it  can be verified directly that
 $$
 z^i_t =    \int_0^t \int_0^{t_i} \cdots \int_0^{t_{2}}    d \bfx_{t_1} \otimes \cdots  \otimes  d \bfx_{t_{i-1}}  \otimes d \bfx_{t_i}  ,~~ i=1,\dots, N.
 $$

  According to the definition of   solution of   RDE, for any sequence $(x^n)_n$ in $C^{1\tvr} ([0,T]; \RR^m)$ satisfying \eref{eqn3.1i},   there exists a subsequence of $(x^n)_n$ (which we still denote by $(x^n)_n $) such that
 $S_N (x^n)$ converges to the solution $(1, z^1, \dots, z^N)$ of \eref{e4.12} uniformly. On the other hand, according to   Proposition \ref{prop 3.2}, we can choose the sequence $(x^n)_n $ such that $S_N(x^n)$ converges to $S_N(\bfx)$ uniformly. Therefore, we must have $S_N(\bfx) = (1,z^1, \dots, z^N)$. This completes the proof. \end{proof}

 \smallskip

 Following is our main result in this subsection, which can be shown by  approximation and  with the help of  Proposition \ref{prop 4} and   Lemma \ref{lem3.5}.
 For   $\al \in \Gamma $ such that $|\al|=r$, we denote
 \begin{eqnarray*} 
  x^{\al}_{s,t}&:=&  \pi_{r}(S_{r} (\bfx) )^{\al}_{s,t}  =\int_s^t \int_s^{t_i} \cdots \int_s^{t_{2}}    d \bfx_{t_1}^{ \al_{1}} \cdots     d \bfx_{t_{i-1}}^{\al_{r-1}}    d \bfx_{t_i}^{\al_{r}},
\end{eqnarray*}
  where the second equality holds because of Lemma \ref{lem3.5}.

 \begin{prop}
 We take $\bfx \in C^{p \tvr } ( [0,T]; G^{\lfloor p \rfloor} (\RR^m) ) $.
 Let $\gamma^1$, $\dots$, $\gamma^p $ be multi-indices in $\Gamma$. We  denote  $r= |\gamma^1| +\dots+|\gamma^p|$  and $\vec{\tau} = (\tau_1,\dots, \tau_p)$ such that $\tau_i = |\gamma^1| +\dots+|\gamma^i| $, $i=1,\dots, p$. Denote $\gamma = (\gamma^1, \dots, \gamma^p) \in \Gamma$.  Then
$$
  \int^{t }_{ s }   \int^{t_p }_{s }  \cdots   \int^{t_{2 } }_{ s } d \bfx^{\gamma^1}_{s,t_1}  \cdots d\bfx^{\gamma^{p-1}}_{s,t_{p-1}}  d\bfx^{\gamma^p }_{s,t_p}
          =
           \sum_{ \rho \in \Xi_r ( \vec{\tau}\, )
 }  x^{\gamma \circ \rho^{-1} }_{s, t }\,.
           $$

\end{prop}

 \section{Appendix}

 \subsection{Estimates of some multiple   integrals}
 In this subsection we provide some estimates  on multiple Riemann-Stieltjes integrals
needed in this paper.   We also refer to  \cite{hustochastics} for more studies.

 We take  $r \in \mathbb{N}$ and $\beta_j \in (\frac12, 1 ] $, $j=1,\dots, r$, and $s,s' \in [0,T]$. Let  $g^j  $ be a H\"older continuous function of order $\beta_j  $ on $[s, s' ]$ for $j=1,\dots, m$.
  In this subsection, we consider  the following multiple  integral,
\begin{eqnarray}\label{a.1i}
   g^{\al}_{s,t}  : =
\int_s^t   \int_{s }^{s_1}       \cdots  \int_{ s  }^{s_{r-1}}
 dg^{ {1 }}_{s_{r }} \cdots dg^{ r-1}_{s_2} dg^{  r }_{s_1} \, ,
\end{eqnarray}
where $\al=(1,\dots,r)$.

Recall that for $f \in C^{\beta'}$ and $ h \in C^{\beta}$ such that $\beta+\beta'>1$, we have the following estimate   (see, for instance, \cite{HLN}):
\begin{eqnarray}\label{e7.2i}
\left| \int_{s}^{t} fdh \right| &\leq& K\left[ \|f\|_{s,t,\infty}+\|f\|_{s,t,\beta'} (t-s)^{\beta'} \right] \|h\|_{\beta}(t-s)^{\beta}.
\end{eqnarray}
The following lemma provides an estimate for \eref{a.1i}, which is obtained applying repeatedly  (\ref{e7.2i}).

\begin{lemma}\label{lem 8.1i}
 There exists a constant $K>0$ such that for   $t,t' \in [s,s']$,    we have
\begin{equation}\label{eqn10.1i}
\left|    g^{\al}_{s,t'} -g^{\al}_{s,t}   \right| \leq
      K
  \left(
    \prod_{j =1}^r    \| g^j  \|_{ s,s',  \beta_j  }
    \right) (s-s')^{   \sum_{j =1   }^{r-1} \beta_j   }
       (t'-t)^{\beta_{r } } .
   \end{equation}
\end{lemma}
\noindent \textit{Proof:} \quad We prove the lemma by induction. The inequality is clear when $r = 1$.
Suppose the lemma is true for   $ 1,\dots, r-1$.
   In the case $\beta_{r } = 1 $, we  have
\begin{eqnarray}\label{e 7.1}
 \left|  g^{\al}_{s,t'}-g^{\al}_{s,t}     \right|
 =
 \left| \int_t^{t'}   g^{\al- }_{s,u}   dg^{r }_{u} \right|
 \leq
 \|g^{ r  }\|_{ s,s',\beta_{r } } \sup_{  [s, s' ] }  \left|   g^{\al- }_{s,\cdot}    \right| (t'-t)   .
\end{eqnarray}
 By induction assumption we have
 \begin{eqnarray*}
 \left|  g^{\al- }_{s,t}     \right|  \leq
      K
  \left(
    \prod_{j =1 }^{r-1}    \|g^j \|_{s,s' ,   \beta_j   }
    \right) (s'-s)^{ \sum_{j =1 }^{ r -1} \beta_j  }, ~~s_1 \in [s,s' ] \, .
\end{eqnarray*}
  Substituting the above inequality into \eref{e 7.1} we obtain the   estimate \eref{eqn10.1i}.

In the case $\beta_{r } \in (\frac12, 1) $, it follows from    inequality \eref{e7.2i} that
\begin{eqnarray}\label{eqn 6.1}
 \left|  g^{\al}_{s,t'}-g^{\al}_{s,t}       \right|
 &=&
 \left| \int_t^{t'}   g^{\al-}_{s,u}   d g^{r }_{u }  \right|
 \nonumber
 \\
 &\leq& K
  \left( \sup_{  [s,s']} |  g^{\al- }_{s,\cdot}     |  + \| g^{\al- }_{s,\cdot}   \|_{s, s',\beta_{r-1 }  } ( s' -s)^{  \beta_{r-1 }} \right) \|g^{r } \|_{ s, s', \beta_{r } } ( t' - t )^{ \beta_{r } }
 .
 \end{eqnarray}
By induction assumption we have
 \begin{eqnarray*}
 \left\|  g^{\al- }_{s,\cdot}     \right\|_{s,s', \infty }  \leq
      K
  \left(
    \prod_{j =1 }^{r-1}    \|g^j \|_{s, s',   \beta_j   }
    \right) (s'-s)^{  \sum_{j =1 }^{ r-1 } \beta_j  } ,
\end{eqnarray*}
and
\begin{eqnarray*}
 \left\|  g^{\al- }_{s,\cdot}     \right\|_{s,s',  \beta_{r-1}  }  \leq
      K
  \left(
    \prod_{j =1 }^{r-1}    \|g^j \|_{ s,s', \beta_j  }
    \right) (s'-s)^{ \sum_{j =1 }^{ r-2  } \beta_j  } .
\end{eqnarray*}
Substituting  the above two inequalities into \eref{eqn 6.1} we have
\begin{eqnarray*}
 \left|  g^{\al }_{{s, }t'} -g^{\al }_{s, t}    \right|
 \leq
 K
  \left(
    \prod_{j =1}^{r-1}    \| g^j \|_{  s,s',\beta_j  }
    \right) (s'-s)^{   \sum_{j =1   }^{r-1} \beta_j   }
     \| g^{r} \|_{ s,s',  \beta_{r } } (t'-t)^{\beta_{r } } .
 \end{eqnarray*}
  The proof is now complete.  \end{proof}

\medskip

Let $f_t= (f^1_t,\dots, f^r_t)$, $t\in [0,T]$, be  a function in $ C^\be(\RR^r) $, where $\beta \in (1/2,1 ]$. We denote
$$
 {g}^{j }_{s,  t}(f^{j})=\int_{s}^{t} f^{j}dg^{j},
$$
 and
\begin{eqnarray*}
  {g}^{\al }_{s,  t}  (f^{\al})
  = \int_s^t         \cdots  \int_{ s  }^{s_{2}}
f_{s_1 }^1
 dg^{ {1 }}_{s_{1 }} \cdots   f_{s_r }^r dg^{  r }_{s_r} \, .
 \end{eqnarray*}

\begin{lemma}\label{lem 7.2}
 There exists a constant $K>0$ such that for   $t,t' \in [s,s']$,    we have
\begin{eqnarray*}
\left|    {g}^{\al }_{s,  t'}(f^{\al})   -   {g}^{\al }_{s,  t}(f^{\al})    \right| & \leq &
      K  (s-s')^{   \sum_{j =1   }^{r-1} \beta_j   }
       (t'-t)^{\beta_{r } }
  \left(
    \prod_{j =1}^r    \| g^j \|_{ s,s', \beta_j  } ( \|f^j\|_{ s,s' , \infty } +  \|f^j\|_{s,s' ,\be }  )
    \right) .
   \end{eqnarray*}
\end{lemma}

\begin{proof}  Applying Lemma \ref{lem 8.1i} yields
 \begin{equation} \label{a.5}
\left|    {g}^{\al }_{s,  t'}(f^{\al})   -   {g}^{\al }_{s,  t}(f^{\al})    \right|  \leq
      K
  \left(
    \prod_{j =1}^r    \|  {g}^j_{s, \cdot} (f^{j}) \|_{  s,s', \beta_j  }
    \right) (s-s')^{   \sum_{j =1   }^{r-1} \beta_j   }
       (t'-t)^{\beta_{r } } .
   \end{equation}
  In the case $\beta_j \in (\frac12, 1) $, it follows from inequality \eref{e7.2i} that
\begin{align}\label{a.6}
 \left\|   {g}^j_{s, \cdot}(f^{j}) \right\|_{s,s',\be_j }   \leq  C(\|f^j\|_{s,s',\infty  } + \|f^j \|_{s,s',\beta_j  }  ) \|g^j\|_{s,s',\beta_j }   .
  \end{align}
 In the case $\be_j = 1$,
 we have
 \begin{align}\label{a.7}
 \left\|   {g}^j_{s, \cdot}(f^{j}) \right\|_{s,s',\be_j }   \leq    \|f^j\|_{s,s',\infty  }   \|g^j\|_{s,s',\be_j }   .
  \end{align}
  The lemma then follows by substituting \eref{a.6} and \eref{a.7} into \eref{a.5}.
\end{proof}

\subsection{Estimates of   numerical solutions}
  In this subsection, we derive   upper bound estimates for
  the numerical solutions. We  follow the approaches of   \cite{HLN}.
 We first state an auxiliary result  that provides     estimates  on   integrals whose integrands are step functions.
We define   the seminorm,
\begin{eqnarray*}
  \|x\|_{a,b, \beta, n} = \sup \left\{\frac{|x_u- x_v|}{|v-u|^{\beta}}; \quad u,v \in D \right\}.
\end{eqnarray*}
Recall that $D=\{kT/n: k=0,1,\dots, n\}$ is a partition of $[0,T]$.
When $a=0$ and $b=T$,   we simply write $\|x\|_{\beta,n} = \|x\|_{a,b,\beta,n}$. We will denote $\eta(t)=t_{k}$ for $t\in [t_{k},t_{k+1})$.
\begin{lemma}\label{lem2.1}
Let   $y=\{y_t, t\in [0,T]\}$  be a function with values in  $\mathbb{R}^m$ such    that   $ \| y\|_{\beta ,n} <\infty $,  $n\ge 1$. Take   $V \in C^{1}_{b}(\mathbb{R} ^{m} )$, and  $x \in C^{\beta'} ([0,T])$  such that $\beta+\beta'>1$.
 Then
 for  $s,t \in D$ such that $ s< t$  we have
 \begin{eqnarray*}\label{equation estimate frac int in lemma}
  \left|\int^t_s V(   y_{\eta(r)})dx_r \right|
 & \leq     &
 K \left[ 1
      +        \| y\|_{s, t , \beta, n}      (t-s)^{  \beta    } \right]  \|x \|_{\beta'} (t-s)^{\beta'},
  \end{eqnarray*}
  where the  $K$ is a constant depending on $\beta$, $\beta'$,  $\|V\|_{\infty} $ and $ \| \partial V \|_{\infty} $. 
 \end{lemma}

\begin{proof} See \cite{HLN}.
\end{proof}

\smallskip

Assume that $g=(g^1,\dots, g^m)$ and $g^i \in C^{\be } ([0,T]) $, $i=1,\dots, m$ for  $\beta   >\frac 12$. We fix   $n \in \mathbb{N}$ and   the partition of $[0, T]$ given by $t_i = i \frac Tn$, $ i=0,1,\dots, n $. Consider   the following differential   equation,
\begin{align}\label{a.1}
y_t = y_0   +\sum_{l=1}^N \sum_{\al\in \Gamma_l }  \int_{0}^{t}   \varphi_{ \al } (y_{\eta(s)})   dg_{\eta(s),s}^{\al}   , ~ t \in [0, T],
\end{align}
where $N \in \NN$ is some constant,    and $ \varphi_{\al}$, $  \al \in \cup_{l=1}^N \Gamma_l $ are   functions with values in     $\mathbb{R}^{d \times m}$. Recall that $\Gamma_l$ is the collection of multi-indices of length $l$ with elements in $\{1,\dots, m\}$.   We shall
 derive some estimates for the H\"older seminorm and supremum norm of the solution of this equation.

The constants appearing in the following results  depend on $\beta$,  $T$, $\|  \varphi_{\al}\|_{\infty}$ and $\| \partial\varphi_{\al}\|_{ \infty }$  for $\al \in \Gamma$ of length less or equal to $N$.

\begin{lemma}\label{prop 4.4}
Let $y$ be the solution of   equation \eref{a.1}. Assume that  $\varphi_{\al} \in C^{1}_{b}$ for $  \al \in \cup_{l=1}^N \Gamma_l $.
Then there exists a positive constant $C$ such that
\begin{eqnarray}\label{e 6.8}
 \| y \|_{\beta} & \leq&  C \|g\|_\be \vee \|g\|_\be^{1/\be +N-1},
\end{eqnarray}
and
\begin{eqnarray}\label{a.9}
 \| y \|_{\infty} & \leq&  |y_0|+ C \|g\|_\be \vee \|g\|_\be^{1/\be +N-1}.
\end{eqnarray}
Furthermore, there exists $K_{0}>0$ such that for $s,t \in [0,T]$ and $ \|g\|_{\beta} |t-s|^{\beta} \leq K_{0} $, we have
\begin{eqnarray}\label{e8.13}
\|y\|_{s,t,\beta} &\leq& K \|g\|_{\beta} \vee  \|g\|_{\beta}^{N}.
\end{eqnarray}

\end{lemma}

\begin{proof}   Let $s,t \in [0, T]$ be such that $s <t$.
It follows from \eref{a.1}   that
\begin{eqnarray}\label{a.5i}
  |y_t -y_s| &\leq &
 \sum_{j=1}^m \left|
 \int_{s}^{t}   \varphi_{j } (y_{\eta(u)})    dg_u^{j}
 \right|
 \nonumber
\\
&& +
  \sum_{l=2}^N \sum_{\al\in \Gamma_l } \left| \int_{s}^{t}   \varphi_{\al } (y_{\eta(u)})    dg_{\eta(u),u}^{\al}\right|
 \,.
 \end{eqnarray}

We first derive an estimate for  $\|y\|_{\be,n}$. Assume that $s,t \in D$, that is, $s = \eta(s)$ and $t=\eta(t)$.  Applying Lemma \ref{lem2.1} to  the first term on the right-hand side of the above inequality yields
\begin{align}\label{a.7ii}
\sum_{j=1}^m \left|
 \int_{s}^{t}  \varphi_{j } (y_{\eta(s)})    dg_s^{j}
 \right|
 \leq C(1+ \|y\|_{s,t,\beta, n}(t-s)^{\beta}) \|g\|_{\beta}(t-s)^{\beta}.
\end{align}
On the other hand,   
for  $\al\in \Ga_l$   with $l=1,\dots, N$, we have
 \begin{eqnarray}
 &&  \left| \int_{s}^{t}   \varphi_{\al } (y_{\eta(u)})     dg_{\eta(u),u}^{\al}\right|
 \nonumber
    \leq
    \sum_{k=   ns/T   }^{  nt/T   -1}     \left| \int_{t_k   }^{t_{k+1}  }   \varphi_{\al } (y_{t_{k}})    dg_{\eta(u),u}^{\al}\right|
\nonumber
\\
  &&\qquad \quad \leq
  C
    \sum_{k=ns/T}^{nt/T-1} |g_{t_k, t_{k+1}}^{\al} |
  \leq
 C    {n(t-s)}
   (  {n^{ - \beta} }
  \|g\|_{\beta}   )^l
   =
 C  n^{1- \be l}  { (t-s)}
    \|g\|_{\beta}^l
    \, ,
    \label{a.15i}
 \end{eqnarray}
 where the third inequality follows from Lemma \ref{lem 8.1i}.
So, the second term on the right-hand side of \eref{a.5i}   is bounded by
  \begin{eqnarray}
  \sum_{l=2}^N \sum_{\al\in \Gamma_l } \left| \int_{s}^{t} \varphi_{ \al } (y_{\eta(u)})    dg_{\eta(u),u}^{\al}\right|
  \leq
    C    { (t-s)}
   \sum_{l=2}^N \|g\|_{\beta}^l.
  \label{a.8ii}
 \end{eqnarray}

Substituting   \eref{a.7ii} and \eref{a.8ii} into \eref{a.5i},  we obtain
 \begin{eqnarray*}
 |y_t-y_s| \leq   C\|y\|_{s,t,\be,n}\|g\|_\be (t-s)^{2\beta} +  C    { (t-s)^{\beta}}
 \sum_{l=1}^N
     \|g\|_{\beta}^l\,.
\end{eqnarray*}
Dividing both sides of the above inequality by $(t-s)^{\beta}$, and then taking the   seminorm $\|\cdot\|_{s,t,\beta, n}$ on the left-hand side we obtain
 \begin{eqnarray}\label{a.10}
 \|y\|_{s,t,\beta, n} \leq   C\|y\|_{s,t,\be,n}\|g\|_\be (t-s)^{\beta} +  C
 \sum_{l=1}^N
     \|g\|_{\beta}^l\,.
\end{eqnarray}

If we assume that   $T/n \leq \frac12   (2C \|g\|_{\be})^{-1/\beta}$, then we can find an integer $k_{0}$ such that
 \begin{eqnarray}\label{e.9.17}
 \frac12   (2C \|g\|_{\be})^{-1/\beta} \leq k_{0}T/n \leq   (2C  \|g\|_{\be})^{-1/\beta}.
 \end{eqnarray}
Denote $\Delta:= k_{0}T/n$ and take $u,v $ such that $ u-v =\Delta  $, then from the second inequality in \eref{e.9.17} we have
\begin{eqnarray*}
C\|g\|_{\be}(u-v)^{\be}  &\leq& \frac12.
\end{eqnarray*}
Applying this inequality to \eref{a.10} we obtain
 \begin{eqnarray*}
 \|y\|_{v,u,\beta, n} \leq   \frac12 \|y\|_{v,u,\be,n}  +  C
 \sum_{l=1}^N
     \|g\|_{\beta}^l\,,
\end{eqnarray*}
or
 \begin{eqnarray}\label{e7.18}
 \|y\|_{v,u,\beta, n} \leq    2 C
 \sum_{l=1}^N
     \|g\|_{\beta}^l\,.
\end{eqnarray}
This inequality provides the upper bound for $ \|y\|_{v,u,\beta, n} $ for $u,v \in D: v-u=\Delta$.

For any $s,t \in D$ such that $t-s>\Delta$.
\begin{eqnarray*}
\frac{|y_{t}-y_{s}|}{(t-s)^{\beta}} &\leq& \frac{|y_{s+\Delta}-y_{s}|}{(t-s)^{\beta}} + \frac{|y_{s+2\Delta}-y_{s+\Delta}|}{(t-s)^{\beta}}+\cdots+ \frac{|y_{t}-y_{s+\lfloor \frac{t-s}{\Delta} \rfloor \Delta}  |}{(t-s)^{\beta}}
\\
&\leq & ( \lfloor \frac{t-s}{\Delta} \rfloor +1 ) \sup_{v\in [0,T-\Delta]}\|y\|_{v,v+\Delta,\beta,n} \frac{\Delta^{\beta}}{ (t-s)^{\beta} }\,.
\end{eqnarray*}
Taking the supremum over $s,t \in D$ on  both sides of the above inequality  and taking into account \eref{e7.18}, we obtain
 \begin{eqnarray*}
 \|y\|_{ \beta, n}   \leq C  (\frac T \Delta + 1 )^{1-\beta} (   2 C
 \sum_{l=1}^N
     \|g\|_{\beta}^l) \,.
\end{eqnarray*}
From the first inequality in  \eref{e.9.17} we have
$$ (\frac T  \Delta + 1 )^{1- \beta} \leq C  (   \|g\|_{\be}^{ 1/\beta -1}+1).$$
Therefore,
 \begin{eqnarray}\label{a.19}
 \|y\|_{ \beta, n} &\leq&  C (  \|g\|_\be^{1/\be-1}+1) (
 \sum_{l=1}^N
     \|g\|_{\beta}^l)
     \nonumber
     \\
     &\leq& C \|g\|_\be \vee \|g\|_\be^{1/\be +N-1} \,.
\end{eqnarray}
 If we assume that  $T/n \geq \frac12   (2C \|g\|_{\be})^{-1/\beta}$ or $ n \leq 2T (2C \|g\|_{\be})^{1/\be}   $\,.
Applying \eref{a.15i} to  \eref{a.5i} we obtain
\begin{eqnarray*}
 |y_t -y_s| \leq &
  \sum_{l=1}^N
    C  n^{1- \be l}  { (t-s)}
    \|g\|_{\beta}^l
     \,.
\end{eqnarray*}
Dividing both sides of the above inequality by $(t-s)^{\be}$, and then taking the supremum over all $s,t \in D$, we obtain
  \begin{eqnarray*}
 \|y \|_{\be,n} \leq &
   \sum_{l=1}^N
    C  n^{1- \be l}
    \|g\|_{\beta}^l
     \,.
\end{eqnarray*}
Since $ n \leq 2T(2C \|g\|_{\be})^{1/\be}   $, we have
\begin{eqnarray}\label{a.20}
 \|y \|_{\be,n} \leq &
  C \|g\|_{\be}^{1/\beta}
 +
   \sum_{l=2}^N
    C
    \|g\|_{\beta}^l
     \,.
\end{eqnarray}
Combining  \eref{a.19} and \eref{a.20} we obtain the estimate
     \begin{eqnarray}\label{a.21i}
 \|y\|_{ \beta, n}   \leq C \|g\|_\be \vee \|g\|_\be^{1/\be +N-1} \,.
\end{eqnarray}

It follows from inequality \eref{a.5i} that for $s,t \in [t_{k},t_{k+1}]$, $k=0,1,\dots, n-1$,
\begin{eqnarray}\label{a.21}
\|y\|_{s,t,\be } \leq \sum_{l=1}^{N} \|g\|_{\beta} .
\end{eqnarray}
For any $s,t \in [0,T]$, we have
 \begin{eqnarray*}
\frac{|y_t - y_s|}{|t-s|^\be} \leq \frac{|y_{\eta(s)+\frac Tn} - y_s|}{|\eta(s)+\frac Tn - s |^\be} +
\frac{| y_{ \eta(t) } - y_{\eta(s)-\frac Tn}  |}{| \eta(t) - \eta(s)+\frac Tn |^\be} +
\frac{| y_t - y_{ \eta(t) }    |}{| t- \eta(t)   |^\be}.
\end{eqnarray*}
We apply  \eref{a.21} to the first and third term on  the
  right-hand side of the above inequality and apply \eref{a.21i} to the second term    to obtain
 \begin{eqnarray*}
\frac{|y_t - y_s|}{|t-s|^\be} \leq  C \|g\|_\be \vee \|g\|_\be^{1/\be +N-1}
.
\end{eqnarray*}
The estimate \eref{e 6.8} then follows by taking the supremum  over $s,t \in [0,T]$ on the above left-hand side.

The estimate of $\|y\|_{\infty}$ follows immediately from   \eref{e 6.8}. Indeed, by the definition of $\| \cdot \|_{\be }$ we have
 \begin{eqnarray*}
|y_t| \leq |y_0|+   T^\be ( C \|g\|_\be \vee \|g\|_\be^{1/\be +N-1} ).
\end{eqnarray*}
Taking the supremum of $|y_t|$ over $t\in [0,T]$ we obtain \eref{a.9}.

Finally, it is easy to derive   inequality \eref{e8.13}  from \eref{a.10}.
\end{proof}

\end{document}